\documentclass{amsart}
\usepackage[T1]{fontenc}
\usepackage{amsmath,a4wide}
\usepackage{amssymb}
\usepackage{amscd}
\usepackage{epsfig}
\usepackage[T1]{fontenc}
\usepackage{epsfig}
\usepackage{amsmath}
\usepackage{graphicx}
\usepackage{subfigure}
\usepackage{mathrsfs}
\usepackage{epstopdf}
\usepackage{color}
\usepackage{booktabs}
\usepackage{ifpdf}
\usepackage{cite}
\usepackage{amsthm}
\usepackage{amsmath}
\definecolor{myblue}{rgb}{0.2,0,0.9}
\definecolor{blue-violet}{rgb}{0.54, 0.17, 0.89}
\usepackage{enumitem}
\usepackage{algorithm,algorithmic}

\usepackage{pgfplots}
\usepgfplotslibrary{groupplots}
\pgfplotsset{compat=1.12}

\usepackage{tikz}
\usetikzlibrary{arrows,intersections}
\usepackage{boldline}
\usepackage{multirow}
\usepackage[margin=1.4in]{geometry}

\usepackage{varioref}
\usepackage{xr-hyper}
\usepackage{hyperref}
\hypersetup{
	colorlinks,linkcolor=blue,citecolor=blue,filecolor=red,urlcolor=blue}

\definecolor{myblue}{rgb}{0.2,0,0.9}
\definecolor{blue_violet}{rgb}{0.54, 0.17, 0.89}
\definecolor{darkgreen}{rgb}{0,0.35,0}

\makeatletter
\DeclareRobustCommand*\cal{\@fontswitch\relax\mathcal}
\makeatother

\newtheorem{thm}{Theorem}[section]
\newtheorem{pro}[thm]{Proposition}

\newtheorem{cor}[thm]{Corollary}
\newtheorem{lem}[thm]{Lemma}
\numberwithin{equation}{section}

\theoremstyle{definition}
\newtheorem{rem}[thm]{Remark}

\newtheorem{as}[thm]{Assumption}

\usepackage{xparse}
\makeatletter
\RenewDocumentCommand{\title}{om}{%
	\IfNoValueTF{#1}
	{\gdef\shorttitle{}}
	{\gdef\shorttitle{#1}}%
	\gdef\@title{#2}%
}
\makeatother

\title{Sensitivity of robust optimization problems\\ under drift and volatility uncertainty}
\author{Daniel Bartl}
\address{Faculty of Mathematics, University of Vienna}
\email{daniel.bartl@univie.ac.at}

\author{Ariel Neufeld}
\address{Division of Mathematical Sciences, Nanyang Technological University}
\email{ariel.neufeld@ntu.edu.sg}

\author{Kyunghyun Park}
\address{Division of Mathematical Sciences, Nanyang Technological University}
\email{kyunghyun.park@ntu.edu.sg}

\thanks{\textit{Funding:}
D.~Bartl is grateful for financial support through the Austrian Science Fund (grant DOI: 10.55776/ESP31 and 10.55776/P34743). For open access purposes, the author has applied a CC BY public copyright license to any author accepted manuscript version arising from this submission.
A.~Neufeld gratefully acknowledges financial support by the Nanyang Assistant Professorship Grant (NAP Grant) {\it Machine Learning based Algorithms in Finance and Insurance}.  
K.~Park acknowledges support of the Presidential Postdoctoral Fellowship of Nanyang Technological University.}

\date{\today.}

\begin{document}
\begin{abstract}
	We examine optimization problems in which an investor has the opportunity to trade in $d$ stocks with the goal of maximizing her worst-case cost of cumulative gains and losses. 
	Here, worst-case refers to taking into account all possible drift and volatility processes for the stocks that fall within a $\varepsilon$-neighborhood of predefined fixed baseline processes.
	Although solving the worst-case problem for a fixed  $\varepsilon>0$ is known to be very challenging in general, we show that it can be approximated  as  $\varepsilon\to 0$ by the baseline problem (computed using the baseline processes) in the following sense:
	Firstly, the value of the worst-case problem is equal to the value of the baseline problem plus $\varepsilon$ times a correction term.
	This correction term can be computed explicitly and quantifies how sensitive a given optimization problem is to model uncertainty. 
	Moreover, approximately optimal trading strategies for the worst-case problem can be obtained using optimal strategies from the corresponding baseline problem. 
\medskip

\noindent \textit{Key words.} sensitivity analysis, robust stochastic optimization, model uncertainty, It\^o-semimartingale, backward stochastic differential equations

\medskip

\noindent \textit{MSC classifications.} Primary 90C31, 60G65, 60H05; secondary 91G10
\end{abstract}

\maketitle

\section{Introduction}\label{sec:intro}
Consider a semimartingale  $S=(S_t)_{t\in[0,T]}$ representing the evolution of the value of a  $d$-dimensional (discounted) stock price over time. We assume that a decision maker holds a financial position of the form $l( S_T) $ and aims to hedge against possible losses. 
To that end, she starts with an initial capital $x_0\in\mathbb{R}$ and has the opportunity to buy and sell the stock $S$ without  transaction costs. 
If she invests according to the trading strategy $H$ (i.e., a predictable process), her capital at every time $t\in[0,T]$  equals her initial capital $x_0$ plus the cumulated sums of gains and losses from trading, i.e.\ the stochastic integral $(H\cdot S)_t:=\int_0^t H_s^\top dS_s$. 
The central objective the decision maker faces is to solve the following  optimization problem
\begin{align}
	\label{eq:intro.non.robust}
	\inf_{H} \mathbb{E}\bigg[\int_0^Tg\Big(t,  x_0+(H\cdot S)_t,H_t\Big) dt+ f\Big( x_0 + (H\cdot S)_T , l(S_T) \Big)  \bigg],
\end{align}
where the infimum is taken over a suitable class of trading strategies $H$, and $g\colon \Omega\times [0,T]\times \mathbb{R}\times \mathbb{R}^d \to \mathbb{R}$ and $f\colon  \Omega\times \mathbb{R}^2\to\mathbb{R}$ are the individual {intermediate} and terminal \emph{cost functions} of the decision~maker. 
Notable examples that fall within this framework are \emph{mean-variance hedging} (e.g., \cite{cadenillas1995stochastic,DR1991,czichowsky2012convex,Schweizer1992,ZhouLi2000}) in which case $g(\omega,t,x,h)=A_t(\omega)x^2+xB _t^\top(\omega) h+h^\top C_t(\omega) h$ and $f(\omega,x,y)=(x-\xi(\omega)-y)^2$ for some predictable processes $(A,B,C)$ having values in $(\mathbb{R},\mathbb{R}^d,\mathbb{R}^{d\times d})$ and a random (${\cal F}_T$-measurable) target level\;$\xi$, and \emph{utility maximization} (e.g., \cite{karatzas1991martingale,cvitanic1992convex,kramkov1999asymptotic,cvitanic2001utility,hugonnier2004optimal}) in which case $g(\omega,t,x,h)=0$ and $f(\omega,x,y)= - U(-x-\xi(\omega)-y) $ for a concave and increasing utility function~$U\colon\mathbb{R}\to\mathbb{R}$ and a random liability~$\xi$.

Arguably, one of the most common models for the underlying $S$ is that it follows the dynamics 
\begin{align}
	\label{eq:intro.def.S} 
	dS_t:= dS^{b,\sigma}_t= b_t dt + \sigma_t dW_t, \qquad S^{b,\sigma}_0=s_0\in\mathbb{R}^d,
\end{align}
where $W$ is a Brownian motion and the (possibly random and time-dependent) drift and volatility are given by the predictable processes $b$ and $\sigma$, respectively.
We refer to \cite{karatzas1991martingale,hu2005utility,horst2014forward,DR1991,ZhouLi2000} for a handful of the many articles analyzing problems similar to \eqref{eq:intro.non.robust} in the setting \eqref{eq:intro.def.S}. 

\vspace{0.5em}
A significant challenge arises when implementing this framework in practice: the true values of the parameters  $b$ and $\sigma$ used to define $S$ in \eqref{eq:intro.def.S} are  not perfectly known. 
Instead, they are typically estimated using e.g.\ historical market data and expert insights.  
Even though these estimation techniques strive to provide parameter values that closely approximate their actual counterparts, a margin for potential inaccuracies inherently exists. 
While this issue has always been present, it gained particularly prominent attention following the financial crisis in 2008.  Since then, a substantial body of research in mathematical finance has been dedicated to developing methods that can accommodate potential model misspecification.

\vspace{0.5em}
The most prominent approach can be traced back to the seminal papers \cite{dow1992uncertainty,gilboa1989maxmin,chen2002ambiguity} and is widely recognized as the `worst-case' approach, or Knightian approach to model uncertainty. 
In our current context, this approach has been pioneered by \cite{PengS_gexp1997,CHMP2002,DenisHuPeng2011,Peng07,Peng08} and it consists in fixing an entire set $\mathcal{U}$ of parameters $(b,\sigma)$ that one would  consider reasonable candidates for representing the actual drift and volatility.
Subsequently, the objective is to address the worst-case optimization problem given by
\begin{align}\label{eq:intro.worst.case.general}  
	\inf_{H} \sup_{ (b,\sigma)\in \mathcal{U}  } \mathbb{E}\bigg[\int_0^Tg\Big(t , x_0+(H\cdot S^{b,\sigma})_t,H_t\Big) dt+  f\Big(x_0 + (H\cdot S^{b,\sigma})_T , l(S_T^{b,\sigma}) \Big)  \bigg].
\end{align}
It is evident that \eqref{eq:intro.worst.case.general} and \eqref{eq:intro.non.robust} coincide when the set $\mathcal{U}$ is a singleton; however, in general, there exists a significant degree of latitude in selecting $\mathcal{U}$. One seemingly intuitive approach consists in starting with parameters $(b^o,\sigma^o)$ that are derived from some estimation procedure which one would typically employ in the non-robust problem \eqref{eq:intro.non.robust}, and then defining $\mathcal{U}$ by adding  all small perturbations of these parameters, effectively creating a small neighborhood around $(b^o,\sigma^o)$.

Before examining that specific choice of $\mathcal{U}$ in more details, we note that problems of the form in \eqref{eq:intro.worst.case.general} have conceived a considerable amount of attention in the mathematical finance community. 
In fact, most of the fundamental and often technically demanding mathematical questions therein are understood fairly well by now and genuinely have affirmative answers.
We~refer e.g.\ to\cite{neufeld2018robust,yang2019constrained,denis2013optimal,biagini2019robust,tevzadze2013robust} for the existence of an optimal strategy of \eqref{eq:intro.worst.case.general}; to  \cite{lin2021optimal,pham2022portfolio,biagini2017robust,Fouque_et_al16} for an analysis of the dynamic programming principle; to \cite{BKN21,schied2007optimal,hernandez2007control,nutz2015robust,dolinsky2014martingale,PW22} for the relation of the present hedging problem to a dual pricing problem; to \cite{PengS_gexp1997,DenisHuPeng2011,Peng07,NeufeldNutz2014,neufeld2017nonlinear,nutz2013random} for modeling uncertainty in semimartingale characteristics; to \cite{bayraktar2013multidimensional,bayraktar2011optimal_a,bayraktar2011optimal_b,PW23,park2023robust} for extensions of \eqref{eq:intro.worst.case.general} with optimal stopping time; and to \cite{cheridito2007second,soner2012wellposedness,soner2013dual,matoussi2015robust} for relations to second-order backward stochastic differential equations.

That being said, it is crucial to highlight that the mere existence of an optimal strategy, while theoretically intriguing, may not be particularly practical. Indeed, the primary interest of a decision maker often lies rather in finding methods to compute it. However, this is precisely where the worst-case approach encounters a notable \emph{limitation}: The computation of both the value and the optimal strategy in \eqref{eq:intro.worst.case.general} is notoriously difficult and except in a few exceptional cases (see \cite{lin2021optimal,PW22,Fouque_et_al16,neufeld2018robust,biagini2019robust}), explicit solutions are unknown.
The goal of this article is to overcome this limitation.

\vspace{0.5em}
Specifically, our main result states that in the setting where $\mathcal{U}$ is a small neighborhood of fixed parameters $(b^o,\sigma^o)$, the following holds:
\begin{enumerate}[leftmargin=2em]
	\item[] Each optimization problem \eqref{eq:intro.non.robust} has an associated number  that quantifies how \emph{sensitive} it is towards model uncertainty, that is, by how much larger \eqref{eq:intro.worst.case.general} is than \eqref{eq:intro.non.robust}. 
	Moreover, that number can be computed explicitly.
\end{enumerate}
In particular, the value of the robust optimization problem \eqref{eq:intro.worst.case.general} can be approximated accurately by \eqref{eq:intro.non.robust}.
More importantly, we rigorously establish what is typically called an `envelope theorem' in the economics literature: that an almost optimal strategy for the robust optimization problem \eqref{eq:intro.worst.case.general} can be derived using an optimal strategy for the classical one \eqref{eq:intro.non.robust}.

\vspace{0.5em}
We proceed to  describe our results more rigorously. For $p\geq 1$, denote by $\mathbb{L}^p$  and $\mathbb{H}^p$ the set of all predictable processes with values in $\mathbb{R}^d$  and $\mathbb{R}^{d\times d}$, respectively. 
We endow $\mathbb{L}^p$ and $\mathbb{H}^p$ with the norms $\|\cdot \|_{\mathbb{L}^p}$ and $\|\cdot \|_{\mathbb{H}^p}$, respectively, defined for every $b\in \mathbb{L}^p$ and $\sigma \in \mathbb{H}^p$ by
\[
	\|b\|_{\mathbb{L}^p}^p:=  \mathbb{E}\Bigg[ \int_0^T |b_t|^p \,dt \Bigg], \qquad \|\sigma\|_{\mathbb{H}^p}^p:=  \mathbb{E}\Bigg[ \bigg(\int_0^T \|\sigma_t\|_{\operatorname{F}}^2 \,dt\bigg)^{\frac{p}{2}} \Bigg],
\]
where $|\cdot|$ is the Euclidean norm and $\| \cdot \|_{\operatorname{F}}$ is the Frobenius / Hilbert-Schmidt norm.

Let $p>3$ and set $\gamma,\eta\geq 0$. Fix  baseline parameters $b^o\in \mathbb{L}^p$ and $\sigma^o\in\mathbb{H}^p$, e.g.\ the `estimators', and for $\varepsilon\geq 0$ denote  by 
\begin{align}\label{eq:amb_set}
	{\cal B}^\varepsilon:=\big\{ (b,\sigma)  \in \mathbb{L}^p \times\mathbb{H}^p : 
	\|b-b^o\|_{\mathbb{L}^p} \leq \gamma\varepsilon\;
	\text{ and }\;
	\|\sigma-\sigma^o\|_{\mathbb{H}^p} \leq \eta \varepsilon \big\}
\end{align}
the set of all parameters that fall in the $\varepsilon$-neighborhood of the  baseline parameters, weighted by the `aversion parameters' $\gamma,\eta$. 
Arguably, the cases of particular interest are when $\gamma, \eta \in \{0,1\}$:
for $\gamma=1$ and $\eta=0$, the set ${\cal B}^\varepsilon$ corresponds to drift uncertainty (sometimes referred to as the $g$-expectation framework introduced by \cite{PengS_gexp1997,CHMP2002}); 
for $\gamma=0$ and $\eta=1$, it corresponds to volatility uncertainty (also called the $G$-expectation framework proposed by \cite{DenisHuPeng2011,Peng07,Peng08,nutz2013random,SonerTouziZhang2011}); and for $\gamma=\eta=1$ it captures uncertainty in both the drift and volatility (e.g., \cite{NeufeldNutz2014,neufeld2017nonlinear,neufeld2018robust}).

With this notation set in place, consider the robust optimization problem \eqref{eq:intro.worst.case.general} with the choice $\mathcal{U}={\cal B}^\varepsilon$, that is,
\begin{align}\label{eq:intro.worst.case.ball}
	V(\varepsilon) := \inf_{H} \sup_{(b,\sigma)\in {\cal B}^\varepsilon } \mathbb{E}\bigg[\int_0^Tg\Big(t , x_0+(H\cdot S^{b,\sigma})_t,H_t\Big) dt+ f\Big(x_0 + (H\cdot S^{b,\sigma})_T , l(S_T^{b,\sigma}) \Big)  \bigg],
\end{align}
where the infimum is taken over all predictable trading strategies $H$ taking values in a predictably measurable subset of $\mathbb{R}^d$ (i.e., a set-valued map which is predictable; see Assumption~\ref{dfn:control}). 
The latter is a technical condition that is not too restrictive; see also Lemma~\ref{lem:set_value_ex}.
In particular, $V(0)$ is the non-robust optimization problem computed using the baseline parameters $b^o$ and $\sigma^o$.

\vspace{0.5em}
In Theorem \ref{thm:main} we show that if $(x,h)\to g(\omega,t,x,h)$ and $(x,y)\to f(\omega,x,y)$ are convex and twice continuously differentiable for every $(\omega,t)\in\Omega\times [0,T]$, $l$ is twice continuously differentiable, $\sigma^o$ is non-degenerate, and $g,f,l,b^o,\sigma^o$ satisfy modest growth assumptions, then the following hold. As $\varepsilon\downarrow 0$, 
\begin{align}
	\label{eq:intro.sensitivity}
	\begin{split}
	V(\varepsilon) 
	= V(0) + \varepsilon\Big(  \gamma \big\lVert Y^*H^*+ \mathcal{Y}^* \big\rVert_{\mathbb{L}^q} 
		+ \eta \big\lVert Z^* (H^*)^\top + \mathcal{Z}^* \big\rVert_{\mathbb{H}^q}\Big) + O(\varepsilon^2),
	\end{split}
\end{align}
where  $O$ denotes the Landau symbol, $H^\ast$ is the unique optimizer for $V(0)$, and $q=\frac{p}{p-1}$ is the conjugate H\"older exponent to $p$.
Moreover, the processes $Y^\ast,\mathcal{Y}^\ast,Z^\ast,\mathcal{Z}^*$ appearing in \eqref{eq:intro.sensitivity} take the following form:
set $S^o=S^{b^o,\sigma^o}$ to be stock following the baseline parameters and for simplicity here in the introduction let $d=1$.
Then 
\begin{align*}
Y_t^*&=\mathbb{E}\bigg[\partial_xf\Big(x_0+ (H^*\cdot S^{o})_T , l(S_T^{o})\Big)+\int_t^T\partial_xg\Big(s,x_0+ (H^*\cdot S^{o})_s,H_s^*\Big) ds\,\bigg|\,{\cal F}_t \bigg], \\
{\cal Y}^\ast_t
&=\mathbb{E}\bigg[\partial_y f\Big(x_0+ (H^*\cdot S^{o})_T, l(S_T^{o})\Big) l'(S_T^{o}) \,\bigg|\, {\cal F}_t\bigg] ,\\
Z_t^*&= \frac{d}{dt}\langle Y^*, W\rangle_t,
 \quad\text{ and } \quad
 \mathcal{Z}_t^\ast=\frac{d}{dt} \langle{\cal Y}^*, W\rangle_t,
\end{align*}
where $\langle M, N\rangle$ denotes their predictable covariation between two martingales $M$ and $N$. We refer to Section \ref{sec:exmp_mv_hedging} for a more explicit expression of $V'(0)$ in the  special case of robust mean-variance hedging.

Next, imposing the same assumption made before on $g,f,l,b^o,\sigma^o$, recall that $H^\ast$ is the unique optimizer for the non-robust problem $V(0)$, and set 
\[ 
V^\ast(\varepsilon) =\sup_{(b,\sigma)\in {\cal B}^\varepsilon } \mathbb{E}\Bigg[\int_0^Tg\Big(t , x_0+(H^*\cdot S^{b,\sigma})_t,H^*_t\Big) dt+ f\Big( x_0 + (H^*\cdot S^{b,\sigma})_T\,,\,l(S_T^{b,\sigma}) \Big)  \Bigg]. 
\]
In other words, $V^\ast(\varepsilon)$ tracks how well the optimal strategy computed for the baseline parameters $(b^o, \sigma^o)$ performs in the worst case when nature is allowed to maliciously select parameter variations $(b, \sigma) \in {\cal B}^\varepsilon$.
Hence one might actually argue that the quantity $V^\ast(\varepsilon)$ is practically more relevant than $V(\varepsilon)$.
As an application of our main result, Theorem \ref{thm:main2} shows that~as~$\varepsilon\downarrow 0$,
\begin{align*}
	V^\ast(\varepsilon) 
	&=  V(\varepsilon)   +  O(\varepsilon^2).
\end{align*}
In particular $(V^\ast)'(0)=V'(0)$, and up to a second order correction term, the unique optimal strategy $H^\ast$ computed for $(b^o,\sigma^o)$ will perform just as good as the strategy computed for the robust optimization problem $V(\varepsilon)$. 
In other words, there is no genuine need for the challenging computation of a strategy that is optimal for the robust strategy -- a  phenomenon that is commonly referred to as the envelope theorem in economic literature.

\vspace{0.5em}
\emph{Related literature.}
The results that are closest to the present ones were obtained in the context of \emph{Wasserstein} distributionally robust optimization, i.e.\ when ${\cal B}^\varepsilon$ consists in all probability laws (for $S^{b,\sigma}$) which are close in Wasserstein distance to the distribution of $S^{o}$. 
Starting with \cite{blanchet2019quantifying,gao2023distributionally,mohajerin2018data,pflug2007ambiguity}, this branch of research has gained a significant amount of attention, { and in that setting the following is known.
In a one-period framework, a sensitivity analysis similar to ours was obtained in  \cite{bartl2021sensitivity,nendel2022parametric,gao2022wasserstein,obloj2021distributionally}. 
On the other hand, in a framework with more than one period the classical Wasserstein distance is not a suitable distance because it neglects the temporal structure of stochastic processes (see e.g.\ \cite{backhoff2020adapted,pflug2010version,pflug2012distance}) and an adapted variant takes its role.
In a discrete-time multi-period framework,  a sensitivity analysis similar to ours (but w.r.t\ the adapted Wasserstein distance) was established in  \cite{bartl2023sensitivity,yifan}.
It is important to highlight that although certain proof techniques in the present article bear similarities to those in the above settings, there are significant distinctions. 
Most notably, our choice of ${\cal B}^\varepsilon$ is far more rigid than its Wasserstein counterpart. 
Consequently, we cannot directly apply the duality between $L^p$ and $L^q$ spaces to deduce the value of $V'(0)$ (e.g.\ as in \cite{bartl2021sensitivity,bartl2023sensitivity}); instead, a substantial portion of our efforts revolves around establishing a suitable representation involving BSDE's which allow to express $V'(0)$ as the supremum over certain linear functions.
In a continuous-time framework, there are multiple natural adapted variants of the Wasserstein distance. 
A continuous dependence of various optimization problems (w.r.t.\ one of these variants) has been established in \cite{backhoff2020adapted}, and more recently (after the completion of this work) \cite{jiang2024sensitivity,sauldubois2024first} analyze certain sensitivities.
While the findings in \cite{bartl2021limits,bartl2023sensitivity,fuhrmann2023wasserstein,jiang2024sensitivity,sauldubois2024first} indicate a potential  connection between the sensitivity computed for uncertainty w.r.t.\ an adapted Wasserstein distance and the present setting with (only) drift uncertainty, exploring the details of these relations first requires a comprehensive understanding of the sensitivity of adapted Wasserstein distance in continuous time. 
This is beyond the scope of the present paper, and we leave it for future work.


Moving away from the above Wasserstein-setting to a parametric setting, \cite{herrmann2017model,herrmann2017hedging} analyze the sensitivity of utility maximization problems to volatility-uncertainty in a continuous-time framework, but in a  somewhat different setting to ours.
}
Roughly put, instead of taking the supremum only over those processes $\sigma$  that satisfy $\|\sigma-\sigma^o\|_{\mathbb{H}^p}\leq\varepsilon$, they allow for all $\sigma$'s and penalize by $\frac{1}{\varepsilon}$ multiplied by `far' $\sigma$ is from $\sigma^o$.
The notion of being far takes a form similar to a KL-divergence but, crucially, with the specific choice of the utility function (i.e.\ $f$ in the present setting) appearing in its definition. In particular, the sensitivity to model uncertainty (corresponding to $V'(0)$ in the present setting) obtained in \cite{herrmann2017model,herrmann2017hedging} \emph{does not} depend on the choice of the utility function nor on the optimal trading strategy $H^\ast$.


\vspace{0.25em}
Finally, for completeness, let us also mention that analyzing the  continuity and sensitivity of \emph{non-robust}  optimization problems w.r.t.\ directional perturbations in the underlying  model is a classical topic in optimization, we refer e.g.\ to \cite{bonnans2013perturbation,calafiore2007ambiguous,lam2016robust,vogel1988stability} and the references therein for some general background.
In the present framework, this would correspond to \emph{fixing} two processes $\tilde{b}$ and $\tilde{\sigma}$ and analyzing the behavior of 
\[ \varepsilon \mapsto \inf_H \mathbb{E}\left[ f\left(x_0 + (H\cdot S^{b^o+\varepsilon \tilde{b}, \sigma^o + \varepsilon \tilde{\sigma}} )_T ,\, l(S_T^{b^o+\varepsilon \tilde{b}, \sigma^o + \varepsilon \tilde{\sigma}}) \right) \right] \]
as $\varepsilon \downarrow 0$. 
Under different sets of assumptions (though always assuming that $l\equiv 0$) the continuity and sensitivity of this problem is studied in  \cite{bayraktar2013stability,larsen2007stability,mostovyi2019sensitivity,mostovyi2024quadratic}.

\section{Main results}\label{sec:model}

\subsection{Notation and preliminaries}\label{sec:preliminary} 

Fix $d\in \mathbb{N}$.
We endow $\mathbb{R}^d$ and $\mathbb{R}^{d\times d}$ with the Euclidean inner product $\langle \cdot,\cdot\rangle$ and the Frobenius inner product $\langle \cdot,\cdot \rangle_{\operatorname{F}}$, respectively. Let $\mathbb{S}^d$ be the set of all symmetric matrices, let $\mathbb{S}^d_+\subset \mathbb{S}^d$ be the subset of all positive semi-definite $d\times d$ matrices and let $\mathbb{S}^d_{++}\subset \mathbb{S}^d_+$ be the subset of all strictly positive definite matrices.
We denote by $I_d$ the identity matrix on $\mathbb{R}^{d\times d}$ and for two $d\times d$ matrices $A$ and $B$, we write $A\leq B$ if $B-A\in \mathbb{S}_+^d$.

Next, let us introduce the following function spaces:
\begin{itemize}[leftmargin=2em]
	\item [$\cdot$] $C([0,T];\mathbb{R}^d)$ (resp.\;$C([0,T];\mathbb{R}^{d\times d})$) is the set of all $\mathbb{R}^d$-valued (resp. $\mathbb{R}^{d\times d}$-valued), continuous functions on $[0,T]$;
	\item [$\cdot$] $C^k(\mathbb{R}^d)$ is the set of all real-valued, $k$-times continuously differentiable functions on $\mathbb{R}^d$. 
	We denote by $\nabla_s g\colon \mathbb{R}^d\ni (s_1,\dots,s_d)^\top=:s\mapsto \nabla_sg(s):=(\partial_{s_1}g(s),\dots,\partial_{s_d}g(s))^\top\in\mathbb{R}^d$ the gradient of $g$ and by $D_s^2 g:\mathbb{R}^d\ni s \mapsto D^2_sg(s):=[\partial_{s_is_j}g(s)]_{1\leq i,j \leq d}\in \mathbb{S}^d$ its Hessian.
\end{itemize}

\vspace{0.5em}

Let $(\Omega,{\cal F}, \mathbb{F}:=({\cal F}_t)_{t\in [0,T]},\mathbb{P})$ be a filtered probability space that satisfies the usual conditions of right-continuity and completeness, where $T$ is a fixed finite time horizon. 
We assume that ${\cal F}_0$ is trivial. 
Fix a $d$-dimensional Brownian motion $W=(W_t)_{t\in[0,T]}$ on that filtered probability space.
For any probability measure $\mathbb{Q}$ on $(\Omega,{\cal F})$, we write  $\mathbb{E}^{\mathbb{Q}}[\cdot]$ for the expectation under $\mathbb{Q}$ and set $\mathbb{E}[\cdot]:=\mathbb{E}^{\mathbb{P}}[\cdot]$.
For sufficiently  integrable $\mathbb{R}^d$-valued processes ${ Y}$ and $Z$, let 
\[\langle { Y},Z \rangle_{\mathbb{Q}\otimes dt} := \mathbb{E}^{\mathbb{Q}}\Bigg[\int_0^T \langle { Y}_t,Z_t \rangle dt\Bigg].\]
In a similar manner, $\langle {Y}, Z \rangle_{\mathbb{Q}\otimes dt,\operatorname{F}} := \mathbb{E}^{\mathbb{Q}}[\int_0^T \langle { Y}_t,Z_t \rangle_{\operatorname{F}} dt]$ for $\mathbb{R}^{d\times d}$-valued processes.
We denote by $\xrightarrow[]{\mathbb{Q}}$ the convergence in $\mathbb{Q}$-measure and adopt a  similar notation  for convergence in $\mathbb{Q}\otimes dt$-measure. Finally, for $p\geq 1$, denote by $C_{\operatorname{BDG},p}\geq 1$ the constant appearing in the (upper) Burkholder-Davis-Gundy (BDG) inequality with exponent $p$ (see \cite[Theorem 92, Chap.\;VII]{dellacherieprobabilities}).

\vspace{0.5em}
For any real-valued semimartingales $M=(M_t)_{t\in[0,T]}$ and $N=(N)_{t\in[0,T]}$ with locally integrable quadratic variation, we denote by $([M,N]_t)_{t\in[0,T]}$ the quadratic co-variation and by $(\langle M,N\rangle_t)_{t\in[0,T]}$ the $\mathbb{F}$-predictable quadratic co-variation (i.e., the compensator of $([M,N]_t)_{t\in[0,T]}$). 
 Note that by \cite[Theorem 4.52, p.\;55]{JacodShirayev2013}, for any  $t\in[0,T]$, we have that $[M,M]_t:= \langle M^c,M^c\rangle_t +\sum_{0\leq s \leq t} \lvert \Delta M_s\rvert^2$ where $M^c$ is the continuous martingale part of $M$ and $\Delta M_t:=M_{t}-M_{t-}$.
 
Finally, for any $p \geq 1$, consider the following spaces.
\begin{itemize}[leftmargin=2em]
	\item [$\cdot$] $L^p({\cal F}_T;\mathbb{R}^d)$ is the set of all $\mathbb{R}^d$-valued, ${\cal F}_T$-measurable random variables $X$ such that $\lVert X \rVert_{L^p}^p:=\mathbb{E}[\lvert  X\rvert^p ]<\infty$;
	\item [$\cdot$] 
	$\mathscr{S}^{p }(\mathbb{R})$ is the set of all real-valued, $\mathbb{F}$-progressively measurable c\`adl\`ag (i.e., right-continuous with left-limits) processes $S=(S_t)_{t\in[0,T]}$ such that $\lVert S \rVert_{\mathscr{S}^{p }}^p   := \mathbb{E}[\sup_{t\in [0,T]}| S_t |^p  ]<\infty$;
	\item [$\cdot$] $\mathscr{H}^{p }(\mathbb{R}^d)$ is the set of all $\mathbb{R}^d$-valued, $\mathbb{F}$-predictable processes $H=(H_t)_{t\in[0,T]}$ such that $\lVert H \rVert_{\mathscr{H}^{p }}^p   := \mathbb{E}[(\int_0^T \lvert H_t \rvert^2 dt)^{p /2}]<\infty$;
		\item [$\cdot$] $\mathscr{M}^{p }(\mathbb{R})$ is the set of all real-valued, c\`adl\`ag $(\mathbb{F},\mathbb{P})$-martingales $M=(M_t)_{t\in [0,T]}$ such that $\lVert M \rVert_{\mathscr{M}^{p }}^p   := \mathbb{E}[[M,M]_T^{p/2}]<\infty$;
	\item [$\cdot$] $\mathbb{L}^p=\mathbb{L}^p(\mathbb{R}^d)$ and $\mathbb{H}^p=\mathbb{H}^p(\mathbb{R}^{d\times d})$ are defined in the introduction.
\end{itemize}

\subsection{The market model(s)}\label{sec:semimtg}
For every (sufficiently integrable) predictable processes $b$ and $\sigma$ taking values in $\mathbb{R}^d$ and $\mathbb{R}^{d\times d}$ respectively, we define the It\^o $(\mathbb{F},\mathbb{P})$-semimartingale $S^{b,\sigma}$ by
\begin{align}\label{eq:posterior_semi_ito}
	\quad S_t^{b,\sigma} := s_0+ \int_0^tb_u du + \int_0^t \sigma_u dW_u,\;\; t\in[0,T],
\end{align}
where $s_0\in\mathbb{R}^d$ is fixed and does not depend on $b$ and $\sigma$, and we recall that $W=(W_t)_{t\in[0,T]}$ is a fixed $d$-dimensional Brownian motion. Moreover, recall that $b^o $ and $\sigma^o$ are two \emph{fixed} baseline processes, that $\gamma,\eta\geq 0$ are fixed constants, and that
\begin{align}
\label{eq:def.ball}
 \mathcal{B}^\varepsilon=\{ (b,\sigma) \in  \mathbb{L}^p(\mathbb{R}^d)\times \mathbb{H}^p(\mathbb{R}^{d\times  d}) : \| b - b^o \|_{\mathbb{L}^p}  \leq\gamma\varepsilon, \,  \| \sigma - \sigma^o \|_{\mathbb{H}^p}  \leq \eta\varepsilon \}.
 \end{align}
The exact value of $p$ is specified in Assumption \ref{as:posterior_refer}.
For shorthand notation, set $S^o:=S^{b^o,\sigma^o}$.
\begin{as}\label{as:posterior_refer} 
The following conditions hold: 
	\begin{itemize}
		\item [(i)] $b^{o} \in \mathbb{L}^p(\mathbb{R}^d)$ and $\sigma^o \in\mathbb{H}^p(\mathbb{R}^{d\times  d})$ for some $p>3$.
		\item [(ii)] $\sigma^o$ is invertible, i.e., there is $(\sigma^o)^{-1}$ such that $\sigma_t^o(\sigma_t^o)^{-1}=I_d$ $\mathbb{P}\otimes dt$-a.e.. 
		Furthermore, the  process ${\cal D}=({\cal D}_t)_{t\in[0,T]}$ defined by
		\[
		\quad\qquad{\cal D}_t:= {\cal E} \left(-\big((\sigma^o)^{-1}{b}^o\big) \cdot {W}\right)_t=  \exp\left(-\frac{1}{2}\int_0^t \lvert (\sigma^o_u)^{-1}{b}_u^o \rvert^2 du -\int_0^t\big((\sigma^o_u)^{-1}{b}_u^o \big)^\top d{W}_u\right)
		\]
		is a $(\mathbb{F},\mathbb{P})$-martingale satisfying ${\cal D}_T\in L^\beta({\cal F}_T;\mathbb{R})$ for every $\beta \geq 1$.
	\end{itemize}
\end{as}

\begin{rem}
\label{rem:p_char}
	Assumption \ref{as:posterior_refer}\;(i) implies that for every  $\varepsilon\geq 0$, 
	$\sup_{(b,\sigma)\in \mathcal{B}^\varepsilon}\|S^{b,\sigma} \|_{\mathscr{S}^p}<\infty$.
	In addition, Assumption \ref{as:posterior_refer}\;(ii) implies that  the measure $\mathbb{Q}$ defined by
\begin{align}
\label{eq:mtg_measure}
	\Big.\frac{d\mathbb{Q}}{d\mathbb{P}}\Big|_{{\cal F}_T}:={\cal D}_T
\end{align}
	is a probability measure equivalent to $\mathbb{P}$, and that ${W}^{\mathbb{Q}} := {W}+\int_0^\cdot (\sigma_u^o)^{-1}{b}_u^o du$ is a $d$-dimensional Brownian motion under $\mathbb{Q}$ (by Girsanov's theorem).
	These observations will be used later.
\end{rem}

Let us provide sufficient conditions for Assumption \ref{as:posterior_refer} to hold true.
The proofs can be found in Section \ref{sec:remaining.proofs}.

\begin{lem}\label{lem:posterior_refer2} Suppose that $\sigma^o$ is invertible. Then each of the following conditions is sufficient for Assumption \ref{as:posterior_refer} to hold: 
	\begin{itemize}
		\item [(i)] $b^o$ and $\sigma^o$ are uniformly bounded and $(\sigma^o)^\top\sigma^o$ satisfies a uniform ellipticity condition, i.e., there are $C_{{b},{\sigma}}>0$ and  $\underline{c}\in \mathbb{S}^d_{++}$ s.t.\  $\lvert b_t^o \rvert +\lVert \sigma_t^o \rVert_{\operatorname{F}} \leq C_{{b},{\sigma}}$ and  $(\sigma_t^o)^\top\sigma_t^o\geq \underline{c}$ $\mathbb{P}\otimes dt$-a.e.
		\item [(ii)] 
		$b^o$ and $\sigma^o$ are of the following form:
		\begin{itemize}
			\item [$\cdot$] $b_t^o= \widetilde{b}^o(t,S_t^o),$ $\sigma_t^o= \widetilde{\sigma}^o(t,S_t^o)$ $\mathbb{P}\otimes dt$-a.e. \textnormal{(SDE)},
			\item [$\cdot$] $(\sigma^o_t)^{-1} {b}_t^o= \theta(t,{W})$ $\mathbb{P}\otimes dt$-a.e. 
			\textnormal{(Bene\v{s} condition)},
		\end{itemize}
		where
		$\widetilde{b}^o:[0,T]\times \mathbb{R}^d\rightarrow \mathbb{R}^d$ and $\widetilde{\sigma}^o:[0,T] \times \mathbb{R}^d\rightarrow \mathbb{R}^{d\times d}$ are Borel  functions, there is $C_{\widetilde{b},\widetilde{\sigma}}>0$ such that for every $t\in[0,T]$ and $x,y\in \mathbb{R}^d$
		\begin{align*}
			\begin{aligned}
				\lvert \widetilde{b}^o(t,x) - \widetilde{b}^o(t,y) \lvert + \lVert \widetilde{\sigma}^o(t,x)- \widetilde{\sigma}^o(t,y) \rVert_{\operatorname{F}} 
				&\leq C_{\widetilde{b},\widetilde{\sigma}} \lvert x-y \rvert,\\
				\lvert \widetilde{b}^o(t,x) \lvert + \lVert \widetilde{\sigma}^o(t,x) \rVert_{\operatorname{F}} 
				&\leq C_{\widetilde{b},\widetilde{\sigma}} (1+\lvert x\rvert),
			\end{aligned}
		\end{align*}
		$\theta\colon [0,T]\times C([0,T];\mathbb{R}^d) \rightarrow  \mathbb{R}^d$ is progressively measurable\footnote{
		That is, $\theta$ is jointly Borel-measurable and for every $t\in[0,T]$ and $x,y\in C([0,T];\mathbb{R}^d)$ that satisfy $x(s)=y(s)$ for $s\in[0,t]$, it holds that $\theta(t,x)=\theta(t,y)$, see, e.g.,  \cite[5, IV.94-103, pp.~145-152]{dellacherieprobabilitiesA}.
		}, and there is $C_\theta>0$ s.t.\ for all $t\in[0,T]$ and $x=(x_t)_{t\in[0,T]}\in C([0,T];\mathbb{R}^d)$, $\lvert \theta(t,x)\rvert \leq C_{\theta} (1 + \sup_{s\in [0,t]} \lvert x_s \rvert)$.
	\end{itemize}
\end{lem}

\begin{rem}
	In Lemma \ref{lem:posterior_refer2}, the uniform boundedness condition of $(b^o,\sigma^o)$ and the SDEs with Lipschitz and linear growth conditions are commonly adopted in continuous-time robust optimization problems with model uncertainty (see, e.g., \cite{BKN21,PW22,PW23,Fouque_et_al16,biagini2017robust}). Moreover, the uniform ellipticity condition of $\sigma^o$ and the Bene\v{s} condition fit into classical utility maximization problems with a dual/martingale approach (see, e.g., \cite{KLS,LabHeu2007,KW2000}).
\end{rem}

Having completed the description of the underlying processes, we can proceed to describe the decision maker's model.

\begin{as}\label{dfn:control}
	Let ${\cal K}:\Omega \times [0,T]\ni (\omega,t)\twoheadrightarrow {\cal K}(\omega,t)\subseteq \mathbb{R}^d$ be a correspondence (i.e., a set-valued map) satisfying that
		\begin{itemize}
			\item [(i)] ${\cal K}$ is weakly $\mathbb{F}$-{predictably} measurable, i.e., 
			\[
				{\cal K}^{-1}(F):=\{(\omega,t)\in \Omega \times [0,T]\,|\, {\cal K}(\omega,t)\cap F \neq \emptyset\}\in {\cal P}
			\]
			for every open set $F\subseteq \mathbb{R}^d$, where ${\cal P}$ denotes the $\mathbb{F}$-predictable $\sigma$-algebra;
			\item [(ii)] ${\cal K}(\omega,t)$ is a non-empty, closed and convex subset of $\mathbb{R}^d$ for every $(\omega,t)\in \Omega \times [0,T]$;
			\item [(iii)] It holds that $K:=\sup_{(\omega,t)\in\Omega \times [0,T]}\sup_{v\in {\cal K}(\omega,t)}|v|<\infty$. 
		\end{itemize}
\end{as}

Then the set of \emph{admissible controls/strategies} is given by 
	\[
	{\cal A}:=\left\{H\;|\;\mbox{$\mathbb{R}^d$-valued, $\mathbb{F}$-predictable processes $H=(H_t)_{t\in[0,T]}$ s.t.\ $H\in {\cal K}$ $\mathbb{P}\otimes dt$-a.e.}\right\}. 
	\]
Note that ${\cal A}$ is non-empty.
This follows from the Kuratowski-Ryll-Nardzewski selection theorem (see, e.g., \cite[Theorem 18.13]{CharalambosKim2006infinite}) which guarantees the existence of predictable selectors $(\omega,t)\to H_t(\omega)\in {\cal K}(\omega,t)$. 

	Let us provide some examples for correspondences ${\cal K}$ satisfying Assumption \ref{dfn:control}. The proofs can be found in Section~\ref{sec:remaining.proofs}.

\begin{lem}\label{lem:set_value_ex}
	Define the following two correspondences.
	\begin{enumerate}
		\item [(i)] ${\cal K}^{\operatorname{Ball}}: \Omega \times [0,T]\ni (\omega,t)\twoheadrightarrow {\cal K}^{\operatorname{Ball}}(\omega,t)\subseteq \mathbb{R}^d$ is given by
		\[
			\quad\qquad {\cal K}^{\operatorname{Ball}}(\omega,t):=\Big\{v\in\mathbb{R}^d\,\Big|\, \big |c_t^{\operatorname{center}}(\omega)-v \big|\leq r^{\operatorname{radius}}_t(\omega)\Big\},
		\]
		where $(c^{\operatorname{center}}_t)_{t\in[0,T]}$ and $(r^{\operatorname{radius}}_t)_{t\in[0,T]}$ are $\mathbb{F}$-predictable processes having values in a bounded subset of $\mathbb{R}^d$ and $[0,+\infty)$, respectively.
		\item [(ii)] ${\cal K}^{\operatorname{Box}}: \Omega \times [0,T]\ni (\omega,t)\twoheadrightarrow {\cal K}^{\operatorname{Box}}(\omega,t)\subseteq\mathbb{R}^d$ is given by
		\begin{align*}
			\quad\quad\qquad {\cal K}^{\operatorname{Box}}(\omega,t):= \prod_{i=1}^d \,[\underline a^i_t(\omega),\overline a^i_t(\omega)]=\Bigg\{v=(v^1,\dots,v^d)^\top\in\mathbb{R}^d\,\bigg|\, \begin{aligned}
				&\,v^i\in[\underline a^i_t(\omega),\overline a^i_t(\omega)] \subset \mathbb{R}\;\; \\
				&\,\,\forall \,i=1,\dots,d.
			\end{aligned}
			\Bigg\},
		\end{align*}
		where for every $i=1\dots,d$, $(\underline a_t^i)_{t\in[0,T]}$ and $(\overline a_t^i)_{t\in[0,T]}$ are $\mathbb{F}$-predictable processes having values in a bounded subset of $\mathbb{R}$ such that $\underline a^i_t(\omega)\leq \overline a^i_t(\omega)$ for every~$(\omega,t)\in\Omega\times[0,T]$.
	\end{enumerate}
	Then, both correspondences ${\cal K}^{\operatorname{Ball}}$ and ${\cal K}^{\operatorname{Box}}$ satisfy Assumption \ref{dfn:control}.
\end{lem}

Fix $x_0\in \mathbb{R}$, the initial  capital of the market participant.
For $(b,\sigma) \in \mathcal{B}^\varepsilon$ and $H\in {\cal A}$,\;the market participant's controlled process $X^{H;b,\sigma}=(X_t^{H;b,\sigma})_{t\in [0,T]}$ is given by the stochastic\;integral
\begin{align}\label{eq:control_process}
	X^{H;b,\sigma}_t := x_0 + (H\cdot S^{b,\sigma})_t  = x_0 + \int_0^t H_u^\top b_u du + \int_0^t H_u^\top \sigma_u dW_u,\quad t\in[0,T].
\end{align}
For shorthand notation, given $H\in{\cal A}$, denote by $X^{H;o}$ the controlled process under the postulated It\^o semimartingale $S^o$, i.e., $X^{H;o}_t := x_0+(H\cdot S^o)_t$.

\vspace{0.5em}
Assume that the decision maker holds a financial position $l(S^{b,\sigma}_T)$ where $l\colon\mathbb{R}^d\to\mathbb{R}$ is a given function.
Moreover, her intermediate and terimal cost functions (representing her individual preferences) are given by $g:\Omega\times[0,T]\times \mathbb{R}\times\mathbb{R}^d\to \mathbb{R}$ and $f:\Omega\times\mathbb{R}^2\rightarrow \mathbb{R}$, and she aims to solve the robust optimization problem
	\begin{align*}
		V(\varepsilon):= \inf_{H\in {\cal A}} {\cal V}(H,\varepsilon):= \inf_{H\in {\cal A}}\sup_{(b,\sigma)\in \mathcal{B}^\varepsilon}\mathbb{E} \bigg[\int_0^Tg\big(t,X_t^{H;b,\sigma},H_t\big)dt+ f \big(X_T^{H;b,\sigma},l(S_T^{b,\sigma})\big)\bigg].
	\end{align*}

We impose the following conditions on the functions $g,f$ and $l$.
\begin{as}\label{as:objective} The cost functions $g:\Omega\times[0,T]\times \mathbb{R}\times \mathbb{R}^d \ni (\omega,t,x,h) \to g(\omega,t,x,h)\in \mathbb{R}$ and $f:\Omega\times\mathbb{R}\times \mathbb{R}\ni (\omega,x,y)\rightarrow f(\omega,x,y)\in \mathbb{R}$ satisfy the following conditions:
	\begin{itemize}
		\item [(i)] $g(\cdot,\cdot,x,h)$ is $\mathbb{F}$-predictable\,and\,$f(\cdot,x,y)$ is ${\cal F}_T$-measurable for every $(x,y,h)\in \mathbb{R}\times\mathbb{R}\times\mathbb{R}^d$. Furthermore, $g(\omega,t,\cdot,\cdot )\in C^{2}(\mathbb{R}\times \mathbb{R}^d)$, $f(\omega,\cdot,\cdot)\in C^2(\mathbb{R}^2)$ for every $(\omega,t)\in \Omega \times [0,T]$.
		\item [(ii)] there exist $0<r<\min\{\frac{p-2}{2}, p-3\}$ and $\overline C_{2}>0$ such that 
		\begin{align*}
			\hspace{3.em}
			\begin{aligned}
			\|D^2_{x,h}g(\omega,t,x,h) \|_{\operatorname{F}}+\lVert D_{x,y}^2f(\omega,x,y) \rVert_{\operatorname{F}} \leq \overline C_{2}(1+\lvert (x,h)^\top|^{r}+|y|^r)
			\end{aligned}
		\end{align*}
		for every  $(\omega,t,x,y,h) \in \Omega\times[0,T]\times \mathbb{R}\times \mathbb{R}\times\mathbb{R}^d $, with $p> 3$ defined in Assumption~\ref{as:posterior_refer}\;(i);
		\item [(iii)] there exists $\underline C_0\in \mathbb{R}$ such that $g\geq \underline C_0$ and $f\geq \underline C_0$; 
		\item [(iv)] $D_{x,h}^2g\in \mathbb{S}_+^{d+1}$ (i.e., $g$ is convex in $(x,h)$) and there is $\underline C_{2}>0$ such that~$\partial_{xx}f\geq \underline C_{2}$ (i.e., $f$ is strongly convex in~$x$).
	\end{itemize}
	The function $l:\mathbb{R}^d\ni s\to l(s)\in \mathbb{R}$ satisfies the following:
	\begin{itemize}
		\item [(v)] $l\in C^2(\mathbb{R}^d)$;
		\item [(vi)] all first order and second order derivatives of $l$ are Lipschitz continuous and bounded.
	\end{itemize}
\end{as}

\begin{rem}\label{rem:objective} 
	If Assumption \ref{as:objective} is satisfied, a straightforward application of the fundamental theorem of calculus implies that the following growth conditions hold (which we shall use often in what follows).
	\begin{itemize}
		\item [(i)] There is $\widetilde{C}>0$ such that for every $(\omega,t,x,y,h) \in \Omega\times [0,T]\times \mathbb{R}\times \mathbb{R}\times \mathbb{R}^d$,
		\begin{align*}
			\begin{aligned}
			|g(\omega,t,x,h) |+\lvert  f(\omega,x,y)  \rvert&\leq \widetilde{C}(1 +\lvert x \rvert^{r+2}+\lvert h \rvert^{r+2} +\lvert y \rvert^{r+2} ),\\
			|\nabla_{x,h} g(\omega,t,x,h) |+\lvert \nabla_{x,y} f(\omega,x,y) \rvert&\leq \widetilde{C}(1+  \lvert x \rvert^{r+1} + \lvert h \rvert^{r+1}+ \lvert y \rvert^{r+1}).
			\end{aligned}
		\end{align*}
		\item [(ii)] There is $C_l>0$ such that for every $s,\hat{s}\in \mathbb{R}^d$,
		\begin{align*}
		\lvert l (s)\rvert \leq C_l (1+ \lvert s \rvert)
		\quad\text{and}\quad 
		\lvert \nabla_s l (s) \rvert +\lVert {D}^2_s l (s) \rVert_{\operatorname{F}}  &\leq C_{l}, \\
		\lvert l({s}) -  l(\hat{s}) \rvert + \lvert \nabla_s l({s}) -  \nabla_s l(\hat{s}) \rvert +\lVert D_s^2 l(s)- D_s^2 l(\hat{s}) \rVert_{\operatorname{F}}  
		&\leq C_{l}\lvert s-\hat{s} \rvert.
		\end{align*}
	\end{itemize}
\end{rem}

\subsection{Optimization for the baseline parameters and BSDEs} 

In this section we collect some preliminary results, including the existence of an optimal strategy $H^\ast\in\mathcal{A}$.
In particular, it turns out to be convenient to characterize the processes $Y, \mathcal{Y}, Z, \mathcal{Z}$ that appear in the formula for $V'(0)$ (see \eqref{eq:intro.sensitivity}) using an auxiliary BSDE formulation. 

\begin{pro}\label{pro:main0}
	Suppose that Assumption \ref{as:posterior_refer}, \ref{dfn:control}, and \ref{as:objective} are satisfied. 
	Then the following hold.
	\begin{itemize} 
		\item [(i)] There exists a unique optimizer $H^*\in {\cal A}$ satisfying
		\[
		V(0) = {\cal V}(H^*,0)=\mathbb{E}\bigg[\int_0^Tg\big(t,X_t^{H^*;o},H_t^*\big)dt+f \big(X_T^{H^*;o},l(S_T^{o})\big)\bigg].
		\]
		\item [(ii)] There exists a unique solution $(Y^*,Z^*,L^*)\in \mathscr{S}^{2}(\mathbb{R}) \times \mathscr{H}^{2}(\mathbb{R}^d) \times \mathscr{M}^2(\mathbb{R})$ of for $t\in [0,T]$
		\begin{align}\label{eq:BSDE_re_x}
			\quad Y_t^* = \partial_{x}f\big(X_T^{H^*;o},l(S_T^o)\big)+\int_t^T\partial_x g\big(s,X_s^{H^*;o},H_s^*\big)ds  - \int_t^T (Z_u^*)^\top dW_u-(L_T^*-L_t^*), 
		\end{align}
		with $L^*_0=0$, where $L^*$ and $(\int_0^t (Z_s^*)^\top dW_s)_{t\in[0,T]}\in\mathscr{M}^2(\mathbb{R})$ are strongly orthogonal.\footnote{We refer to \cite[Chapter IV.3, p.148]{Protter2005} for the definition of the strong orthogonality and the following equivalences: two martingales $M,N\in \mathscr{M}^2(\mathbb{R})$ are strongly orthogonal if and only if $(M_tN_t)_{t\in[0,T]}$ is a uniformly integrable $(\mathbb{F},\mathbb{P})$-martingale if and only if $([M,N]_t)_{t\in[0,T]}$ is a uniformly integrable $(\mathbb{F},\mathbb{P})$-martingale.}
		\item[(iii)]
		There exists a unique solution
		 $({{\cal Y}}^*,{{\cal Z}}^*,{{\cal L}}^*)\in(\mathscr{S}^{2}(\mathbb{R}))^d \times (\mathscr{H}^{2}(\mathbb{R}^d))^d \times (\mathscr{M}^2(\mathbb{R}))^d$ 
		of
		\begin{align}\label{eq:BSDE_re_y}
			\quad {\mathcal{Y}}_t^* = \partial_{y} f\big(X_T^{H^*;o},l(S_T^o)\big)\nabla l(S_T^o)   - \int_t^T {{\cal Z}}^*_u dW_u-({{\cal L}}_T^* -{{\cal L}}_t^* ),\quad t\in [0,T],
		\end{align}
		with ${{\cal L}}^*_0=0$,  where ${{\cal L}}^{\ast}$ and $(\int_0^t \mathcal{Z}_s^{*} dW_s)_{t\in[0,T]}$ are strongly orthogonal.\footnote{Two $d$-dimensional martingales $M=(M^1,\dots,M^d)^\top$ and $N=(N^1,\dots,N^d)^\top$ are strongly orthogonal if $M^i, N^i\in \mathscr{M}^2(\mathbb{R})$ are strongly orthogonal for all $i=1,\dots,d$.}
	\end{itemize}
\end{pro}

\begin{rem}
	In the context of solutions to BSDE's, throughout this article  we denote the solutions to \emph{multi-dimensional} BSDE's as in \eqref{eq:BSDE_re_y} by \emph{calligraphic} letters.
	In particular, \eqref{eq:BSDE_re_y} is to be  understood in the following sense: $\mathcal{Y}^*=(\mathcal{Y}^{*,1},\dots,\mathcal{Y}^{*,d})^\top$,  $\mathcal{Z}^*=(\mathcal{Z}^{*,1},\dots,\mathcal{Z}^{*,d})^\top$, and $\mathcal{L}^*=(\mathcal{L}^{*,1},\dots,\mathcal{L}^{*,d})^\top$ are vectors consisting of processes, and for every $1\leq i \leq d$,
	\begin{align*}
		\quad\quad {\mathcal{Y}}_t^{*,i} = \partial_{y} f\big(X_T^{H^*;o},h(S_T^o)\big)\partial_{s_i}l(S_T^o)   - \int_t^T ({{\cal Z}}^{*,i}_u)^\top dW_u-({{\cal L}}_T^{*,i} -{{\cal L}}_t^{*,i}),\quad t\in [0,T].
	\end{align*} 
\end{rem}

Let us note that the proof of Proposition \ref{pro:main0} is relatively standard: part (i) follows from a Koml\'os-type argument and parts (ii)  and (iii) follow from employing the Galtchouk-Kunita-Watanabe (GKW) decomposition. The complete proof can be found in Section \ref{sec:proof:pro:main0}.

Before we proceed to state our main result, let us briefly comment on some properties of the BSDE-solution in Proposition \ref{pro:main0}.

\begin{lem}\label{lem:explicit.Y.Z} Suppose that Assumptions \ref{as:posterior_refer}, \ref{dfn:control}, and \ref{as:objective} are satisfied and let $(Y^*,Z^*,L^*)$ be the unique solution of \eqref{eq:BSDE_re_x}.
Then, for every $t\in[0,T]$, $\mathbb{P}$-a.s.,
	\begin{align*}
		Y_t^* &=\mathbb{E}\bigg[\partial_{x}f\big(X_T^{H^*;o},l(S_T^o)\big)+\int_t^T\partial_x g\big(s,X_s^{H^*;o},H_s^*\big)ds \, \bigg|\, {\cal F}_t\bigg],\\
		Z_t^* &= \frac{d}{dt}\Big(\langle Y^*,W^1\rangle_t,\cdots,\langle  Y^*,W^d\rangle_t\Big)^\top.
	\end{align*}
	 Moreover, if $(\mathcal{Y}^*,\mathcal{Z}^*,\mathcal{L}^*)$ is the unique solution of  \eqref{eq:BSDE_re_y}, then for every  $t\in[0,T]$, $\mathbb{P}$-a.s., 
	\begin{align*}
		\mathcal{Y}^\ast_t
		&=\mathbb{E}\Big[\partial_{y}f\big(X_T^{H^*;o},l(S_T^o)\big)\nabla_s l(S_T^o)\, \Big|\, {\cal F}_t\Big],\\
		\mathcal{Z}_t^{*,i}
		&= \frac{d}{dt}\Big(\langle  \mathcal{Y}^{*,i},W^1\rangle_t,\cdots,\langle \mathcal{Y}^{*,i},W^d\rangle_t\Big)^\top, \quad i=1,\dots,d.
	\end{align*}
\end{lem}

Under certain (rather strong) conditions on the regularity of $Y^*$ and $\mathcal{Y}^*$, Lemma \ref{lem:explicit.Y.Z} and It\^o's formula ensure that  $Z^*$ and $\mathcal{Z}^*$ can be calculated via a  Feynman-Kac representation.
To formulate the result denote by $\mathcal{C}^{1,2,2}$ the set of all continuous functions from $[0,T]\times \mathbb{R}\times \mathbb{R}^d$ to $\mathbb{R}$  which are continuously differentiable on $[0,T)$ and twice continuously differentiable on $\mathbb{R}\times\mathbb{R}^d$. 

\begin{cor}\label{cor:regularity} 
	Suppose that Assumptions \ref{as:posterior_refer}, \ref{dfn:control}, and \ref{as:objective} are satisfied. Let  $H^*$ be the unique optimizer of $V(0)$ defined in Proposition \ref{pro:main0}\;(i).
	 Moreover,  assume that there are $J\in \mathcal{C}^{1,2,2} $ and $\mathcal{J}=( {\cal J}^1,\dots,{\cal J}^d)^\top \in (\mathcal{C}^{1,2,2})^d$ such that for every $t\in[0,T]$, $\mathbb{P}$-a.s.
	\begin{align}\label{eq:condi_forms}
		\begin{aligned}
			Y_t^* = J\big(t,X_t^{H^*;o},S_t^o\big),\qquad
			\mathcal{Y}_t^*=\mathcal{J} \big(t,X_t^{H^*;o},S_t^o\big).
		\end{aligned}
	\end{align}
	Then, for every $t\in[0,T)$, $\mathbb{P}$-a.s., 
	\begin{align*}
		\begin{aligned}
			Z_t^*
			&=(\sigma_t^o)^\top \big[\partial_xJ(t,X_t^{H^*;o},S_t^o)H_t^* + \nabla_s J(t,X_t^{H^*;o},S_t^o) \big],\\
			\mathcal{Z}^{*,i}_t
			&=(\sigma_t^o)^\top \big[\partial_x \mathcal{J}^i(t,X_t^{H^*;o},S_t^o)H_t^* + \nabla_s \mathcal{J}^i(t,X_t^{H^*;o},S_t^o) \big], \quad i=1,\dots,d,
		\end{aligned}
	\end{align*}
	where $\nabla_s \mathcal{J}^i$ denotes the gradient of $\mathcal{J}^i(t,x,s)$ with respect $s$.
\end{cor}

The proofs of Lemma \ref{lem:explicit.Y.Z} and Corollary \ref{cor:regularity}  are presented in  Section \ref{sec:remaining.proofs}, where we also provide some sufficient conditions for the regularity assumption in Corollary \ref{cor:regularity}.

\subsection{Main results} 

With all this notation set in place, we present the main result of this article pertaining the characterization of the behavior of 
\[
		V(\varepsilon)= 
		\inf_{H\in {\cal A}} {\cal V}(H,\varepsilon)
		= \inf_{H\in {\cal A}}\sup_{(b,\sigma)\in \mathcal{B}^\varepsilon}\mathbb{E} \bigg[\int_0^Tg\big(t,X_t^{H;b,\sigma},H_t\big)dt+ f \big(X_T^{H;b,\sigma},l(S_T^{b,\sigma})\big)\bigg]
\]
for small $\varepsilon$.
Recall that $\mathcal{B}^\varepsilon$  and ${\cal A}$ are defined in \eqref{eq:def.ball} and \eqref{dfn:control}, respectively.

\begin{thm}\label{thm:main}
	Suppose that Assumptions \ref{as:posterior_refer}, \ref{dfn:control}, and \ref{as:objective} are satisfied and set $q:=\frac{p}{p-1}$ to be the conjugate H\"older exponent of $p> 3$ (given in Assumption \ref{as:posterior_refer}). 
	Let  $H^*$ be the unique optimizer of $V(0)$ (see Proposition \ref{pro:main0}\;(i)) and let $(Y^*,Z^*)$ and $(\mathcal{Y}^*,\mathcal{Z}^*)$ be the first two components of the solutions of \eqref{eq:BSDE_re_x} and \eqref{eq:BSDE_re_y}, respectively. 
	Then, as $\varepsilon \downarrow 0$,
	\[ V(\varepsilon) = V(0) + \varepsilon V'(0) + O(\varepsilon^2),\]
	where
	\begin{align*}
	V'(0) 
	&= \gamma \lVert Y^*H^*+\mathcal{ Y}^* \rVert_{\mathbb{L}^q}   +  \eta \lVert Z^* (H^*)^\top + \mathcal{ Z}^* \rVert_{\mathbb{H}^q} \\
	&=  \gamma \mathbb{E}\Bigg[ \int_0^T |Y_t^*H_t^*+\mathcal{Y}^*_t |^q \,dt \Bigg]^{1/q}+
	 \eta \mathbb{E}\Bigg[ \bigg( \int_0^T \| Z_t^* (H_t^*)^\top + \mathcal{ Z}^*_t \|_{\rm F}^2 \,dt \bigg)^{q/2} \Bigg]^{1/q}. 
	\end{align*}
\end{thm}

We emphasize that $Y^*,Z^*,\mathcal{Y}^*,\mathcal{ Z}^*$ are given explicitly in Lemma \ref{lem:explicit.Y.Z}.
In particular, the statements made in the introduction on the form of $V'(0)$ follow from a combination of Theorem~\ref{thm:main} and Lemma \ref{lem:explicit.Y.Z}. 

\vspace{0.5em}

Next, recall that $H^*$ is the unique optimizer for $V(0)$ (given in Proposition \ref{pro:main0}\;(i)) and for any $\varepsilon\geq0$, consider
\begin{align}\label{eq:value_worst}
	V^\ast(\varepsilon):= {\cal V}(H^*,\varepsilon) =\sup_{(b,\sigma)\in \mathcal{B}^\varepsilon} \mathbb{E} \bigg[\int_0^Tg\big(t,X_t^{H^*;b,\sigma},H^*_t\big)dt+ f \big(X_T^{H^*;b,\sigma},l(S_T^{b,\sigma})\big)\bigg].
\end{align}
Therefore, $V^\ast(\varepsilon)$ represents the worst-case value when the market participant sticks to the optimal strategy $H^\ast$ calculated based on $b^o$ and $\sigma^o$ while the actual parameters lie within $ \mathcal{B}^\varepsilon$. 
As it happens, the values of  $V^\ast(\varepsilon)$ and $V(\varepsilon)$ are equal up to a second-order correction term.

\begin{thm}\label{thm:main2}
	Suppose that Assumptions \ref{as:posterior_refer}, \ref{dfn:control}, and \ref{as:objective} are satisfied. 
	Then, as $\varepsilon \downarrow 0$,
	\[
	V^\ast(\varepsilon)=V(\varepsilon)+O(\varepsilon^2).
	\]
	In particular, $(V^\ast)'(0)=V'(0)$.
\end{thm}
\begin{rem}
	Theorem \ref{thm:main} builds on  Assumptions \ref{as:posterior_refer}, \ref{dfn:control}, and \ref{as:objective}.
	We already showcased in Lemma \ref{lem:posterior_refer2} that part (ii) of Assumption \ref{as:posterior_refer} is not too restrictive and naturally satisfied in various examples.
	Additionally, note that part (i) of Assumption  \ref{as:posterior_refer} (the $p$-integrability of $b$ and~$\sigma$) is related to Assumption \ref{as:objective} and is almost minimal: if $f$ or $g$ have faster than $p$-polynomial growth, then the optimization problem $V(0)$ we consider is not even well-posed in general.
	In a similar manner, the assumptions that $f$ and $g$ are twice differentiable are needed to ensure the approximation of $V(\varepsilon)$ up to an $O(\varepsilon^2)$ error term; if they are only once differentiable the best one can hope of is an $o(\varepsilon)$ error term.
	Finally, it is conceivable that Assumption~\ref{dfn:control} can be relaxed, and instead of requiring $K(\omega) = \sup_{t\in[0,T]} \sup_{v\in \mathcal{K}(w,t)} |v|$ to be  bounded uniformly in $\omega$, it seems  sufficient for  $K(\cdot)$ to be  integrable in a suitable sense. 
	However, we believe that the technical difficulties associated with this relaxation would introduce greater complexities than the added benefit, and thus we follow the many works studying related topics (see, e.g. \cite{bartl2023sensitivity,czichowsky2012convex,karatzas2007numeraire,cvitanic1992convex,LabHeu2007}) in which uniform boundedness of  $K(\cdot)$ is assumed.

	Finally, let us note that by the choice of the set of parameters  $\mathcal{B}^\varepsilon$, the robust optimization problem is not time-consistent. 
	One possible approach to establish time-consistency seems to be by considering the case $p=\infty$ (but with the essential supremum taken over both $\omega$ and $t$).
	It is likely that our arguments can be extended to this setting, also using methods from our paper  \cite{bartl2024numerical} (which focuses on numerical methods for the implementation of robust pricing via solving nonlinear Kolmogorov PDEs).	
\end{rem}


\section{Example: sensitivity of robust mean-variance hedging}\label{sec:exmp_mv_hedging} 
To illustrate our main results, let us consider the following robust mean-variance hedging problem (e.g., \cite{DR1991,czichowsky2012convex,Schweizer1992,ZhouLi2000}): for  every $\varepsilon\geq 0$, let
\begin{align}\label{eq:primal_quad}
	V(\varepsilon) = \inf_{H\in {\cal A}}\sup_{(b,\sigma)\in \mathcal{B}^\varepsilon}\mathbb{E} \bigg[\int_0^Tg^{\operatorname{MV}}(t,X_t^{H;b,\sigma},H_t)dt+ f^{\operatorname{MV}}(X_T^{H;b,\sigma}) \bigg],
\end{align}
where $g^{\operatorname{MV}}:\Omega\times[0,T]\times \mathbb{R}\times \mathbb{R}^d\to \mathbb{R}$ and $f^{\operatorname{MV}}:\Omega\times\mathbb{R}\to \mathbb{R}$ are given by 
\begin{align*}
	g^{\operatorname{MV}}(\omega,t,x,h):=A_t(\omega) x^2+x B_t^\top(\omega) h+ h^\top C_t(\omega) h, \qquad f^{\operatorname{MV}}(\omega,x):= (x-\xi(\omega))^2,
\end{align*}
for every $(\omega,t,x,h)\in\Omega\times[0,T]\times \mathbb{R}\times \mathbb{R}^d$,  where $\xi$ is a real-valued ${\cal F}_T$-measurable random variable and $(A_t,B_t,C_t)_{t\in[0,T]}$ is a triplet of $(\mathbb{R},\mathbb{R}^d,\mathbb{R}^{d\times d})$-valued $\mathbb{F}$-predictable processes satisfying
	\[
	\sup_{(\omega,t)\in\Omega\times[0,T]}\Big\{|\xi(\omega)|+|A_t(\omega)|+|B_t(\omega)|+\|C_t(\omega)\|_{\operatorname{F}}\Big\}<\infty,
	\]
	and where for every $(\omega,t,x,h)\in\Omega\times[0,T]\times \mathbb{R}\times \mathbb{R}^d$,
	\begin{align*}
			D_{x,h}^2g^{\operatorname{MV}}(\omega,t, x, h)=\begin{pmatrix}
			2A_t(\omega) &B_t^\top(\omega) \\
			B_t(\omega) &2C_t(\omega) 
		\end{pmatrix}
	\end{align*}
belongs to $\mathbb{S}_+^{d+1}$. 
Under this setting, $g^{\operatorname{MV}}$ and $f^{\operatorname{MV}}$ satisfy the conditions in Assumption~\ref{as:objective}.

\begin{as}\label{as:quad_risk} The following conditions hold:
	\begin{itemize}
		\item [(i)] $\mathbb{F}$ is the augmented filtration generated by the $d$-dimensional Brownian motion $W$.
		\item [(ii)] $b^o$ and $\sigma^o$ are uniformly bounded and $\sigma^o$ is also invertible. 
		\item [(iii)] The correspondence ${\cal K}:\Omega\times [0,T]\twoheadrightarrow \mathbb{R}^d$ (see Assumption~\ref{dfn:control}) is constant-valued, i.e., ${\cal K}(\omega,t)=\overline{{\cal K}}$ for every $(\omega,t)\in \Omega\times [0,T]$ with respect to some non-empty, closed, convex and bounded subset
		$\overline {\cal K}\subset \mathbb{R}^d$ containing $0$.
	\end{itemize}
\end{as}

Assumption \ref{as:quad_risk} enables us to make use  of the stochastic maximum principle (e.g., \cite[Theorem 1.5]{cadenillas1995stochastic}) to obtain necessary and sufficient optimality conditions for the problem~$V(0)$ given in \eqref{eq:primal_quad}.
Moreover, using the convex duality approach (see e.g., \cite[Theorems 7,\;9]{li2018constrained}, \cite[Theorems 2,\;3]{vega2023duality}) one can characterize the optimizer $H^*$  (for $V(0)$) via dual forward-backward SDEs revolving around the stochastic maximum principle of the corresponding dual problem of $V(0)$.

In what following, we state the results connecting the optimizer $H^*$ and the BSDE formulation given in Proposition \ref{pro:main0} and provide the corresponding sensitivity results. 
To that end, set ${\mathcal{G}}:=\mathbb{R}\times \mathscr{H}^2(\mathbb{R})\times \mathscr{H}^2(\mathbb{R}^d)$ and denote the dual problem of $V(0)$ given in \eqref{eq:primal_quad} by 
\begin{align}\label{eq:dual_quad}
	\Psi:=\inf_{(\gamma,\alpha,\beta)\in {\mathcal{G}}} \left(x_0\gamma  + \mathbb{E}\left[\int_0^T\widetilde g^{\operatorname{MV}}(t,\alpha_t,\beta_t)dt+\frac{1}{4}(\Gamma_T^{\gamma;\alpha,\beta})^2-\Gamma_T^{\gamma;\alpha,\beta} \xi \right]  \right),
\end{align}
where $\widetilde g^{\operatorname{MV}}:[0,T]\times \mathbb{R}\times \mathbb{R}^d\ni (\omega,t,\alpha,\beta)\to \widetilde g^{\operatorname{MV}}(\omega,t,\alpha,\beta)\in\mathbb{R}$ denotes the conjugate function of the map $(x,h)\to g^{\operatorname{MV}}(\omega,t,x,h)$ with constraint on $h\in \overline{\cal K}$ for each $(\omega,t)\in\Omega\times[0,T]$, i.e.,
\[
	\widetilde g^{\operatorname{MV}}(\omega,t,\alpha,\beta):= \sup_{x\in \mathbb{R},\,h\in \overline{\cal K}} \Big\{x\alpha+h^\top \beta-g^{\operatorname{MV}}(\omega,t,x,h)\Big\},
\] 
and where $\Gamma^{\gamma,\alpha,\beta}:=(\Gamma^{\gamma,\alpha,\beta}_t)_{t\in[0,T]}$ denotes a dual process satisfying
\begin{align}\label{eq:dual}
	\begin{aligned}
		\Gamma^{\gamma,\alpha,\beta}_t =& \gamma + \int_0^t \alpha_udu+\int_0^t (\beta_u)^\top  dW_u\\
		&+\int_0^t \frac{({\sigma}_{u,d}^{o})^\top}{|{\sigma}_{u,d}^{o}|^2}   \left(\beta_u-b^o_u\Gamma_u^{\gamma,\alpha,\beta} - \sum_{i=1}^d {\sigma}_{u,i}^{o}\beta_u^i\right)dW_u^d,
	\end{aligned}
\end{align}
with ${\sigma}_{\cdot,i}^{o}=({\sigma}_{t,i}^{o})_{t\in[0,T]}$, $i=1,\dots,d$, denoting the $i$-th `column' vector process of $\sigma^o$. 

Recall that  $H^*\in {\cal A}$ is the unique optimizer for  $V(0)$ (see Proposition \ref{pro:main0}\;(i)) and that the corresponding controlled process $X^{H^*;o}$ is given by $X^{H^*;o}_t =x_0 + (H^*\cdot S^o)_t$. 
Under Assumption~\ref{as:quad_risk}\;(i), the BSDE given in Proposition \ref{pro:main0}\;(ii) can be rewritten as 
\begin{align}\label{eq:BSDE_quad_primal}
	Y_t^* = 2 (X^{H^*;o}_T -\xi)+\int_t^T(2A_tX_t^{H^*;o}+B_t^\top H_t^*)dt-\int_t^T (Z_u^*)^\top dW_u,
\end{align}
with $(Y^*,Z^*)\in\mathscr{S}^{2}(\mathbb{R}) \times \mathscr{H}^{2}(\mathbb{R}^d)$.

Furthermore, let $(\gamma^*,\alpha^*, \beta^*)\in {\mathcal{G}}$  be the optimizer of $\Psi$ given in \eqref{eq:dual_quad} (see e.g., \cite[Proposition 5.4]{LabHeu2007}) and denote by $\Gamma^*:=\Gamma^{\gamma^*,\alpha^*,\beta^*} $ the corresponding optimal dual process given in \eqref{eq:dual}. 
Let $(P^*,Q^*)\in\mathscr{S}^{2}(\mathbb{R}) \times \mathscr{H}^{2}(\mathbb{R}^d) $ be the unique solution of the BSDE (see e.g., \cite[Theorem~6.2.1]{pham2009continuous}) 
\begin{align}\label{eq:BSDE_quad_dual}
	P_t^* = \frac{1}{2}\Gamma_T^* -\xi - \int_t^T Q^{*,d}_u \frac{({\sigma}_{u,d}^{o})^\top}{|{\sigma}_{u,d}^{o}|^2}   b_u^o  du  - \int_t^T (Q^*_u)^\top dW_u,
\end{align}
with $Q^{*,i}$, $i=1,\dots,d$, denoting the $i$-th component of $Q^*$.

With all this preparation, we can state the stochastic maximum principle for $V(0)$ and $\Psi$, and provide the sensitivity of the problem $V(\varepsilon)$ given in \eqref{eq:primal_quad}. 
\begin{cor}
	Suppose that Assumption \ref{as:quad_risk} holds. Let $(Y^*,Z^*)$ and $(P^*,Q^*)$ be the solution of the BSDEs given in \eqref{eq:BSDE_quad_primal} and \eqref{eq:BSDE_quad_dual}, respectively. Then the following hold:
	\begin{itemize}
		\item [(i)] 
		$H^*\in{\cal A}$ is optimal for $V(0)$ if and only if the controlled process $X^{H^*;o}$ and $(Y^*,Z^*)$ 
		satisfy that for every $H\in {\cal A}$, $\mathbb{P}\otimes dt$-a.e.,
		\begin{align*}
			(H_t^*-H_t)^\top \big(- b_t^o Y_t^* - \sigma_t^o Z_t^*-X_t^{H^*;o}B_t  -2C_tH_t^* \big)\geq 0.
		\end{align*}
		Furthermore, $(\gamma^*,\alpha^*, \beta^*)\in \mathcal{G}$ is optimal for $\Psi$ if and only if the dual process $\Gamma^*=\Gamma^{\gamma^*,\alpha^*, \beta^*}$ and  $(P^*,Q^*)$ satisfy that for $i=1,\dots,d-1$, $\mathbb{P}\otimes dt$-a.e.
		\begin{align*}
			Q_t^{*,i}- Q^{*,d}_t \frac{({\sigma}_{t,d}^{o})^\top}{|{\sigma}_{t,d}^{o}|^2}  {\sigma}_{t,i}^{o}=0,\qquad  - Q^{*,d}_t \frac{{\sigma}_{t,d}^{o}}{|{\sigma}_{t,d}^{o}|^2}    \in {\cal K},\quad \mbox{and} \quad P_0^*= -x_0.
		\end{align*}
		In particular, the optimizer $H^*$ is characterized by $\mathbb{P}\otimes dt$-a.e.
		\[
		H_t^* = -Q^{*,d}_t \frac{({\sigma}_{t,d}^{o})^\top}{|{\sigma}_{t,d}^{o}|^2}.
		\]
		\item[(ii)] As $\varepsilon\downarrow 0$,
		\[
		V(\varepsilon) = V(0)+ \varepsilon\cdot  \Bigg( \gamma \bigg\lVert Y^* Q^{*,d}  \frac{{\sigma}_{\cdot ,d}^{o}}{|{\sigma}_{\cdot ,d}^{o}|^2}  \bigg\rVert_{\mathbb{L}^q}   +  \eta \bigg\lVert Z^* Q^{*,d}  \frac{({\sigma}_{\cdot ,d}^{o})^\top}{|{\sigma}_{\cdot ,d}^{o}|^2}  \bigg\rVert_{\mathbb{H}^q}\Bigg)+O(\varepsilon^2).
		\]
	\end{itemize}
\end{cor}

\section{Proof of Proposition \ref{pro:main0} and Theorem \ref{thm:main}}
\label{sec:proof}

We start by highlighting  the main ideas used in the proof of Theorem \ref{thm:main}.
For simplicity we focus on the notationally lighter case when $g=0$ and $l=0$.

The first step involves a second-order Taylor expansion,  showing that for every $H\in\mathcal{A}$  and $(b,\sigma)\in\mathcal{B}^\varepsilon$,
\begin{align}\label{eq:highlight.proof}
	\mathbb{E}\Big[  f\big(X_T^{H;b,\sigma},0\big)  \Big] - \mathbb{E}\Big[  f\big(X_T^{H;o},0\big)  \Big] 
= \mathbb{E}\Big[ \partial_{x} f\big(X_T^{H;o},0\big) \big(X_T^{H;b,\sigma} - X_T^{H;o}\big) \Big] + O(\varepsilon^2).
\end{align}

Subsequently, a crucial observation is  that the expectation appearing in the right-hand side of \eqref{eq:highlight.proof} can be expressed as a linear form involving $b-b^o$ and $\sigma-\sigma^o$:
 If $(Y^H,Z^H,L^H)$ solves the following BSDE 
	\[
			\quad\quad Y_t^H = \partial_{x}f\big(X_T^{H;o},0\big)  - \int_t^T (Z_u^H)^\top dW_u-(L_T^H-L_t^H),
			\]
			with $L_0^H=0$, then
\begin{align}
\label{eq:linear}
		\mathbb{E}\Big[ \partial_{x} f\big(X_T^{H;o},0\big) \big(X_T^{H;b,\sigma} - X_T^{H;o}\big) \Big]
		&= \langle Y^H H, b-b^o \rangle_{\mathbb{P}\otimes dt}+ \langle Z^HH^\top, \sigma-\sigma^o \rangle_{ \mathbb{P}\otimes dt,\operatorname{F}}. 
\end{align}
Next, setting $q:=\frac{p}{p-1}$ to be the conjugate H\"older conjugate of $p$ and relying on the duality within the pairings $\langle \mathbb{L}^p(\mathbb{R}^d),\mathbb{L}^q(\mathbb{R}^d), \langle\cdot,\cdot \rangle_{\mathbb{P}\otimes dt}\rangle$ and $\langle \mathbb{H}^p(\mathbb{R}^{d\times d}),\mathbb{H}^q(\mathbb{R}^{d\times d}), \langle\cdot,\cdot \rangle_{\mathbb{P}\otimes dt,\operatorname{F}}\rangle$, it follows that
\[\sup_{(b,\sigma)\in\mathcal{B}^\varepsilon}\Big(
\langle Y^H H, b-b^o \rangle_{\mathbb{P}\otimes dt}+ \langle Z^HH^\top, \sigma-\sigma^o \rangle_{ \mathbb{P}\otimes dt,\operatorname{F}} \Big)
= \varepsilon \left( \gamma \lVert Y^HH  \rVert_{\mathbb{L}^q}   + \eta \lVert Z^H H^\top \rVert_{\mathbb{H}^q} \right).\]

The final ingredient in the proof lies in showing that if $(H^\varepsilon)_{\varepsilon>0}$ are (almost) optimizers for the robust optimization problems $V(\varepsilon)$, then the $H^\varepsilon$'s converge to the unique optimizer $H^\ast$ of $V(0)$ as $\varepsilon\downarrow 0$.

\vspace{0.5em}
In the subsequent sections, we will establish the technical details required to rigorously prove these arguments.  Since our results are concerned with the behavior as $\varepsilon\downarrow 0$, we can and do assume without loss of generality that $\gamma,\eta\leq 1$. 

\subsection{Preliminary estimates and GKW decomposition}\label{sec:prelimi_GKW}
Let us provide some observations that play an instrumental role in the proof of Proposition \ref{pro:main0} and Theorem \ref{thm:main}. 
In what follows, we often make use of the following elementary inequality: for every $\beta\geq 0$ and $m\in\mathbb{N}$, 
\begin{align}\label{eq:power_triangle}
	\Big\lvert \sum_{i=1}^m a_i \Big\rvert^\beta \leq m^\beta\cdot \sum_{i=1}^m\lvert a_i \rvert^\beta,\qquad\mbox{for every}\;\;\{a_i\}_{i=1}^m\subset \mathbb{R}.
\end{align}
We call \eqref{eq:power_triangle}  the `power triangle inequality'.

Let us begin with a priori estimates on $X^{H;b,\sigma}$ (given in \eqref{eq:control_process}) and $S^{b,\sigma}$ (given in~\eqref{eq:posterior_semi_ito}).

\begin{lem}\label{lem:p_integrability}
	Suppose that Assumptions \ref{as:posterior_refer} and \ref{dfn:control} are satisfied.
	Then the following hold:
	\begin{itemize}
		\item [(i)] For every $\varepsilon\geq0$,
		\begin{align*}
			\quad \quad \sup_{(b,\sigma)\in \mathcal{B}^\varepsilon}\sup_{H\in{\cal A}} \big\lVert  X^{H;b,\sigma}- X^{H;o} \big\rVert_{\mathscr{S}^p}^p \leq C_{1} \varepsilon^p,\quad \sup_{(b,\sigma)\in \mathcal{B}^\varepsilon} \big\lVert  S^{b,\sigma}- S^{o} \big\rVert_{\mathscr{S}^p}^p \leq C_{2} \varepsilon^p,
		\end{align*}
		with $\lVert \, \cdot \, \rVert_{\mathscr{S}^{p }}^p = \mathbb{E}[\sup_{t\in [0,T]}|\cdot |^p]$, where $C_{1} :=  2^pK^p (T^{p-1}+ C_{\operatorname{BDG},p})$, $C_{2} := 2^p(T^{p-1}+ C_{\operatorname{BDG},p})$, and $C_{\operatorname{BDG},p}\geq1$ is the BDG constant (in Section \ref{sec:preliminary}) with the exponent $p>3$.
		\item [(ii)] The constants ${C}_{3}$, ${C}_{4}$ defined by
		\begin{align*}
			\quad \quad {C}_{3}:=\sup_{(b,\sigma)\in \mathcal{B}^1}\sup_{H\in{\cal A}} \big\lVert X^{H;b,\sigma} \big\rVert_{\mathscr{S}^p}^p,\quad {C}_{4}:=  \sup_{(b,\sigma)\in \mathcal{B}^1} \big\|l(S^{b,\sigma}) \big\|_{\mathscr{S}^p}^p
		\end{align*}
		satisfy ${C}_{3}, {C}_{4}<\infty$.
	\end{itemize}
\end{lem}
\begin{proof} 
We start by proving (i). 
An application of the power triangle inequality (see \eqref{eq:power_triangle}) and the BDG inequality shows that for every $\varepsilon\geq 0$, $(b,\sigma)\in \mathcal{B}^\varepsilon$ and $H\in{\cal A}$, 
\begin{align*}
	\begin{aligned}
&\big\lVert X^{H;b,\sigma}-X^{H;o}\big\rVert^p_{\mathscr{S}^p}=\mathbb{E}\bigg[\sup_{t\in[0,T]}\big|X_t^{H;b,\sigma}-X_t^{H;o}\big|^p\bigg]\\
&\quad \leq 2^p\Bigg(\mathbb{E}\Bigg[\sup_{t\in[0,T]}\Big|\int_0^t H_s^\top (b_s-b_s^o) ds\Big|^p \Bigg]+C_{\operatorname{BDG},p}\mathbb{E}\Bigg[\Big( \int_0^T \big\lvert H_t^\top  (\sigma_t-\sigma_t^o )\big\rvert^2 dt \Big)^{\frac{p}{2}} \Bigg]\Bigg). 
	\end{aligned}
\end{align*}
Moreover, since the function $x\rightarrow |x|^p$ with $p>3$ is convex and $\lvert H_t \rvert \leq K$ $\mathbb{P}\otimes dt$-a.e.\ for every $H\in {\cal A}$ (see Assumption \ref{dfn:control}), Jensen's inequality ensures that 
\[
\mathbb{E}\Bigg[\sup_{t\in[0,T]}\Big|\int_0^t H_s^\top (b_s-b_s^o) ds\Big|^p \Bigg]\leq T^{p-1} \mathbb{E}\Bigg[\int_0^T |H_s^\top (b_s-b_s^o) |^pds \Bigg]\leq T^{p-1} K^p \|b-b^o \|_{\mathbb{L}^p}^p. 
\]
Finally,  since $\lvert (\sigma_t-\sigma_t^o )^\top H_t  \lvert \leq   K \lVert \sigma_t-\sigma_t^o \rVert_{\operatorname{F}} $ $\mathbb{P}\otimes dt$-a.e.\ for every $H\in {\cal A}$, we conclude that
\begin{align}\label{eq:resd_1}
	\begin{aligned}
		&\sup_{(b,\sigma)\in \mathcal{B}^\varepsilon}\sup_{H\in{\cal A}}\big\lVert X^{H;b,\sigma}-X^{H;o}\big\rVert^p_{\mathscr{S}^p} \\
		&\quad \leq 2^pK^p  \sup_{(b,\sigma)\in \mathcal{B}^\varepsilon}  \left(T^{p-1}\lVert b- b^{o} \rVert_{\mathbb{L}^p}^p + C_{\operatorname{BDG},p} \lVert \sigma- \sigma^{o} \rVert_{\mathbb{H}^p}^{p}\right)\leq 
		C_1 \varepsilon^p.
	\end{aligned}
\end{align}

In a similar manner, it follows that
\[
\sup_{(b,\sigma)\in \mathcal{B}^\varepsilon} \lVert  S^{b,\sigma}- S^{o} \rVert_{\mathscr{S}^p}^p
\leq   2^p \sup_{(b,\sigma)\in \mathcal{B}^\varepsilon}\Big(T^{p-1}\lVert b-b^{o} \rVert_{\mathbb{L}^p}^p +C_{\operatorname{BDG},p}  \lVert \sigma-\sigma^o \rVert_{\mathbb{H}^p}^p\Big)
\leq  C_2\varepsilon^p.
\]

\vspace{0.5em}
\noindent Now let us prove (ii). Using Assumption \ref{as:posterior_refer}\;(i), the same arguments as presented for the proof of (i) can be used to show that 
\begin{align*}\quad\quad
		\sup_{H\in{\cal A}}\big\lVert X^{H;o} \big\rVert^p_{\mathscr{S}^p}
	\leq  3^p\left( |x_0|^p +K^p  \left(T^{p-1} \|b^o\|_{\mathbb{L}^p}^p +C_{\operatorname{BDG},p}\lVert \sigma^o \rVert_{\mathbb{H}^p}^{p}\right)  \right)
	<\infty.
\end{align*}
Hence, the claim that $C_{3} <\infty$ follows from the triangle inequality and \eqref{eq:resd_1}. 

Similarly, since $\sup_{(b,\sigma)\in\mathcal{B}^1}\lVert S^{b,\sigma}\rVert^p_{\mathscr{S}^p}<\infty$ (see Remark \ref{rem:p_char}) and $\lvert l({s})\rvert \leq C_l(1+\lvert {s} \rvert )$ for every~${s}\in \mathbb{R}^d$ (see Remark \ref{rem:objective}\;(ii)), the triangle inequality and the estimate given in (i) ensure that $ C_{4} <\infty$.  
\end{proof}

The following a priori estimate is based on the Galtchouk-Kunita-Watanabe decomposition (see, e.g., \cite{KunitaWatanabe67} or \cite[Theorem 4.27, p.126]{Jacod1979}).

\begin{lem}\label{lem:BSDE}
	For every $X\in L^2({\cal F}_T;\mathbb{R})$ and $R=(R_t)_{t\in[0,T]}\in \mathbb{L}^2(\mathbb{R})$, the BSDE given by
	\begin{align}\label{eq:BSDE_benchmark}
		Y_t = X +\int_t^TR_s ds - \int_t^T(Z_u)^\top dW_u -(L_T-L_t),\quad t\in [0,T],\quad \mbox{with $L_0=0$,}
	\end{align}
	has a unique solution $(Y,Z,L)\in\mathscr{S}^{2}(\mathbb{R}) \times \mathscr{H}^{2}(\mathbb{R}^d) \times \mathscr{M}^2(\mathbb{R})$. 
	In particular, $L$ is strongly orthogonal to $(\int_0^t(Z_u)^\top dW_u)_{t\in[0,T]}\in \mathscr{M}^2(\mathbb{R})$. 		
	Moreover, if we set $C_{\operatorname{ap}}:=\max\{2^8,(2^8+2^2)T\}$, then the solution satisfies the a priori estimate 
	\begin{align}\label{eq:BSDE_priori}
		\lVert Y \rVert_{\mathscr{S}^2}^2 + \lVert Z \rVert_{\mathscr{H}^2}^2 + \lVert L \rVert_{\mathscr{M}^2}^2 \leq {C}_{\operatorname{ap}} \big( \lVert X \rVert_{L^2}^2+ \|R \|_{\mathbb{L}^2}^2 \big).
	\end{align}
\end{lem} 
\begin{proof}
	Set $\widetilde Y_t:= \mathbb{E}[X+\int_0^TR_s ds\,|\,{\cal F}_t]$ for $t\in[0,T]$. An application of Doob's inequality together with the power triangle inequality (see \eqref{eq:power_triangle}) and Jensen's inequality shows~that
	\begin{align}\label{eq:priori1}
		\lVert \widetilde Y \rVert_{\mathscr{S}^2}^2 = \mathbb{E}\Bigg[\sup_{t\in[0,T]}  |\widetilde Y_t|^2\Bigg] 
		\leq 2^2 \mathbb{E}\Bigg[\bigg| X+\int_0^TR_s ds \bigg|^2\Bigg]\leq 2^4 \big( \|X \|_{L^2}^2 + T\cdot  \| R\|_{\mathbb{L}^2}^2\big)<\infty,
	\end{align}
	which ensures that $\widetilde Y=(\widetilde Y_t)_{t\in[0,T]}\in \mathscr{S}^2(\mathbb{R})$. Moreover, since $W^1,\dots,W^d\in\mathscr{M}^2(\mathbb{R})$, we can apply the Galtchouk-Kunita-Watanabe decomposition in \cite[Theorem 4.27, p.126]{Jacod1979} to represent the c\`adl\`ag process $\widetilde Y$ in terms of the orthogonal decomposition w.r.t.~$W=(W^1,\dots,W^d)^\top$, i.e.,
	\begin{align}\label{eq:cond_repre}
	\widetilde Y_t=\mathbb{E}\bigg[X +\int_0^TR_s ds \, \bigg|\, {\cal F}_t\bigg] =\widetilde Y_0+ \int_0^t (Z_s)^\top dW_s +L_t,\quad t\in[0,T],
	\end{align}
	where $Z=(Z_t)_{t\in[0,T]}\in \mathscr{H}^2(\mathbb{R}^d)$ and $L=(L_t)_{t\in[0,T]}\in \mathscr{M}^2(\mathbb{R})$ satisfies $L_0=0$ and is strongly orthogonal to $W^i$ for every $i=1,\dots,d$. This shows the existence of the solution $(\widetilde Y,Z,L)\in\mathscr{S}^{2}(\mathbb{R}) \times \mathscr{H}^{2}(\mathbb{R}^d) \times \mathscr{M}^2(\mathbb{R})$ satisfying 
	\begin{align*}
		\widetilde Y_t = X +\int_0^TR_s ds - \int_t^T(Z_u)^\top dW_u -(L_T-L_t),\quad t\in [0,T],\quad \mbox{with $L_0=0$}.
	\end{align*}
	
	Then, we claim that $(Y_t)_{t\in[0,T]}$, defined by $Y_t:=\widetilde{Y}_t-\int_0^tR_sds$ for $t\in[0,T]$, is in $\mathscr{S}^2(\mathbb{R})$. 
	Indeed, by the power triangle inequality and Jensen's inequality, 
	\begin{align}\label{eq:Y_Ytild}
			\|Y\|_{\mathscr{S}^2}^2\leq 2^2\Bigg(\|\widetilde Y\|_{\mathscr{S}^2}^2+\mathbb{E}\Bigg[\sup_{t\in[0,T]} \Big|\int_0^tR_sds \Big|^2 \Bigg] \Bigg)\leq 2^2\Big(\|\widetilde Y\|_{\mathscr{S}^2}^2+ T\cdot \|R\|_{\mathbb{L}^2}^2 \Big)<\infty,
	\end{align}
	where the last inequality holds because $\widetilde Y\in \mathscr{S}^2(\mathbb{R})$ (see \eqref{eq:priori1}) and $R\in \mathbb{L}^2(\mathbb{R})$.
	
	Therefore, we have a solution $(Y,Z,L)\in \mathscr{S}^2(\mathbb{R})\times \mathscr{H}^{2}(\mathbb{R}^d)\times \mathscr{M}^2(\mathbb{R})$ of \eqref{eq:BSDE_benchmark}.

	Furthermore, by the strong orthogonality between $L$ and $W^i$, $i=1,\dots,d$, and since $Z\in \mathscr{H}^2(\mathbb{R}^d)$ is $\mathbb{F}$-predictable,  an application of \cite[Lemma 2 \& Theorem 35, p.149]{Protter2005} shows that the two square integrable martingales $L$ and $(\int_0^t Z_s^\top dW_s)_{t\in[0,T]}$ are strongly orthogonal, i.e.,  $(L_t\int_0^t Z_u^\top dW_u)_{t\in[0,T]}$ is a uniformly integrable $(\mathbb{F},\mathbb{P})$-martingale. 
	Therefore, it follows from the It\^o-isometry (w.r.t. $|\widetilde Y_T-\widetilde Y_0|^2$; see \eqref{eq:cond_repre}) that
	\begin{align}\label{eq:priori2}
		\begin{aligned}
			\lVert Z \rVert_{\mathscr{H}^2}^2 + \lVert L \rVert_{\mathscr{M}^2}^2&=\mathbb{E}\Bigg[\int_0^T \lvert Z_s\rvert^2ds + [L,L]_T\Bigg]= \mathbb{E}\Big[|\widetilde Y_T-\widetilde Y_0|^2 \Big] \\
			&\leq 2^2 \mathbb{E}\Big[\max\{|\widetilde Y_T|^2,|\widetilde Y_0|^2\} \Big]\leq 
			   2^2   \lVert \widetilde Y \rVert_{\mathscr{S}^2}^2.
		\end{aligned}
	\end{align}
	Hence, the a priori estimate \eqref{eq:BSDE_priori} follows from combining \eqref{eq:priori1}, \eqref{eq:Y_Ytild} and \eqref{eq:priori2}. The uniqueness immediately follows from \eqref{eq:BSDE_priori}. Indeed, if there are two solutions to \eqref{eq:BSDE_benchmark}, then their difference solves \eqref{eq:BSDE_benchmark} with $X=0$ and $R=0$. This completes the proof. 
\end{proof}

\begin{pro}\label{pro:BSDE}
	Suppose that Assumptions \ref{as:posterior_refer}, \ref{dfn:control}, and \ref{as:objective} are satisfied. For every $H\in{\cal A}$, set
	\begin{align}\label{eq:terminal}
		\quad
		\begin{aligned}
		A^H&:=\partial_{x}f\big(X_T^{H;o},l(S_T^o)\big),\quad (R_t^H)_{t\in[0,T]}:=( \partial_xg(t,X_t^{H;o},H_t))_{t\in[0,T]},\\
		B^H&:= \partial_{y}f\big(X_T^{H;o},l(S_T^o)\big)\nabla l(S_T^o).
		\end{aligned}
	\end{align}
	Then the following hold. 
	\begin{itemize}
		\item [$(i)$] $A^H\in L^2({\cal F}_T;\mathbb{R})$, $R^H\in \mathbb{L}^2(\mathbb{R})$ and $B^H\in L^2({\cal F}_T;\mathbb{R}^d)$. 
		\item [$(ii)$] There exists a unique solution $(Y^H,Z^H,L^H)\in \mathscr{S}^2(\mathbb{R})\times \mathscr{H}^2(\mathbb{R}^d)\times \mathscr{M}^2(\mathbb{R})$ of 
		\begin{align}\label{eq:BSDE_univ1}
			\quad\quad Y_t^{H} = A^H+\int_t^T R^H_sds- \int_t^T (Z_s^H)^\top dW_s-(L_T^H-L_t^H),\quad t\in [0,T], 
		\end{align}
		with $L^H_0=0$, where $L^H$ and $(\int_0^t(Z_s^H)^\top dW_s)_{t\in[0,T]}\in \mathscr{M}^2(\mathbb{R})$ are strongly orthogonal.
		\item [$(iii)$]  There exists a unique solution  $({\cal Y}^H, {\cal Z}^H, {\cal L}^H)\in (\mathscr{S}^{2}(\mathbb{R}))^d \times (\mathscr{H}^{2}(\mathbb{R}^d))^d \times (\mathscr{M}^2(\mathbb{R}))^d$ of 
		\begin{align}\label{eq:BSDE_univ3}
			\quad {\cal Y}^H_t &= B^H - \int_t^T {{\cal Z}}^H_u dW_u-({{\cal L}}^H_T-{{\cal L}}^H_t)
		,\quad t\in [0,T],
		\end{align}
	with ${{\cal L}}^H_0=0$, where ${{\cal L}}^{H}$ and $(\int_0^t \mathcal{Z}_s^{H} dW_s)_{t\in[0,T]}$ are strongly orthogonal.
	\end{itemize}
\end{pro}
\begin{proof} 
	We start by proving (i). 
	Let $0<  r <  \min\{\frac{p-2}{2}, p-3\}$ be as in Assumption \ref{as:objective}\;(ii). 
	It follows from Remark \ref{rem:objective}\;(i) and the power triangle inequality that 
	\begin{align*}
			\lvert \partial_{x} f(\omega,x,y)\rvert^2+\lvert \partial_{x} g(\omega,t,x,h)\rvert^2 \leq (4\widetilde{C})^2\left(1+ \lvert x \rvert^{2(r+1)}+ \lvert h \rvert^{2(r+1)}+ \lvert y \rvert^{2(r+1)}  \right)
	\end{align*}
	for every $(\omega,t,x,y,h) \in \Omega \times [0,T]\times \mathbb{R}\times \mathbb{R}\times \mathbb{R}^d$. Therefore, by Lemma \ref{lem:p_integrability}\;(ii) and H\"older's inequality (with exponent $\frac{p}{2(r+1)}>1$), we obtain that 
	\begin{align*}
		\begin{aligned}
			\sup_{H\in {\cal A}}
			\big\|A^H\big\|_{L^2}^{2}
			&\leq (4\widetilde{C})^2\sup_{H\in {\cal A}} \mathbb{E}\bigg[1 +|X_T^{H;o}|^{2(r+1)}+  | l(S_T^o)|^{2(r+1)}\bigg]\\
			&\leq (4\widetilde{C})^2 \Big(1 +\sup_{H\in {\cal A}}\lVert X^{H;o} \rVert_{\mathscr{S}^p}^{{2(r+1)}}+ \big\|l(S^o)\big\|_{\mathscr{S}^p}^{{2(r+1)}}\Big)<\infty.
		\end{aligned}
	\end{align*}
	Similarly, since $\lvert H_t \rvert \leq K$ $\mathbb{P}\otimes dt$-a.e., 
	\begin{align*}
		\begin{aligned}
		\sup_{H\in {\cal A}} 
		\big\|R^H\big\|_{\mathbb{L}^2}^2
		&\leq (4\widetilde{C})^2\sup_{H\in {\cal A}} \mathbb{E}\bigg[\int_0^T\Big(1 +|X_t^{H;o}|^{2(r+1)}+  | H_t|^{2(r+1)}\Big)dt\bigg]\\
		&\leq (4\widetilde{C})^2 T \Big(1 +\sup_{H\in {\cal A}}\lVert X^{H;o} \rVert_{\mathscr{S}^p}^{{2(r+1)}}+ K^{2(r+1)}\Big)<\infty.
		\end{aligned}
	\end{align*}

	Next, since $\lvert \nabla_s l (\cdot)\rvert \leq C_l$ (see Remark \ref{rem:objective}\;(ii)), the same arguments devoted for showing that $\sup_{H\in {\cal A}} \| A^H\|_{L^2}<\infty$ ensure that $\sup_{H\in {\cal A}} \| B^H\|_{L^2}<\infty$, as claimed.
	
	The statements (ii) and (iii) follow directly from Lemma \ref{lem:BSDE}.
\end{proof}

\subsection{Stability results}\label{sec:stability}
This section is devoted to showing stability results of the forward process $X^{H;o}$ and the backward triplets $(Y^H,Z^H,L^H)$ and $({\mathcal{Y}}^H,{\mathcal{Z}}^H,{\mathcal{L}}^H)$ (introduced in Proposition \ref{pro:BSDE}\;(ii)) with respect to $H\in{\cal A}$, which will play an essential role in the proof of Theorem~\ref{thm:main}. 
We begin with  a priori estimates of the forward process.

\begin{lem}\label{lem:stability0}
	Suppose that Assumption \ref{as:posterior_refer} is satisfied, let $\beta \in(1,p)$, and set $\alpha(\beta)= \frac{p\beta}{p-\beta}>1$. Then, for any $G,H\in {\cal A}$,
	\[
	\quad	\big\lVert X^{G;o}- X^{H;o}\big\rVert^\beta_{\mathscr{S}^\beta}\leq  C({\beta})\Bigg(\lVert G-H \rVert_{\mathbb{L}^{\alpha(\beta)}}^\beta\lVert b^o \rVert_{\mathbb{L}^p}^{\beta} + \mathbb{E}\Bigg[\bigg(\int_0^T \big \lvert (\sigma_t^o)^\top (G_t-H_t)\big \rvert^{2} dt \bigg)^{\frac{\beta}{2}}\Bigg] \Bigg),
	\]
	where  $C({\beta}):= 2^\beta \max\{C_{\operatorname{BDG},\beta}, T^{\beta -1}\}>0$. 
\end{lem}
\begin{proof}
	Set $\Delta:=G-H$ and write $\alpha=\alpha(\beta)$. 
	First note that by the power triangle inequality and the BDG inequality,
	\begin{align*}
		\begin{aligned}
			\big\lVert X^{G;o}- X^{H;o}\big\rVert^\beta_{\mathscr{S}^\beta}&\leq  2^\beta\Bigg( \mathbb{E}\Bigg[\sup_{t\in[0,T]}\bigg| \int_0^t \Delta_s^\top b_s^o ds \bigg|^\beta \Bigg]+C_{\operatorname{BDG},\beta}\mathbb{E}\Bigg[\bigg(\int_0^T \big \lvert (\sigma_t^o)^\top \Delta_t\big \rvert^{2} dt\bigg)^{\frac{\beta}{2}}\Bigg]\Bigg)\\
			&\leq  2^\beta\Bigg( \mathbb{E}\Bigg[\bigg( \int_0^T |\Delta_t^\top b_t^o| dt \bigg)^\beta \Bigg]+C_{\operatorname{BDG},\beta}\mathbb{E}\Bigg[\bigg(\int_0^T \big \lvert (\sigma_t^o)^\top \Delta_t\big \rvert^{2} dt\bigg)^{\frac{\beta}{2}}\Bigg]\Bigg).
		\end{aligned}
	\end{align*}
	Moreover, a twofold application of H\"older's inequality (first with  exponent $p>3$  followed by exponent  $\frac{p}{\beta}>1$) shows that 
	\[
	\mathbb{E}\Bigg[\bigg( \int_0^T |\Delta_t^\top b_t^o| dt \bigg)^\beta \Bigg]
	\leq \mathbb{E}\Bigg[\bigg(\int_0^T \lvert \Delta_t\rvert^{q}dt \bigg)^{^{\frac{\beta}{q}}}\bigg(\int_0^T \lvert b_t^o \rvert^p dt \bigg)^{\frac{\beta}{p}}\Bigg] \leq \mathbb{E}\Bigg[\bigg(\int_0^T \lvert  \Delta_t\rvert^{q}dt \bigg)^{\frac{\alpha}{q}} \Bigg]^{\frac{\beta}{\alpha}} \lVert b^o \rVert_{\mathbb{L}^p}^{\beta}.
	\] 
	Finally, since $x\rightarrow |x|^{\frac{\alpha}{q}}$ is convex (noting that $\frac{\alpha}{q}=\frac{\beta p}{q(p-\beta)}>1$), Jensen's inequality ensures that $\mathbb{E}[(\int_0^T \lvert  \Delta_t\rvert^{q}dt )^{\frac{\alpha}{q}} ]^{\frac{\beta}{\alpha}} \leq  T^{\beta-1} \lVert \Delta \rVert_{\mathbb{L}^{\alpha}}^\beta$, as claimed.
\end{proof}


\begin{lem}\label{lem:stability}
	Suppose that Assumptions \ref{as:posterior_refer} and \ref{dfn:control} are satisfied and let 
		 $(H^n)_{n\in \mathbb{N}}\subseteq {\cal A}$ and $H^\star \in {\cal A}$  such that $| X_T^{H^n;o} - X_T^{H^\star;o} |\xrightarrow[]{\mathbb{P}} 0$ as $n\rightarrow \infty$.
	Then, for every $\beta \geq 1$,
	\begin{align*}
		\lVert H^n-H^\star\lVert_{\mathbb{L}^\beta} \rightarrow 0\quad \mbox{as $n\rightarrow 0$}.
	\end{align*}
	Moreover, $\int_0^T\big\lvert (\sigma_t^o)^\top (H_t^{n}-H_t^\star)\big\rvert^2 dt \xrightarrow[]{\mathbb{P}} 0$ as $n\to\infty$.
\end{lem}
\begin{proof}
	Recall that the measure $\mathbb{Q}\sim \mathbb{P}$ is defined in \eqref{eq:mtg_measure} and that ${W}^{\mathbb{Q}}= {W}+\int_0^\cdot (\sigma_s^o)^{-1}{b}_s^o  ds$  is a $(\mathbb{F},\mathbb{Q})$-Brownian motion.
	In particular, $S^o$ is an $(\mathbb{F},\mathbb{Q})$-local martingale since $S^o=S^o_0+\int_0^\cdot \sigma_s^o d{W}_s^{\mathbb{Q}}$.
	 Furthermore, for each $n\in\mathbb{N}$,  $((H^n \cdot S^o)_t)_{t\in[0,T]}$ is an $(\mathbb{F},\mathbb{Q})$-local martingale since
	\begin{align*}
		(H^n\cdot S^o)_t=\int_0^t (H_s^n)^\top \sigma_r^o d{W}_r^{\mathbb{Q}},\quad t\in[0,T].
	\end{align*}
	
	Next note that by the  assumption on $(H^n)_{n\in\mathbb{N}}$,
	\begin{eqnarray}\label{eq:conv1_P}
		\big\lvert X_T^{H^{n};o}-X_T^{H^\star;o}\big\rvert^2 \xrightarrow[]{\mathbb{Q}} 0\quad \mbox{as $n\rightarrow \infty$}.
	\end{eqnarray}
	Moreover, as $\lvert H_t^\star \rvert, \lvert H_t^{n}\rvert  \leq K$ $\mathbb{Q}\otimes dt$-a.e.\;(by Assumption \ref{dfn:control} and the equivalence $\mathbb{Q}\sim \mathbb{P}$), the power triangle inequality, and H\"older's inequality  (with exponent $\frac{p}{2v}>1$), it follows for every $1<v<\frac{p}{2}$ that 
	\begin{align*}
		\begin{aligned}
			\sup_{n\in\mathbb{N}}\mathbb{E}^{\mathbb{Q}} \Big[ \big\lvert X_T^{H^{n};o}-X_T^{H^\star;o}\big\rvert^{2v} \Big]
			&=\sup_{n\in\mathbb{N}}\mathbb{E} \Big[{\cal D}_T \big\lvert X_T^{H^{n};o}-X_T^{H^\star;o}\big\rvert^{2v} \Big]\\
			&\leq 2^{2v}\Big(\sup_{n\in\mathbb{N}}\mathbb{E}\Big[{\cal D}_T \big\lvert X_T^{H^{n};o}\big\rvert^{2v}\Big]+\mathbb{E}\Big[{\cal D}_T \big\lvert X_T^{H^{\star};o}\big\rvert^{2v}\Big] \Big)
			\\
			&\leq 2^{2v} \lVert {\cal D}_T\rVert_{L^{\frac{p}{p-2v}}}\Big(\sup_{n\in\mathbb{N}} \big\lVert X_T^{H^{n};o}\big\rVert_{L^p}^{2v}+\big\lVert X_T^{H^{\star};o}\big\rVert_{L^p}^{2v} \Big).
		\end{aligned}
	\end{align*}
	Hence, using the $L^\beta$-boundedness of the exponential martingale  ${\cal D}_T$ (see Assumption \ref{as:posterior_refer}\;(ii)), and Lemma \ref{lem:p_integrability}\;(ii) with the fact that $(H^{n})_{n\in \mathbb{N}}\subseteq {\cal A}$ and $H^\star\in{\cal A}$, it follows that  
	\[
	\sup_{n\in\mathbb{N}}\mathbb{E}^{\mathbb{Q}} \Big[ \big\lvert X_T^{H^{n};o}-X_T^{H^\star;o}\big\rvert^{2v}  \Big]<\infty.
	\]
	In particular, since $v>1$, the de la Vall\'ee Poussin theorem \cite[Theorem 6.19]{Klenke14} ensures the uniform integrability of $(\lvert X_T^{H^{n};o}-X_T^{H^\star;o}\rvert^2 )_{n\in \mathbb{N}}$ with respect to $\mathbb{Q}$; thus $\mathbb{E}^{\mathbb{Q}} [ \big\lvert X_T^{H^{n;o}}-X_T^{H^\star;o}\big\rvert^2]\to~0$ as $n\to\infty$ by  \eqref{eq:conv1_P} and Vitali's convergence theorem \cite[Theorem 6.25]{Klenke14}.
	Using that $S_\cdot^o =S_0^o+ \int_0^\cdot \sigma_s^o d{W}_s^{\mathbb{Q}}$ under $\mathbb{Q}$, the It\^o-isometry implies that 
	\begin{align}\label{eq:conv_Q}
		0=\lim_{n\rightarrow \infty}\mathbb{E}^{\mathbb{Q}} \Big[ \big\lvert X_T^{H^{n;o}}-X_T^{H^\star;o}\big\rvert^2 \Big]= \lim_{n\rightarrow \infty}\mathbb{E}^\mathbb{Q}\bigg[\int_0^T\big\lvert (\sigma_t^o)^\top (H_t^{n}-H_t^\star)\big\rvert^2 dt \bigg]
	\end{align}
	and in particular that  $\int_0^T\big\lvert (\sigma_t^o)^\top (H_t^{n}-H_t^\star)\big\rvert^2 dt \xrightarrow[]{\mathbb{P}} 0$ as $n\to\infty$ since $\mathbb{Q}\sim\mathbb{P}$.
	
	Moreover, combining \eqref{eq:conv_Q} with the fact that $\sigma^o$ is invertible (see\;Assumption\;\ref{as:posterior_refer}\;(ii))  shows that
	\begin{align}\label{eq:conv_pdt}
		H_t^{n}\xrightarrow[]{\mathbb{P}\otimes dt } H_t^\star\quad \mbox{as $n\rightarrow \infty$}.
	\end{align}
	Finally, since $\lvert H_t^\star \rvert, \lvert H_t^{n}\rvert  \leq K$ $\mathbb{P}\otimes dt$-a.e., the dominated convergence theorem guarantees that \eqref{eq:conv_pdt} implies that for every $\beta\geq 1$,  $\lVert H^n-H^\star\lVert_{\mathbb{L}^\beta} \rightarrow 0$ as $n\to\infty$.
	This completes the proof.
\end{proof}

\begin{lem}\label{lem:stability_1}
	Suppose that Assumptions  \ref{as:posterior_refer}, \ref{dfn:control}, and \ref{as:objective} are satisfied and let  $(H^n)_{n\in \mathbb{N}}\subseteq {\cal A}$ and $H^\star \in {\cal A}$  such that $| X_T^{H^n;o} - X_T^{H^\star;o} |\xrightarrow[]{\mathbb{P}} 0$ as $n\rightarrow \infty$.
	\begin{enumerate}
	\item[(i)] Denote by  $(Y^{n},Z^{n},L^{n})$ and $(Y^\star,Z^\star, L^\star)$ the unique solutions of \eqref{eq:BSDE_univ1} (ensured by Proposition \ref{pro:BSDE}) under the terminal conditions and generators
	\begin{align*}
		\qquad\qquad
	A^{H^n}&=\partial_{x}f\big(X_T^{H^n;o},l(S_T^o)\big),\quad\quad  (R_t^{H^n})_{t\in[0,T]}=(\partial_xg(t,X_t^{H^n;o},H^n_t))_{t\in[0,T]}, \\ 
	A^{H^\star}&=\partial_{x}f\big(X_T^{H^\star ;o},l(S_T^o)\big),\quad\quad\, (R^{H^\star}_t)_{t\in[0,T]}=(\partial_xg(t,X_t^{H^\star;o},H^\star_t))_{t\in[0,T]},
	\end{align*}
	respectively.
	Then, as $n\to\infty$,
	\begin{align}\label{eq:estimate21}
		\lVert Y^n - Y^\star \rVert_{\mathscr{S}^2} + \lVert Z^n - Z^\star \rVert_{\mathscr{H}^2}+\lVert L^n - L^\star \rVert_{\mathscr{M}^2} \rightarrow 0.     	
	\end{align}

	\item[(ii)]  Denote by $({\cal Y}^{n}, {\cal Z}^{n} , {\cal L}^{n})$ and $({\cal Y}^{\star}, {\cal Z}^{\star}, {\cal L}^{\star})$  the unique solutions of \eqref{eq:BSDE_univ3} under the terminal conditions 
	\begin{align*}
		\quad \quad B^{H^n}=\partial_{y}f(X_T^{H^n;o},l(S_T^o))\nabla l_s(S_T^o),\quad  B^{H^\star}=\partial_{y}f(X_T^{H^\star;o},l(S_T^o))\nabla_s l(S_T^o),
	\end{align*}
	respectively.
	Then, for every $i=1,\dots,d$, as $n\to\infty$,
	\begin{align}\label{eq:estimate22}
		\lVert {\cal Y}^{n,i} -{\cal Y}^{\star,i} \rVert_{\mathscr{S}^2} + \lVert {\cal Z}^{n,i} -{\cal Z}^{\star,i}  \rVert_{\mathscr{H}^2}+\lVert {\cal L}^{n,i} -{\cal L}^{\star,i} \rVert_{\mathscr{M}^2} \rightarrow 0.
	\end{align}
	\end{enumerate}
\end{lem}
\begin{proof}
	We start by proving  (i).
	 It follows from Proposition \ref{pro:BSDE} that $(Y^{n}-Y^\star,Z^{n}-Z^\star,L^{n}-L^\star)$ 
	is the unique solution to the BSDE \eqref{eq:BSDE_benchmark} with the terminal condition  $A^{H^n}- A^{H^\star}$ and generator $R^{H^n}-R^{H^\star}$.
	Therefore, using the a priori estimate given in \eqref{eq:BSDE_priori},
	\[\lVert Y^n - Y^\star \rVert_{\mathscr{S}^2}^2 + \lVert Z^n - Z^\star \rVert_{\mathscr{H}^2}^2+\lVert L^n - L^\star \rVert_{\mathscr{M}^2} ^2
	\leq {C}_{\operatorname{ap}} \big( \lVert A^{H^n}- A^{H^\star} \rVert_{L^2}^2+ \|R^{H^n}-R^{H^\star}\|_{\mathbb{L}^2}^2 \big).
	\]
	Thus, all that is left to do is to prove that the two terms $\lVert A^{H^n}- A^{H^\star} \rVert_{L^2}^2$ and $\|R^{H^n}-R^{H^\star}\|_{\mathbb{L}^2}^2$ vanish as $n\to \infty$.
	
	\vspace{0.5em}
	\noindent {\it Step 1a. limit of $\lVert A^{H^n}- A^{H^\star} \rVert_{L^2}^2$ }: By Assumption \ref{as:objective}\;(ii) we have that  $\lvert \partial_{xx}f(\omega,x,y) \rvert \leq \overline C_2(1+\lvert x\rvert^r+\lvert y\rvert^r)$ for every $\omega\in\Omega$ and $(x,y)\in \mathbb{R}^2$. Hence an application of the fundamental theorem of calculus (and the power triangle inequality) shows that  for every $\omega\in\Omega$ and $x,\hat x,y\in \mathbb{R}$, 
	\begin{align}\label{eq:FTC}
		\begin{aligned}
			\big| \partial_{x}f(\omega,x,y)-\partial_{x}f(\omega,\hat x,y)\big|^2
			&\leq (\overline C_2)^2 2^{r+4}(1 + \rvert x\rvert^{2r} +  \lvert \hat x\rvert^{2r}  +\lvert y \rvert^{2r}  )\big\lvert x-\hat x\big\rvert^2.
		\end{aligned}
	\end{align}
	In particular, by the definition of $A^{H^n}$ and $A^{H^\star}$,
	\begin{align*}
	\lVert A^{H^n}- A^{H^\star} \rVert_{L^2}^2
	\leq (\overline C_2)^2 2^{r+4}\mathbb{E}\bigg[ \Big(1 + \rvert X_T^{H^{n};o}\rvert^{2r} +  \lvert X_T^{H^{\star};o}\rvert^{2r}  +\lvert l(S_T^o) \rvert^{2r}  \Big)\big\lvert X_T^{H^{n};o}-X_T^{H^{\star};o}\big\rvert^2  \bigg].
	\end{align*}
	
	In order to estimate the last term, set   $\widetilde{p}:=\frac{2p}{p-2r}$, where we recall that  $p>3$ and $0<r<\min\{\frac{p-2}{2},p-3\}$ (see Assumptions \ref{as:posterior_refer} and \ref{as:objective}\;(ii)), and note that $1< \widetilde{p}< p$.
	It follows from H\"older's inequality (with exponent $\frac{p}{2r}>1$ and conjugate exponent $\frac{\widetilde{p}}{2}>1$) and the power triangle inequality that
	\begin{align}\label{eq:FTC_1}
		\begin{aligned}
	&\lVert A^{H^n}- A^{H^\star} \rVert_{L^2}^2\\
	&\quad \leq (\overline C_2)^2  2^{r+4} 4^\frac{p}{2r} \Big(1+\lVert X_T^{H^{n};o} \rVert_{L^p}^{2r}+\lVert X_T^{H^*;o} \rVert_{L^p}^{2r}+\lVert l(S_T^o) \rVert_{L^p}^{2r}\Big) \big\lVert X_T^{H^{n};o}-X_T^{H^{\star};o}\big\rVert_{L^{\widetilde{p}}}^{2}.
		\end{aligned}
	\end{align}
	Since it holds for every $S\in \mathscr{S}^p(\mathbb{R}^d)$ that $\|S\|_{{L}^p}\leq \|S \|_{\mathscr{S}^p}$, by Lemma \ref{lem:p_integrability}\;(ii),
	\begin{align}\label{eq:estimate_integ}
		\sup_{n\in\mathbb{N}}\Big(1+\lVert X_T^{H^{n};o} \rVert_{L^p}^{2r}+\lVert X_T^{H^*;o} \rVert_{L^p}^{2r}+\lVert l(S_T^o) \rVert_{L^p}^{2r}\Big)<\infty,
	\end{align}
 	thus, it remains to show that $\| X_T^{H^{n};o}-X_T^{H^{\star};o} \|_{L^{\widetilde{p}}}\to 0$ as $n\rightarrow \infty$.
 
 	\vspace{0.5em}
 	To that end, we claim that 
 	\begin{align}\label{eq:sup_conv}
 		\big\lVert X^{H^{n};o}-X^{H^{\star};o}\big\rVert_{\mathscr{S}^{\widetilde{p}}}\to 0\quad \mbox{as $n\to \infty$}.
 	\end{align}
 	Indeed, since $1<\widetilde{p}<p$, an application of Lemma \ref{lem:stability0} (with $\beta =\widetilde{p}\in(1,p)$ and ${\alpha}={\alpha}(\widetilde{p})=\frac{\widetilde{p}p}{p-\widetilde{p}}>1$) ensures that
 	\begin{align*}\quad
 		\begin{aligned}
 			\big\lVert X^{H^{n};o}-X^{H^{\star};o}\big\rVert_{\mathscr{S}^{\widetilde{p}}}^{\widetilde{p}}
 			\leq C(\widetilde{p}) \Bigg( \lVert H^n-H^\star \rVert_{\mathbb{L}^{\alpha}}^{\widetilde{p}}\lVert b_t^o \rVert_{\mathbb{L}^p}^{\widetilde{p}}+ \mathbb{E}\Bigg[\bigg(\int_0^T \big \lvert (\sigma_t^o)^\top ({H}_t^n-H_t^\star)\big \rvert^{2} dt \bigg)^{\frac{\widetilde{p}}{2}}\Bigg] \Bigg),
 		\end{aligned}
 	\end{align*}
 	where $C(\widetilde{p})=2^{\widetilde{p}} \max\{C_{\operatorname{BDG},\beta},T^{\widetilde{p}-1}\}$.
 	Moreover, since $b^o\in \mathbb{L}^p(\mathbb{R}^d)$, it follows from Lemma \ref{lem:stability}  that as $n\to\infty$,
 	\[
 	\lVert H^n-H^\star \rVert_{\mathbb{L}^{\alpha}}^{\widetilde{p}}\cdot \lVert b^o \rVert_{\mathbb{L}^p}^{\widetilde{p}} \rightarrow 0.
 	\]
 	
 	Next we note that  $|H_t^\star- H_t^{n} |  \leq 2K$ $\mathbb{P}\otimes dt$-a.e. and $\sigma^o\in \mathbb{H}^p(\mathbb{R}^{d\times d})$. Therefore, since $\tilde{p}< p$ and $\int_0^T\big\lvert (\sigma_t^o)^\top (H_t^{n}-H_t^\star)\big\rvert^2 dt  \xrightarrow[]{\mathbb{P}} 0$ as $n\to\infty$ (by Lemma \ref{lem:stability}), the dominated convergence theorem shows that as $n\to\infty$,
 	\[
 	\mathbb{E}\Bigg[\bigg( \int_0^T\big\lvert (\sigma_t^o)^\top (H_t^{n}-H_t^\star)\big\rvert^2 dt \bigg)^{\frac{\widetilde{p}}{2}} \Bigg]\to 0.
 	\] 
 	We conclude that indeed $
 	\lim_{n\to \infty}\|X_T^{H^{n};o}-X_T^{H^{\star};o}\|_{L^{\widetilde{p}}} \leq \lim_{n\to \infty}\lVert X^{H^{n};o}-X^{H^{\star};o}\rVert_{\mathscr{S}^{\widetilde{p}}}=0.$ 

	\vspace{0.5em}
	\noindent {\it Step 1b. limit of $\|R^{H^n}-R^{H^\star}\|_{\mathbb{L}^2}^2$} : By Assumption \ref{as:objective}\;(ii), we have that  $\lvert \partial_{xh}g(\omega,t,x,h) \rvert +\lvert \partial_{xx}g(\omega,t,x,h) \rvert\leq  2 \overline C_2(1+\lvert (x,h)^\top\rvert^r)$ for every $(\omega,t,x,h)\in \Omega \times [0,T]\times \mathbb{R}\times \mathbb{R}^d$. Using the same arguments presented for \eqref{eq:FTC}, we have that for every $(\omega,t,x,h)\in \Omega \times [0,T]\times \mathbb{R}\times \mathbb{R}^d$ and $(\hat x,\hat h)\in \mathbb{R}\times \mathbb{R}^d$, 
	\begin{align*}
		\begin{aligned}
			&\big|\partial_xg(\omega,t,x,h)-\partial_xg(\omega,t,\hat x,\hat h)\big|^2\\
			&\quad  \leq 2^2\Big( \big|\partial_xg(\omega,t,x,h)-\partial_xg(\omega,t,x,\hat h)\big|^2+\big|\partial_xg(\omega,t,x,\hat h)-\partial_xg(\omega,t,\hat x,\hat h)\big|^2 \Big) \\
			&\quad \leq (\overline C_2)^2 2^{r+8}\big(1 + \rvert x\rvert^{2r} +  \lvert \hat x\rvert^{2r}  +\lvert h \rvert^{2r}  +\lvert \hat h \rvert^{2r}  \big) \big(\lvert x-\hat x\rvert^2 + \lvert h-\hat h\rvert^2 \big).
		\end{aligned}
	\end{align*}
	In particular, by the definition of $R^{H^n}$ and $R^{H^\star}$, 
	\begin{align*}
		\begin{aligned}
			\|R^{H^n}-R^{H^\star}\|_{\mathbb{L}^2}^2
			&\leq (\overline C_2)^2 2^{r+8} \mathbb{E}\Bigg[\int_0^T\bigg\{\Big(1 + \rvert X_t^{H^{n};o}\rvert^{2r} +  \lvert X_t^{H^{\star};o}\rvert^{2r}  +\lvert H_t^n \rvert^{2r} +\lvert H_t^* \rvert^{2r} \Big)\bigg. \\
			&\hspace{10.em} \times\Big(\big\lvert X_t^{H^{n};o}-X_t^{H^{\star};o}\big\rvert^2+ \big\lvert H_t^{n}-H_t^{\star}\big\rvert^2 \Big)\bigg\} dt \Bigg].
		\end{aligned}
	\end{align*}

	From H\"older's inequality (with exponent $\frac{p}{2r}>1$ and conjugate exponent $\frac{\widetilde{p}}{2}>1$) and the power triangle inequality (with $\widetilde p$ appearing in \eqref{eq:FTC_1}), it follows that
	\begin{align}\label{eq:FTC_2}
		\begin{aligned}
			\|R^{H^n}-R^{H^\star}\|_{\mathbb{L}^2}^2
			& \leq (\overline C_2)^2  2^{r+8} 5^\frac{p}{2r} \Big(1+\lVert X^{H^{n};o} \rVert_{\mathbb{L}^p}^{2r}+\lVert X^{H^*;o} \rVert_{\mathbb{L}^p}^{2r}+\lVert H^n \rVert_{\mathbb{L}^p}^{2r}+\lVert H^\star \rVert_{\mathbb{L}^p}^{2r}\Big) \\
			&\hspace{7.5em} \times \Big(\big\lVert X^{H^{n};o}-X^{H^{\star};o}\big\rVert_{\mathbb{L}^{\widetilde{p}}}^{2}+\big\lVert H^{n}-H^{\star}\big\rVert_{\mathbb{L}^{\widetilde{p}}}^{2} \Big).
		\end{aligned}
	\end{align}
	Since it holds for every $S\in \mathscr{S}^p(\mathbb{R}^d)$ that $\|S\|_{\mathbb{L}^p}\leq T^{1/p} \|S \|_{\mathscr{S}^p}$, by Lemma \ref{lem:p_integrability}\;(ii) and the fact that $\lvert H_t \rvert \leq K$ $\mathbb{P}\otimes dt$-a.e.\ for every $H\in {\cal A}$ (see Assumption \ref{dfn:control}),
	\begin{align*}
		\sup_{n\in\mathbb{N}}\Big(1+\lVert X^{H^{n};o} \rVert_{\mathbb{L}^p}^{2r}+\lVert X^{H^*;o} \rVert_{\mathbb{L}^p}^{2r}+\lVert H^n \rVert_{\mathbb{L}^p}^{2r}+\lVert H^\star \rVert_{\mathbb{L}^p}^{2r}\Big)<\infty.
	\end{align*}
	
	Moreover, since $\lim_{n\to \infty}\lVert X^{H^{n};o}-X^{H^{\star};o}\rVert_{\mathscr{S}^{\widetilde{p}}}=0$ (see \eqref{eq:sup_conv}) and it holds for every $n\in\mathbb{N}$ that $\lVert X^{H^{n};o}-X^{H^{\star};o}\rVert_{\mathbb{L}^{\widetilde{p}}}\leq T^{1/p} \lVert X^{H^{n};o}-X^{H^{\star};o}\rVert_{\mathscr{S}^{\widetilde{p}}}$,  the term $\lVert X^{H^{n};o}-X^{H^{\star};o}\rVert_{\mathbb{L}^{\widetilde{p}}}^{2}$ appearing in \eqref{eq:FTC_2} vanishes as $n\to \infty$. Moreover, by Lemma \ref{lem:stability}, the other term $\lVert H^n-H^\star\lVert_{\mathbb{L}^{\widetilde p}}$ therein vanishes as well.  Hence, we conclude that $\lim_{n\to \infty}\|R^{H^n}-R^{H^\star}\|_{\mathbb{L}^2}=0$.
	

	\vspace{0.5em}
	\noindent
	The proof of part (ii) follows from the same arguments as those used in the proof of (i), and we only sketch it. 
    Note that by Lemma \ref{lem:BSDE}, for every $i=1,\dots,d$,
	\[
	\lVert \mathcal{Y}^{n,i} -  \mathcal{Y}^{\star,i} \rVert_{\mathscr{S}^2} + \lVert  \mathcal{Z}^{n,i} -  \mathcal{Z}^{\star,i} \rVert_{\mathscr{H}^2}+\lVert  \mathcal{L}^{n,i} -  \mathcal{L}^{\star,i} \rVert_{\mathscr{M}^2} 
	\leq C_{\operatorname{ap}}\big\|B^{H^n,i}- B^{H^\star,i} \big\|_{L^2}^2 
	\]
	and  that for every $\omega\in \Omega$ and $x,x^*,y\in \mathbb{R}$, 
	\[
	\big|\partial_{y}f(\omega,x,y)-\partial_{y}f(\omega,x^*,y)\big|^2\leq (\overline C_2)^2 2^{r+4}(1 + \rvert x\rvert^{2r} +  \lvert x^*\rvert^{2r}  +\lvert y \rvert^{2r}  )\big\lvert x-x^*\big\rvert^2.
	\]
	Moreover, since $\lvert \partial_{s_i} h({\cdot}) \rvert \leq C_h$ for every $i=1,\dots,d$, H\"older's inequality shows that
	\begin{align*}
		&
		\big\|B^{H^n,i}- B^{H^\star,i} \big\|_{L^2}^2\\
			&\quad \leq (C_l \overline C_2)^2 2^{r+4} 4^\frac{p}{2r} \Big(1+\lVert X_T^{H^{n};o} \rVert_{L^p}^{2r}+\lVert X_T^{H^*;o} \rVert_{L^p}^{2r}+\lVert l(S_T^o) \rVert_{L^p}^{2r}\Big) \big\lVert X_T^{H^{n};o}-X_T^{H^{\star};o}\big\rVert_{L^{\widetilde{p}}}^{2},
		\end{align*}
	 where $\widetilde{p}=\frac{2p}{p-2r}>1$.
	The claim follows from \eqref{eq:estimate_integ} and since $\lim_{n\rightarrow \infty}\| X_T^{H^{n};o}-X_T^{H^{\star};o} \|_{L^{\widetilde{p}}}= 0$, as was shown in \eqref{eq:sup_conv}. This completes the proof.
\end{proof}

\subsection{Proof of Proposition \ref{pro:main0} \& first order optimality}\label{sec:proof:pro:main0}
\begin{proof}[Proof of Proposition \ref{pro:main0}]  
   We start by proving the statement (i).
   Let $(H^n)_{n\in \mathbb{N}} \subseteq {\cal A}$ be a sequence such that 
	\[
		V(0)=\lim_{n\rightarrow  \infty} \mathbb{E}\bigg[\int_0^Tg\big(t,X_t^{H^n;o},H_t^n\big)dt+ f \big(X_T^{H^n;o},l(S_T^o)\big)\bigg].
	\]
	
	Note that $\mathbb{L}^2(\mathbb{R}^d)$ defined in Section \ref{sec:preliminary} is a reflexive Banach space. Furthermore, since the sequence $(H^n)_{n\in \mathbb{N}} \subseteq {\cal A}$ is $\mathbb{F}$-predictable and ${\cal K}$-valued $\mathbb{P}\otimes dt$-a.e. where the correspondence ${\cal K}$ has uniformly bounded values in $\mathbb{R}^d$ (by Assumption \ref{dfn:control}), the sequence is  bounded  in $\mathbb{L}^2(\mathbb{R}^d)$. Hence, \cite[Theorem 15.1.2,\;p.320]{DelbaenSchacher06} asserts that there are $H^*\in \mathbb{L}^2(\mathbb{R}^d)$ and
	\begin{align}\label{eq:Komlos0}
		\widetilde{H}^n \in \operatorname{conv}(H^{n},H^{n+1},\dots),\;\; \mbox{$n\in\mathbb{N}$,}
	\end{align}
	which satisfy
	\begin{align}\label{eq:Komlos}
		\lVert \widetilde{H}^n-H^*\rVert_{\mathbb{L}^2}\rightarrow 0\quad \mbox{as $n\rightarrow \infty$}.
	\end{align}
	Note that as $(\widetilde{H}^{n})_{n\in \mathbb{N}}$ is ${\cal K}$-valued $\mathbb{P}\otimes dt$-a.e.\;by  \eqref{eq:Komlos0}, $H^*$ is also ${\cal K}$-valued $\mathbb{P}\otimes dt$-a.e.\;as well, thus $H^*\in {\cal A}$. Moreover, from \eqref{eq:Komlos}  it follows that  
	\begin{align}\label{eq:Komlos_1}
		\widetilde{H}^n\to H^*\quad \mbox{$\mathbb{P}\otimes dt$-a.e.}\quad\mbox{as $n\to \infty$}.
	\end{align}
	
	It remains to show that $H^*$ is an optimizer. 
	Since $S^o$ is an It\^o $(\mathbb{F},\mathbb{P})$-semimartingale satisfying \eqref{eq:posterior_semi_ito} and $(\widetilde{H}^n)_{n\in\mathbb{M}}$ is ${\cal K}$-valued $\mathbb{P}\otimes dt$-a.e.\ and satisfies \eqref{eq:Komlos},  there is a subsequence $(\widetilde{H}^{n})_{n\in\mathbb{N}}$ of the one in \eqref{eq:Komlos0} (for notational simplicity, we do not relabel that sequence) for which
	\begin{align}\label{eq:conv_integral}
		(\widetilde{H}^{n} \cdot S^o)_T  \rightarrow (H^*\cdot S^o)_T\quad \mbox{$\mathbb{P}$-a.s. as $n \rightarrow \infty$}.
	\end{align}
	
	Since $|X_T^{\widetilde H^n;o}-X_T^{H^*;o}|\to 0$ $\mathbb{P}$-a.s. as $n \rightarrow \infty$ (by \eqref{eq:conv_integral}) and the limit \eqref{eq:Komlos} holds, we can use the same arguments presented for \eqref{eq:sup_conv} (see Step 1a of part\;(i) in the proof of Lemma~\ref{lem:stability_1}) to have that $\|X^{\widetilde H^n;o}- X^{H^*;o} \|_{\mathscr{S}^{\beta}}\to 0$ as $n \rightarrow \infty$ (with $\beta=\frac{2p}{p+2}\in(1,p)$ and $\alpha=\alpha(\beta)=2$), which ensures that
	\begin{align}\label{eq:conv_integral_strong}
		X^{\widetilde H^n;o}\to X^{H^*;o}\quad \mbox{$\mathbb{P}\otimes dt$-a.e.}\quad\mbox{as $n\to \infty$}.
	\end{align}

	Furthermore, $g$ is convex in $(x,h)$, and so is $f$ in $x$ (see Assumption \ref{as:objective}\;(iv)). Therefore, $(\widetilde{H}^{n})_{n\in\mathbb{N}}$ is still a minimizing sequence, i.e., 
	\begin{align}\label{eq:still_minimize}
		V(0)= \lim_{n\rightarrow  \infty} \mathbb{E}\bigg[\int_0^Tg\big(t,X_t^{\widetilde H^n;o},\widetilde H_t^n\big)dt+ f \big(X_T^{\widetilde{H}^{n};o},l(S_T^o)\big)\bigg].
	\end{align}

	Finally, since both $g(\omega,t,\cdot , \cdot)$ and $f(\omega,\cdot,\cdot)$ are continuous and bounded from below 	for every $(\omega,t)\in\Omega \times [0,T]$ 	 (see  Assumption \ref{as:objective}), 
	a twofold application of Fatou's lemma (together with \eqref{eq:Komlos_1}, \eqref{eq:conv_integral_strong} and \eqref{eq:still_minimize}) shows that 
	\begin{align*}
		\begin{aligned}
			&\mathbb{E}\bigg[\int_0^Tg(t,X_t^{H^*;o},H_t^*)dt+f \big(X_T^{H^*;o},h(S_T^o)\big)\bigg]\\
			&\quad =\mathbb{E}\bigg[\int_0^T\Big(\liminf_{n\to \infty} g(t,X_t^{\widetilde H^n;o},\widetilde H_t^n)\Big)dt + \liminf_{n\rightarrow  \infty}  f \big(X_T^{\widetilde{H}^{n};o},l(S_T^o)\big)\bigg]\leq V(0).
		\end{aligned}
	\end{align*}
	This ensures the optimality of $H^*\in{\cal A}$. 
	The uniqueness of an optimizer follows immediately from both the convexity of $g$ in $(x,h)$ and the `strict' convexity of~$f$ in $x$; see Assumption~\ref{as:objective}\;(iv).  
	
	The claims made in part (ii) and (iii) follow immediately from Proposition \ref{pro:BSDE}. 
\end{proof}

We wrap up this section by establishing a first order optimality condition for the unique optimizer $H^*\in {\cal A}$, which is employed in the proof of Theorem \ref{thm:main}.

\begin{lem}\label{lem:FOC} Suppose that Assumptions \ref{as:posterior_refer}, \ref{dfn:control}, and \ref{as:objective} are satisfied. 
Then the unique optimizer $H^*$ of $V(0)={\cal V}(H^*,0)$ (in Proposition \ref{pro:main0}\;(i)) satisfies the first-order optimality condition: for every $H\in {\cal A}$,
	\begin{align*}
		\mathbb{E}\Bigg[\int_0^T(\nabla_{x,h}g(t,X_t^{H^*;o},H^*_t))^\top
		\begin{pmatrix}
			X_t^{H;o}-X_t^{H^*;o}\\
			H_t-H_t^*
		\end{pmatrix} dt 
		+\partial_{x}f\big(X_T^{H^*;o},h(S_T^o)\big)\big(X_T^{H;o}-X_T^{H^*;o}\big)\Bigg]\geq 0.
	\end{align*}
\end{lem}
\begin{proof}
	Fix $H\in\mathcal{A}$.
	Clearly $H^*+ \theta (H-H^*) \in {\cal A}$ for any $0<\theta<1$, hence  it follows from the optimality of $H^*$ that 
	\begin{align}
		\begin{aligned}
	\label{eq:optimality}
		&\mathbb{E}\Bigg[\int_0^T\frac{1}{\theta}\Big(g\big(t,X_t^{H^*;o}+ \theta \cdot \Delta X_t^{H;o},H_t^*+\theta \cdot \Delta H_t \big)- g\big(t,X_t^{H^*;o},H_t^*\big) \Big) dt\\		
		&\hspace{1.5em} +\frac{1}{\theta}\Big(f \big(X_T^{H^*;o}+ \theta \cdot \Delta X_t^{H;o},l(S_T^o) \big)-f \big(X_T^{H^*;o},l(S_T^o)\big)\Big) \Bigg]\geq 0,
		\end{aligned}
	\end{align}
	where $\Delta X_t^{H;o}:=X_t^{H;o}-X_t^{H^*;o}$ and $\Delta H_t:= H_t-H_t^*$ for $t\in[0,T]$.

	Define for every $t\in[0,T]$ by 
	\begin{align*}
		\begin{aligned}
			\Psi^{H;g}_t&:= \sup_{\theta\in(0,1)}\Big\lvert\frac{1}{\theta}\Big(g \big(t,X_t^{H^*;o} + \theta \cdot \Delta X_t^{H;o},H_t^*+\theta \cdot  \Delta H_t \big)-g(t,X_t^{H^*;o},H_t^*)\Big)\Big\rvert\\
			\Psi^{H;f}&:= \sup_{\theta\in(0,1)}\Big\lvert\frac{1}{\theta}\Big(f \big(X_T^{H^*;o} + \theta \cdot  \Delta H_t,l(S_T^o) \big)-f \big(X_T^{H^*;o},l(S_T^o) \big)\Big)\Big\rvert.
		\end{aligned}
	\end{align*}

	Then we claim that $(\Psi^{H;g}_t)_{t\in [0,T]}\in \mathbb{L}^1(\mathbb{R})$ and $\Psi^{H;f}\in L^1({\cal F}_T;\mathbb{R})$. 
	If that is the case, then the proof follows from 	\eqref{eq:optimality} and the dominated convergence theorem.
	
	\vspace{0.5em}
	\noindent {\it Step 1. estimate $\|\Psi^{H;g}\|_{\mathbb{L}^1}<\infty$ }: First note that by Assumption \ref{as:objective}\;(ii) and Remark \ref{rem:objective}\;(i) (and the power triangle inequality),
	\begin{align*}
		\big\|D^2_{x,h}g(\omega,t,x+\hat x,h+\hat h) \big\|_{\operatorname{F}} &\leq \overline C_{2} 2^{\frac{3r}{2}}(1+\lvert x|^{r}+\lvert h|^{r}+\lvert \hat x|^{r}+|\hat h |^{r}),\\
		\big|\nabla_{x,h}g(\omega,t,x,h)\big|&\leq \widetilde C (1+\lvert x|^{r+1}+\lvert h|^{r+1})
	\end{align*}
	for every $(\omega,t,x,h)\in \Omega\times[0,T]\times \mathbb{R}\times \mathbb{R}^d$ and $(\hat x,\hat h)\in \mathbb{R}\times \mathbb{R}^d$. Hence, a second-order Taylor expansion of $g$ (with the power triangle inequality and the fact that $\lvert H_t \rvert \leq K$ $\mathbb{P}\otimes dt$-a.e.\ for every $H\in {\cal A}$; see Assumption \ref{dfn:control}) implies that  for any $H\in {\cal A}$, 
	\begin{align*}
		\begin{aligned}
			\big\|\Psi^{H;g}\big\|_{\mathbb{L}^1} &\leq  \widetilde C2^{\frac{1}{2}}  \mathbb{E}\bigg[\int_0^T \Big(1+\lvert X_t^{H^*;o} \rvert^{r+1} + K^{r+1} \Big)\Big(\big|\Delta X_t^{H;o} \big|+2K\Big)dt \bigg]\\
			&\quad + \overline C_{2} 2^{\frac{3r}{2}} \mathbb{E}\bigg[\int_0^T \Big(1+\lvert X_t^{H^*;o} \rvert^{r} + \lvert X_t^{H;o} \rvert^{r} + 2\cdot K^{r} \Big)\Big(\big|\Delta X_t^{H;o} \big|^2+(2K)^2\Big) dt \bigg]\\
			&=: \widetilde C2^{\frac{1}{2}} \cdot  \operatorname{I}^{H;g} + \overline C_{2} 2^{\frac{3r}{2}} \cdot  \operatorname{II}^{H;g}.
		\end{aligned}
	\end{align*}
	Since $0<  r <  \min\{\frac{p-2}{2}, p-3\}$ (see Assumptions \ref{as:posterior_refer}\;(i) and \ref{as:objective}\;(ii)), H\"older's inequality (with exponent $\frac{p}{r+1}>1$ and conjugate exponent $\frac{p}{p-(r+1)}>1$) ensures that  
	\begin{align*}
		\operatorname{I}^{H;g}
		\leq   \Big\|1+\lvert X^{H^*;o} \rvert^{r+1}+K^{r+1}\Big\|_{\mathbb{L}^{\frac{p}{r+1}}}  \Big( \big\lVert \Delta X^{H;o}  \big\rVert_{\mathbb{L}^{\frac{p}{p-(r+1)}}}+ 2KT \Big).
	\end{align*}
	
	Furthermore, by Jensen's inequality (noting that $x\rightarrow |x|^{p-(r+1)}$ is convex since $p>r+2$), 
	\[
		\lVert \Delta X^{H;o}  \rVert_{\mathbb{L}^{\frac{p}{p-(r+1)}}}\leq \lVert \Delta X^{H;o} \rVert_{\mathbb{L}^{p}} 
		\leq T^{\frac{1}{p}} \lVert \Delta X^{H;o} \rVert_{\mathscr{S}^{p}},
	\]
	and by the power triangle inequality, 
	\[
	\mathbb{E}\bigg[\int_0^T\left(1+\lvert X_t^{H^*;o} \rvert^{r+1}+K^{r+1}\right)^{\frac{p}{r+1}}dt \bigg]\leq 3^{\frac{p}{r+1}}T \Big( 1+\big\|X^{H^*;o}\big\|_{\mathscr{S}^p}^p +K \Big).
	\]
	Hence, the a priori estimate on $X^{H;b,\sigma}$ (for any $(b,\sigma)\in\mathcal{B}^1$ and $H\in {\cal A}$) detailed in Lemma~\ref{lem:p_integrability}\;(ii) ensures that $\operatorname{I}^{H;g}<\infty$, as claimed.
	
	Similarly, we can deduce that 
	\begin{align*}
		\operatorname{II}^{H;g}\leq  4 T^{\frac{r}{p}} \Big(1+ \big\|X^{H^*;o}\big\|_{\mathscr{S}^p}^p + \big\|X^{H;o}\big\|_{\mathscr{S}^p}^p +(2K)^p\Big)^{\frac{r}{p}}  \big(T^{\frac{2}{p}} \lVert \Delta X^{H;o} \rVert_{\mathscr{S}^{p}}^2+(2KT)^2\big)<\infty,
	\end{align*}
	thus $\Psi^{H;g,2}\in \mathbb{L}^1(\mathbb{R})$, as claimed.
	
	\vspace{0.5em}
	\noindent {\it Step 2. estimate $\|\Psi_T^{H;f}\|_{{L}^1}<\infty$ }: Similarly, by  Assumption \ref{as:objective}\;(ii) and Remark \ref{rem:objective}\;(i) (and the power triangle inequality), for every $\omega\in\Omega$ and $x,\hat {x},y \in \mathbb{R}$
	\begin{align*}
	\lvert \partial_{xx}f(\omega,x+\hat x,y)\rvert 
	&\leq \overline C_2 2^{\frac{r}{2}} \left(1+\lvert x \rvert^r+\lvert \hat x \rvert^r+\lvert y \rvert^r\right), \quad 
	\lvert \partial_{x}f(\omega,x,y) \rvert 
	\leq  \widetilde{C} \left(1  +\lvert {x} \rvert^{r+1} +\lvert y \rvert^{r+1} \right).
	\end{align*}
	Hence, a second-order Taylor expansion of $f$ implies that  for any $H\in {\cal A}$, 
	\begin{align}\label{eq:L_1_psi}
		\begin{aligned}
			\big\|\Psi^{H;f}\big\|_{L^1} &\leq \widetilde{C}  \mathbb{E}\Big[\big(1+\lvert X_T^{H^*;o} \rvert^{r+1}+\lvert l(S_T^o) \rvert^{r+1}\big) \big\lvert \Delta X^{H;o}_T \big\rvert \Big]\\
			&\quad+ \overline {C}_2 2^{\frac{r}{2}} \mathbb{E}\Big[\big(1+\lvert X_T^{H;o} \rvert^{r}+\lvert X_T^{H^*;o} \rvert^{r}+\lvert l(S_T^o) \rvert^{r}\big) \big\lvert \Delta X^{H;o}_T \big\rvert^2 \Big].\
		\end{aligned}
	\end{align}
	
	Since $l(S_T^o)$ appearing in \eqref{eq:L_1_psi} satisfies the a priori estimate given in Lemma~\ref{lem:p_integrability}\;(ii) with the same exponent of the a priori estimate on $X_T^{H^*;o}$ and $X_T^{H;o}$, the claim that $\|\Psi^{H;f}\|_{L^1}<\infty$  follows from the same arguments as those used in the proof of Step 1, and we omit the proof.
\end{proof}

\subsection{Proof of Theorem \ref{thm:main}}
\label{sec:proof:main_ii}

Following the outline of the proof at the beginning of this section, we first establish the crucial observation that the expectation of the first order derivative can be expressed in a way that is {\it linear} in $b-b^o$ and $\sigma-\sigma^o$, i.e.\ \eqref{eq:linear}.
To that end, recall that $\langle X,Y \rangle_{\mathbb{P}\otimes dt}=\mathbb{E}[\int_0^T \langle X_t,Y_t\rangle dt]$ and similarly for	$\langle\cdot,\cdot \rangle_{\mathbb{P}\otimes dt,\operatorname{F}}$, see Section \ref{sec:preliminary}.

\begin{lem}\label{lem:inner_prod}
	Suppose that Assumptions \ref{as:posterior_refer}, \ref{dfn:control}, and \ref{as:objective} are satisfied.
	Let $H\in {\cal A}$, let $(A^H,(R_t^H)_{t\in[0,T]})$ and $B^H$ be given in \eqref{eq:terminal}, and let $(Y^{H},Z^{H},L^{H})$ and $({\mathcal{Y}}^H,\mathcal{Z}^H,\mathcal{L}^H)$ be the unique solution of the BSDEs given in \eqref{eq:BSDE_univ1} and \eqref{eq:BSDE_univ3}, respectively.   
	Then the following holds: for every $\varepsilon\in[0,1]$ and $(b,\sigma)\in \mathcal{B}^\varepsilon$,
	\begin{align}\label{eq:inner_prod}\quad
		\begin{aligned}
			&\mathbb{E}\bigg[\int_0^TR^H_t\big(X_t^{H;b,\sigma} - X_t^{H;o}\big)dt + A^H \big(X_T^{H;b,\sigma} - X_T^{H;o}\big)+(B^H)^\top \big(S_T^{b,\sigma} - S_T^{o}\big) \bigg]\\
			&\quad = \langle Y^H H+{{\cal Y}}^H, b-b^o \rangle_{\mathbb{P}\otimes dt}+ \langle Z^HH^\top+{{\cal Z}}^H, \sigma-\sigma^o \rangle_{ \mathbb{P}\otimes dt,\operatorname{F}}.
		\end{aligned}
	\end{align}
\end{lem}
\begin{proof} 
	{\it Step 1.}  We start by focusing on the first and second terms in \eqref{eq:inner_prod} and show that
	\begin{align}\label{eq:inner_x}
		\begin{aligned}
		&\mathbb{E}\bigg[\int_0^TR^H_t\big(X_t^{H;b,\sigma} - X_t^{H;o}\big)dt + A^H \big(X_T^{H;b,\sigma} - X_T^{H;o}\big) \bigg]\\
		&\quad = \langle Y^H H, b-b^o \rangle_{\mathbb{P}\otimes dt}+ \langle Z^HH^\top, \sigma-\sigma^o \rangle_{ \mathbb{P}\otimes dt,\operatorname{F}}.
		\end{aligned}
	\end{align}
	To that end, denote for every $t\in[0,T]$ by $\mathfrak{I}_t^H:=\int_t^TR_s^H ds$. Then from the integration by~parts, 
	\begin{align*}
		\begin{aligned}
			\int_0^TR^H_t\big(X_t^{H;b,\sigma} - X_t^{H;o}\big)dt= \int_0^T\mathfrak{I}_t^H H_t^\top (b_t- b_t^o)dt + \int_0^T\mathfrak{I}_t^HH_t^\top (\sigma_t- \sigma_t^o)dW_t.
		\end{aligned}
	\end{align*}
	Moreover, we denote by 
	\begin{align*}
		\begin{aligned}
		&\mbox{$\Psi^{b} :=\int_0^T \mathfrak{I}_t^HH_t^\top (b_t- b_t^o)dt$},\;\; &&\mbox{$\Psi^{\sigma}:=\int_0^T \mathfrak{I}_t^HH_t^\top (\sigma_t- \sigma_t^o)dW_t$}, \\
		&\mbox{$\Xi^{b} :=A^H \int_0^T H_t^\top (b_t- b_t^o)dt$},\;\; &&\mbox{$\Xi^{\sigma}:=A^H \int_0^T H_t^\top (\sigma_t- \sigma_t^o)dW_t$},
		\end{aligned}
	\end{align*}
	so that $\mathbb{E}[\int_0^TR^H_t(X_t^{H;b,\sigma} - X_t^{H;o})dt + A^H (X_T^{H;b,\sigma} - X_T^{H;o}) ] = \mathbb{E}[\Psi^{b}+\Psi^{\sigma}+\Xi^{b}+\Xi^{\sigma}]$.   
	
	\vspace{0.5em}
	\noindent {\it Step 1a:} We claim that $\mathbb{E}[\Psi^{b}+\Xi^{b} ]= \langle Y^H H, b-b^o \rangle_{\mathbb{P}\otimes dt}$. First note that
	$\mathbb{E}[\lvert\Psi^{b}\rvert]$, $\mathbb{E}[\lvert\Xi^{b}\rvert]<\infty$. 
	Indeed, by H\"older's inequality (with exponent $p>3$) and the fact that  $\lvert H_t\rvert \leq K$ $\mathbb{P}\otimes dt$-a.e.,
	\begin{align*}
		\mathbb{E}\big[\lvert\Psi^{b}\rvert\big] &\leq \| \mathfrak{I}^H \|_{\mathbb{L}^q} \| H^\top(b-b^o)\|_{\mathbb{L}^p}\leq K \| \mathfrak{I}^H \|_{\mathbb{L}^q} \| b-b^o\|_{\mathbb{L}^p}.
	\end{align*}
	Furthermore, by Jensen's inequality (noting that $x\to |x|^{\frac{2}{q}}$ is convex as~$q<2$) 
	\begin{align}
		\begin{aligned}\label{eq:Integ_I^H}
		\| \mathfrak{I}^H \|_{\mathbb{L}^q}^q&\leq T\cdot \mathbb{E}\Bigg[\sup_{t\in[0,T]}|\mathfrak{I}_t^H|^q\Bigg]=T\cdot \mathbb{E}\bigg[\sup_{t\in[0,T]}\bigg|\int_t^TR_s^Hds\bigg|^q\bigg]\\
		& \leq T\cdot \mathbb{E}\Bigg[\bigg(\int_0^T|R_s^H|ds\bigg)^q\Bigg] \leq T^{\frac{q}{2}+1}\cdot \mathbb{E}\Bigg[\bigg(\int_0^T|R_s^H|^2ds\bigg)^{\frac{q}{2}}\Bigg] \leq T^{\frac{q}{2}+1} \|R^H \|_{\mathbb{L}^2}^{{q}}
		\end{aligned}
	\end{align}
	Combining this with $R^H\in \mathbb{L}^2(\mathbb{R})$ (see Proposition \ref{pro:BSDE}\;(i)) and $\lVert b - b^o\rVert_{\mathbb{L}^p} \leq \gamma \varepsilon\leq 1$, we conclude that $\mathbb{E}[\lvert\Psi^{b}\rvert]<\infty.$ 
	
	Similarly, by H\"older's inequality (with exponent $p>3$)
	\begin{align*}
		\mathbb{E}\big[\lvert\Xi^{b}\rvert\big]
		&\leq \big\lVert A^H \big\rVert_{L^{q}}\Big\lVert  \int_0^T H_t^\top (b_t- b_t^o)dt \Big\rVert_{L^{p}}
		\leq K  \big\lVert A^H \big\rVert_{L^{q}}T^{1-\frac{1}{p}}  \lVert b-b^o \rVert_{\mathbb{L}^{p}}.
	\end{align*}
	Moreover, since $\lvert \partial_x f(\omega,x,y) \rvert \leq \widetilde{C} (1 +\lvert x \rvert^{r+1}+\lvert y \rvert^{r+1})$ for every $\omega\in \Omega$ and $x,y\in \mathbb{R}$ (see Remark~\ref{rem:objective}\;(i)), H\"older's inequality with exponent $\frac{p-1}{(r+1)}>1$ 
	ensures that 
	\begin{align*}
		\begin{aligned}
	\mathbb{E}\big[ \big| A^H \big|^q \big]
	&\leq   3^q \widetilde{C}^q \Big(1+ \mathbb{E}\left[ |X_T^{H;o}|^{q(r+1)} \right] + \mathbb{E}\left[ |l(S_T^o)|^{q(r+1)} \right] \Big)\\
	&\leq 3^q \widetilde{C}^q \Big(1+ \mathbb{E}\left[ |X_T^{H;o}|^{p} \right]^{\frac{r+1}{p-1}}  + \mathbb{E}\Big[ |l(S_T^o)|^{p} \Big]^{\frac{r+1}{p-1}} \Big).
	\end{aligned}
	\end{align*}
	Combining this with Lemma \ref{lem:p_integrability}\;(ii) and $\lVert b - b^o\rVert_{\mathbb{L}^p} \leq  1$, we conclude that $\mathbb{E}[\lvert\Xi^{b}\rvert]<\infty.$ 
	
	Therefore, by Fubini's theorem and since $Y_t^H= \mathbb{E}[A^H+\mathfrak{I}_t^H| {\cal F}_t]$ for $t\in[0,T]$ (see Proposition~\ref{pro:BSDE}\;(ii) and recall that $\mathfrak{I}_t^H=\int_t^TR^H_sds$), 
	\begin{align}\label{eq:inner_b}
		\mathbb{E}\left[\Psi^{b}+\Xi^{b}\right]=  \int_0^T \mathbb{E}\Big[\mathbb{E}\big[A^H+\mathfrak{I}_t^H\big| {\cal F}_t\big] H_t^\top (b_t- b_t^o)\Big]dt=\langle Y^H H, b-b^o \rangle_{\mathbb{P}\otimes dt}.
	\end{align}
	
	\vspace{0.5em}
	\noindent {\it Step 1b:} Next, we claim that $\mathbb{E}[\Psi^{\sigma}+\Xi^{\sigma}]= \langle Z^HH^\top, \sigma-\sigma^o \rangle_{ \mathbb{P}\otimes dt,\operatorname{F}}.$
	Denote for $t\in[0,T]$~by 
	\[
		\qquad M_t^H:=\int_0^t \underline {\mathfrak{I}}_s^H H_s^\top (\sigma_s- \sigma_s^o)dW_s,\quad\mbox{where}\;\;\;\underline {\mathfrak{I}}_s^H:= {\mathfrak{I}}^H_0-{\mathfrak{I}}^H_s=\int^s_0R^H_udu, 
	\]
	so that 
	$\mathbb{E}[\Psi^{\sigma}+\Xi^{\sigma}]= \mathbb{E}[(A^H+{\mathfrak{I}}^H_0)\int_0^T H_t^\top (\sigma_t- \sigma_t^o)dW_t - M_T^H].$ 
	
	Moreover, by Proposition~\ref{pro:BSDE}\;(ii) (noting again that $A^H+{\mathfrak{I}}^H_0=A^H+\int_0^TR_t^Hdt$), 
	\begin{align}
		\nonumber
			\mathbb{E}\left[\Psi^{\sigma}+\Xi^{\sigma}\right]&=  \mathbb{E} \left[\int_0^T (Z_t^H)^\top dW_t \int_0^T H_t^\top (\sigma_t- \sigma_t^o)dW_t+  (L_T^H-L_0^H) \int_0^T H_t^\top (\sigma_t- \sigma_t^o)dW_t\right.\\
			& \left.\qquad \;\;\;+Y_0^H \int_0^T H_t^\top (\sigma_t- \sigma_t^o)dW_t- M_T^H\right]=: \mathbb{E} \left[\operatorname{I}^{\sigma}+\operatorname{II}^{\sigma}+\operatorname{III}^{\sigma}- M_T^H \right].\label{eq:inner_s1}
	\end{align}

	An application of the It\^o-isometry shows that 
	\begin{align*}
			\mathbb{E} \left[\operatorname{I}^{\sigma}\right]= \mathbb{E} \bigg[\int_0^T (Z_t^H)^\top (\sigma_t-\sigma_t^o)^\top H_tdt\bigg] 
			&= \mathbb{E} \bigg[\int_0^T\langle Z_t^HH_t^\top, \sigma_t-\sigma_t^o \rangle_{\operatorname{F}} dt\bigg] \\
			&=\langle Z^H H^\top, \sigma-\sigma^o \rangle_{ \mathbb{P}\otimes dt,\operatorname{F}}.
	\end{align*} 
	Moreover, by \cite[Lemma 2 \& Theorem 35, p.149]{Protter2005},  $L^H$ and $(\int_0^t H_s^\top ({\sigma}_s- \sigma_s^o)dW_s)_{t\in[0,T]}$ are strongly orthogonal, thus $\mathbb{E} [\operatorname{II}^{\sigma}]  =0$.
	Since ${\cal F}_0$ is trivial (see Section \ref{sec:preliminary}),  $\mathbb{E}[\operatorname{III}^{\sigma} ]=0$.
	
	Finally, it remains to show that $\mathbb{E}[M_T^H]=0$. For this, we claim that $\mathbb{E}[\sup_{t\in[0,T]}|M_t^H|]<\infty$.  If that is the case, by \cite[Theorem 47, p. 35]{Protter2005} $M^H$ is a uniformly integrable $(\mathbb{F},\mathbb{P})$-martingale and hence the claim holds. 
	
	
	Indeed, by the BDG inequality and H\"older's inequality with exponent $p>3$  
	\begin{align*}
		\mathbb{E}\Bigg[\sup_{t\in[0,T]}\big|M_t^H\big|\Bigg]&\leq C_{\operatorname{BDG},1} \mathbb{E}\Bigg[\bigg(\int_0^T \big|\underline {\mathfrak{I}}^H_t H_t^\top (\sigma_t- \sigma_t^o)\big|^2dt\bigg)^{\frac{1}{2}}\Bigg]\\ &\leq C_{\operatorname{BDG},1}K \mathbb{E}\Bigg[\bigg(\sup_{t\in[0,T]}|\underline {\mathfrak{I}}^H_t| \bigg)\cdot  \bigg(\int_0^T\|\sigma_t- \sigma_t^o\|_{\operatorname{F}}^2dt\bigg)^{\frac{1}{2}}\Bigg]\\
		&\leq C_{\operatorname{BDG},1}K \big\| \underline {\mathfrak{I}}^H \big\|_{\mathscr{S}^q} \lVert \sigma-\sigma^o \rVert_{\mathbb{H}^{p}} .
	\end{align*}
	Moreover, recall that $\underline {\mathfrak{I}}^H_t:= \int_0^tR_s^Hds$ for $t\in[0,T]$, by Jensen's inequality 
	\[
		\big\| \underline {\mathfrak{I}}^H \big\|_{\mathscr{S}^q}^q=\mathbb{E}\Bigg[\sup_{t\in[0,T]}|\underline {\mathfrak{I}}^H_t|^q\Bigg]\leq \mathbb{E}\Bigg[\bigg(\int_0^T|R_s^H|ds\bigg)^{q}\Bigg] \leq T^{\frac{q}{2}}  \|R^H \|_{\mathbb{L}^2}^{{q}}<\infty.
	\]
	Combining this with $\lVert \sigma-\sigma^o \rVert_{\mathbb{H}^{p}} \leq \eta \varepsilon\leq 1$, we~conclude that $\mathbb{E}[M_T^H]=0$.

	We conclude that
	\begin{align}\label{eq:inner_s}
		\mathbb{E}\left[\Psi^{\sigma}+\Xi^{\sigma}\right]=\mathbb{E}\left[\operatorname{I}^{\sigma}\right]= \langle Z^HH^\top, \sigma-\sigma^o \rangle_{ \mathbb{P}\otimes dt,\operatorname{F}}
	\end{align}
	and combined with \eqref{eq:inner_b}, this shows \eqref{eq:inner_x}. 
	
	\vspace{0.5em}
	\noindent
	{\it Step 2.} 
	We proceed to analyze the second term in \eqref{eq:inner_prod} and show that
	\begin{align}\label{eq:inner_y}
		\mathbb{E}\Big[ \big( B^H  \big)^\top \big(S_T^{b,\sigma} - S_T^{o}\big) \Big]= \langle {{\cal Y}}^H, b-b^o \rangle_{\mathbb{P}\otimes dt}+ \langle {{\cal Z}}^H, \sigma-\sigma^o \rangle_{ \mathbb{P}\otimes dt,\operatorname{F}}.
	\end{align}
	To that end, we first note that 
	\[
	\big(B^H \big)^\top (S_T^{b,\sigma} - S_T^{o}) = \sum_{i=1}^d  \partial_{y}f\big(X_T^{H;o},h(S_T^o)\big)\partial_{s_i}l(S_T^o) \Bigg(  \int_0^T (b_{t}^i- b_{t}^{o,i} )dt +\int_0^T (\sigma_{t}^i- \sigma_{t}^{o,i})dW_t\Bigg),
	\]
	where for $i=1,\dots,d$,  $(b_{t}^i- b_{t}^{o,i})_{t\in[0,T]}$  denotes the $i$-th component of $b- b^o$ and  $(\sigma_{t}^i- \sigma_{t}^{o,i})_{t\in[0,T]}$ denotes the  $i$-th row of $\sigma-\sigma^o$. 
	
	For $i=1,\dots,d,$ set  $B^{H,i}:= \partial_{y}f(X_T^{H;o},h(S_T^o))\partial_{s_i}l(S_T^o)$ and denote by 
	\begin{align*}
		\Xi^{b,i}:=  B^{H,i}  \int_0^T (b_{t}^i- b_{t}^{o,i} )dt,\qquad \Xi^{\sigma,i}:=  B^{H,i} \int_0^T (\sigma_{t}^i- \sigma_{t}^{o,i})dW_t,
	\end{align*}
	so that 
	\begin{align}\label{eq:xi_b_sig_sum}
		\mathbb{E}\Big[ \big(B^H\big)^\top (S_T^{b,\sigma} - S_T^{o}) \Big]=  \mathbb{E}\left[\sum_{i=1}^d\big(\Xi^{b,i}+\Xi^{\sigma,i}\big)\right].
	\end{align}
	
	First note that for every $i=1,\dots,d$, $({\cal Y}^{H,i},{\cal Z}^{H,i},{\cal L}^{H,i}))\in \mathscr{S}^2(\mathbb{R})\times \mathscr{H}^2(\mathbb{R}^d)\times \mathscr{M}^2(\mathbb{R})$ is the unique solution of BSDE with terminal condition $B^{H,i}$.  It follows from the same arguments as given for the proof of \eqref{eq:inner_b} and the fact that $\lvert \partial_{s_i}l(\cdot) \rvert \leq C_l$ (see Remark \ref{rem:objective}\;(ii)) that 
	\begin{align}\label{eq:xi_b_i}
		\mathbb{E}[\Xi^{b,i}]= \mathbb{E}\Bigg[\int_0^T {\cal Y}^{H,i}_t(b_{t}^i- b_{t}^{o,i} )dt\Bigg].
	\end{align}
	
	Similarly, it follows from the same argument as given for the proofs of \eqref{eq:inner_s1} and \eqref{eq:inner_s} (by replacing $Z^H$ and $H^\top (\sigma -\sigma^o)$ with ${\cal Z}^{H,i}$ and $\sigma^{i}- \sigma^{o,i}$, respectively) that 
	\begin{align}\label{eq:xi_sig_i2}
		\mathbb{E}[\Xi^{\sigma,i}]= \langle {\cal Z}^{H,i}, (\sigma^{i}- \sigma^{o,i})^\top \rangle_{\mathbb{P}\otimes dt}.
	\end{align}
	Combining \eqref{eq:xi_b_sig_sum}, \eqref{eq:xi_b_i}, and \eqref{eq:xi_sig_i2}, we  obtain that indeed
	\begin{align*}
		\begin{aligned}
			\mathbb{E}\Big[\big(B^H  \big)^\top (S_T^{b,\sigma} - S_T^{o}) \Big]&= \sum_{i=1}^d\Bigg(\mathbb{E}\bigg[\int_0^T {\cal Y}^{H,i}_t(b_{t}^i- b_{t}^{o,i} )dt\bigg]+ \langle {\cal Z}^{H,i}, (\sigma^{i}- \sigma^{o,i})^\top \rangle_{\mathbb{P}\otimes dt} \Bigg)\\
			&=\langle {{\cal Y}}^H, b-b^o \rangle_{\mathbb{P}\otimes dt}+ \langle {{\cal Z}}^H, \sigma-\sigma^o \rangle_{ \mathbb{P}\otimes dt,\operatorname{F}}.
		\end{aligned}
	\end{align*}
	
	The proof of the lemma now follows from \eqref{eq:inner_x} and \eqref{eq:inner_y}. 
\end{proof}

\begin{lem}\label{lem:up_bound}
	Suppose that Assumptions \ref{as:posterior_refer}, \ref{dfn:control}, and \ref{as:objective} are satisfied. 
	Moreover, let $H^*$ be the unique optimizer for $V(0)$ (given in Proposition \ref{pro:main0}\;(i)) and let $(Y^{*},Z^{*},L^{*})$ and $({\cal Y}^*,{\cal Z}^*,{\cal L}^*)$ be the unique solution of \eqref{eq:BSDE_re_x} and \eqref{eq:BSDE_re_y}, respectively.
Then there exists a constant $C>0$ such that for any $\varepsilon\in[0,1]$, 
	\begin{align*}
		0\leq V(\varepsilon) - V(0) \leq \varepsilon \left(\gamma  \lVert Y^*H^*+ {{\cal Y}}^* \rVert_{\mathbb{L}^q} + \eta \lVert Z^*(H^*)^\top + {{\cal Z}}^*\rVert_{\mathbb{H}^q}  \right)+C \varepsilon^2.
	\end{align*} 
	In particular, $\lim_{\varepsilon \downarrow 0} V(\varepsilon)=V(0)$. 
\end{lem}
\begin{proof} 
	\noindent
	{\it Step 1.} Set 
	\begin{align*}
		&A^{*}:=A^{H^*}=\partial_{x} f\big(X_T^{H^*;o},l(S_T^o)\big),\qquad R^*:=R^{H^*}=(\partial_x g(t,X_t^{H^*;o},H_t^*))_{t\in[0,T]},\\
	 	& B^{*}:=B^{H^*}=\partial_{y} f\big(X_T^{H^*;o},l(S_T^o)\big)\nabla_s l(S_T^o).
	\end{align*}
	For every $\varepsilon\geq 0$  and $(b,\sigma)\in \mathcal{B}^\varepsilon$, denote for every $t\in[0,T]$ by 
	\begin{align}\label{eq:abbrv_delta}
		\quad \Delta X_t^{*;b,\sigma}:= X_t^{H^*;b,\sigma} - X_t^{H^*;o},\qquad  \Delta S_t^{b,\sigma}:= S_t^{b,\sigma} -  S_t^{o},\qquad \Delta l^{b,\sigma}_t:=l(S_t^{b,\sigma})-l(S_t^o).
	\end{align}
	
	A second-order Taylor expansion of $g(t,\;\cdot\;,H_t^*)$ around $X_t^{H^*;o}$ for each $t\in[0,T]$ implies~that for every $\varepsilon\geq 0$  and $(b,\sigma)\in \mathcal{B}^\varepsilon$, 
	\begin{align}\label{eq:taylor_01}
			\int_0^T\Big(g(t,X_t^{H^*;b,\sigma},H_t^*)-g(t,X_t^{H^*;o},H_t^*)\Big)dt = \int_0^T R_t^*\Delta X_t^{*;b,\sigma} dt + \int_0^T \operatorname{I}_t^{b,\sigma} dt, 
	\end{align}
	where $\operatorname{I}_t^{b,\sigma}$ is given by for every $t\in[0,T]$ 
	\begin{align}
		\operatorname{I}_t^{b,\sigma} := (\Delta X_t^{*;b,\sigma})^2\int_0^1(1-\theta)\partial_{xx}g\big(t,X_t^{H^*;o}+\theta \cdot \Delta X_t^{*;b,\sigma},H_t^*\big)d\theta. 
	\end{align}

	Similarly, a second-order Taylor expansion of $f$ and $l$ around $(X_T^{H^*;o},l(S_T^o))$ and $S^o_T$, respectively, implies that  for every $\varepsilon\geq 0$  and $(b,\sigma)\in \mathcal{B}^\varepsilon$,
	\begin{align}\label{eq:taylor_1}
		\begin{aligned}
			f\big(X_T^{H^*;b,\sigma},l(S_T^{b,\sigma})\big)- f\big(X_T^{H^*;o},l(S_T^o)\big)
			= A^{*} \Delta X_T^{*;b,\sigma}+\big( B^{*}  \big)^\top \Delta S_T^{b,\sigma}	+\operatorname{II}^{b,\sigma}+\operatorname{III}^{b,\sigma},
		\end{aligned}
	\end{align}
	where $\operatorname{II}^{b,\sigma}$ and $\operatorname{III}^{b,\sigma}$ are given by 
	\begin{align*}
		\begin{aligned}
			\operatorname{II}^{b,\sigma}&:= \Big(\Delta X_T^{*;b,\sigma}, \Delta l_T^{b,\sigma}\Big)\int_0^1 (1-\theta) D_{x,y}^2 f\big({X}_T^{H^*;o} +\theta \Delta X_T^{*;b,\sigma} , l(S_T^{o})+\theta \Delta {l}_T^{b,\sigma}\big) d\theta
			\begin{pmatrix}
				\Delta X_T^{*;b,\sigma}\\
				\Delta l_T^{b,\sigma}
			\end{pmatrix}
			,\\
			\operatorname{III}^{b,\sigma}&:= \partial_{y} f\big(X_T^{H^*;o},l(S_T^o)\big)  \big(\Delta S_T^{b,\sigma}\big)^\top\bigg(\int_0^1 (1-\theta) D_s^2l\big(S_T^o+\theta \Delta S_T^{b,\sigma}\big)  d\theta\bigg) \;\Delta S_T^{b,\sigma},
		\end{aligned}
	\end{align*}
	where $\Delta X_T^{*;b,\sigma}$, $\Delta S_T^{b,\sigma}$ and $\Delta l^{b,\sigma}_T$ are given in \eqref{eq:abbrv_delta}.
	
	\vspace{0.5em}
	\noindent
	{\it Step 2.}
	We claim that there exist $C_{\operatorname{I}},C_{\operatorname{II}},C_{\operatorname{III}}>0$ such that for every $\varepsilon\in [0,1]$,
	\begin{align}\label{eq:estimate_I}
		\sup_{(b,\sigma)\in \mathcal{B}^\varepsilon}\mathbb{E}\bigg[\sup_{t\in[0,T]}\lvert\operatorname{I}_t^{b,\sigma}\rvert+ \lvert \operatorname{II}^{b,\sigma}| + \lvert \operatorname{III}^{b,\sigma}  \rvert\bigg] \leq \big(C_{\operatorname{I}}+C_{\operatorname{II}}+C_{\operatorname{III}}\big) \varepsilon^2.
	\end{align}
	
	\noindent
	{\it Step 2a.~estimates on $\mathrm{I}^{b,\sigma}$} : For every $t\in[0,T]$, set
	\[
	\Psi_t^{b,\sigma;g}:= 1+\lvert X_t^{H^*;b,\sigma} \rvert^r+\lvert X_t^{H^*;o} \rvert^r+\lvert H_t^* \rvert^r.
	\]
	Then, H\"older's inequality with exponent $\frac{p-2}{r}>1$ (and the power triangle inequality) and the a priori estimates given in Lemma \ref{lem:p_integrability}\;(ii) (with the fact that $|H_t^*|\leq K$ $\mathbb{P}\otimes dt$-a.e.) show that for every $\varepsilon\in[0,1]$,
	\begin{align}\label{eq:V_estimate_0}
		\sup_{\varepsilon\in[0,1]}\sup_{(b,\sigma)\in\mathcal{B}^\varepsilon}\left\lVert \Psi^{b,\sigma;g} \right\rVert_{\mathscr{S}^{\frac{p}{p-2}}}<\infty.
	\end{align}

	By Assumption \ref{as:objective}\;(ii) we have that for every $(\omega,t,x,h)\in\Omega\times [0,T]\times \mathbb{R}\times \mathbb{R}^d$ and $\hat x\in \mathbb{R}$,
	\[|\partial_{xx}g(\omega,t,x+\hat x,h) | \leq \overline C_2 2^{\frac{3r}{2}}(1+\lvert x \rvert^r+\lvert \hat x \rvert^r+\lvert h \rvert^r),\]
	hence it follows that for every $(b,\sigma)\in {\cal B}^\varepsilon$ and $\varepsilon\in[0,1]$, 
	\begin{align}\label{eq:estimate_II1_0}
		\begin{aligned}
			\|\operatorname{I}^{b,\sigma}\|_{\mathscr{S}^1}=\mathbb{E}\bigg[ \sup_{t\in[0,T]}\lvert\operatorname{I}_t^{b,\sigma}\rvert \bigg]
			&\leq \overline C_2 2^{\frac{3}{2}{r}} \cdot  \mathbb{E}\bigg[\sup_{t\in[0,T]} |\Psi^{b,\sigma}_t|\cdot  \sup_{t\in[0,T]}\big|\Delta X_t^{*;b,\sigma}\big|^2  \bigg]
			\\
			&\leq  \overline C_2 2^{\frac{3}{2}{r}} \cdot 
			\lVert \Psi^{b,\sigma;g} \rVert_{\mathscr{S}^{\frac{p}{p-2}}}  \big\lVert \Delta X^{*;b,\sigma}\big \rVert_{\mathscr{S}^p}^2,
		\end{aligned}
	\end{align}
	where we used  H\"older's inequality (with exponent $\frac{p}{2}>1$) in the second inequality.
	
	By Lemma \ref{lem:p_integrability}\;(i), 
	it holds that $\sup_{(b,\sigma)\in \mathcal{B}^\varepsilon}\lVert \Delta X^{*;b,\sigma} \rVert_{\mathscr{S}^p}^2\leq C_{1}^{\frac{2}{p}} \varepsilon^2$. 
	Combined with \eqref{eq:V_estimate_0} and \eqref{eq:estimate_II1_0}, this ensures that there is $C_{\operatorname{I}}>0$ such that $\sup_{(b,\sigma)\in \mathcal{B}^\varepsilon}\lVert\operatorname{I}^{b,\sigma} \rVert_{\mathscr{S}^1}\leq C_{\operatorname{I}} \varepsilon^2$ for every~$\varepsilon\in[0,1]$.

	\vspace{0.5em}
	\noindent
	{\it Step 2b. estimates on $\mathrm{II}^{b,\sigma}$} : The arguments presented below are similar as in Step 2a. Set
	\[
	\Psi_T^{b,\sigma;f}:= 1+\lvert X_T^{H^*;b,\sigma} \rvert^r+\lvert X_T^{H^*;o} \rvert^r+\lvert l(S_T^{b,\sigma}) \rvert^r+\lvert l(S_T^o) \rvert^r.
	\]
	Then, H\"older's inequality with exponent $\frac{p-2}{r}>1$ (and the power triangle inequality) and the a priori estimates given in Lemma \ref{lem:p_integrability}\;(ii) show that for every $\varepsilon\in[0,1]$,
	\begin{eqnarray}\label{eq:V_estimate}
		\sup_{\varepsilon\in[0,1]}\sup_{(b,\sigma)\in\mathcal{B}^\varepsilon}\lVert \Psi_T^{b,\sigma;f} \rVert_{L^{\frac{p}{p-2}}}<\infty.
	\end{eqnarray}
	
	Moreover, by Assumption \ref{as:objective}\;(ii) we have that for every $\omega\in\Omega$ and $x,\hat x,y,\hat y\in\mathbb{R}$,
	\[\lVert D_{x,y}^2 f(\omega,x+\hat x,y+\hat y) \rVert_{\operatorname{F}}\leq \overline C_2 2^{r}(1+\lvert x \rvert^r+\lvert \hat y \rvert^r+\lvert x \rvert^r+\lvert \hat y \rvert^r),\]
	hence it follows that for every $(b,\sigma)\in {\cal B}^\varepsilon$ and $\varepsilon\in[0,1]$, 
	\begin{align}\label{eq:estimate_II1}
		\begin{aligned}
			\mathbb{E}\left[ \lvert \operatorname{II}^{b,\sigma} \rvert \right]
			&\leq \overline C_2 2^{r} \cdot  \mathbb{E}\Big[\Psi_T^{b,\sigma;f}\Big(\big\lvert \Delta X^{*;b,\sigma}_T \big \rvert^2+ \big\lvert \Delta l^{b,\sigma}_T \big\rvert^2\Big) \Big]
			\\
			&\leq  \overline C_2 2^{r} \cdot 
			\lVert \Psi_T^{b,\sigma;f} \rVert_{L^{\frac{p}{p-2}}} \Big( \big\lVert \Delta X^{*;b,\sigma}_T \big \rVert_{L^p}^2+\big\lVert \Delta l^{b,\sigma}_T \big \rVert_{L^p}^2 \Big),
		\end{aligned}
	\end{align}
	where we used  H\"older's inequality (with exponent $\frac{p}{2}>1$) in the second step.
	
	
	Since $\sup_{(b,\sigma)\in \mathcal{B}^\varepsilon}\lVert \Delta X^{*;b,\sigma}_T \rVert_{L^p}^2\leq C_{1}^{\frac{2}{p}} \varepsilon^2$ (see Lemma \ref{lem:p_integrability}\;(i)) and  $\lvert\nabla_s l (s) -  \nabla_s l (\hat{s}) \rvert  \leq C_{l}\lvert s-\hat{s} \rvert$ for every $s,\hat{s}\in \mathbb{R}^d$ (see Remark \ref{rem:objective}\;(ii)), it follows that for every $\varepsilon\in[0,1]$, 
	\begin{align*}
		\sup_{(b,\sigma)\in \mathcal{B}^\varepsilon}\big\lVert \Delta l^{b,\sigma}_T \big \rVert_{L^p}^2 =\sup_{(b,\sigma)\in \mathcal{B}^\varepsilon}\big\lVert l(S_T^{b,\sigma})-l(S_T^{o}) \big \rVert_{L^p}^2  \leq C_l^2 \sup_{(b,\sigma)\in \mathcal{B}^\varepsilon}\big\lVert \Delta S_T^{b,\sigma} \big\rVert_{L^p}^2 \leq C_l^2 C_{2}^{\frac{2}{p}} \varepsilon^2.
	\end{align*}
	
	Combined with \eqref{eq:V_estimate} and \eqref{eq:estimate_II1}, this ensures that there is some $C_{\operatorname{II}}>0$ such that for every $\varepsilon\in[0,1]$, $\sup_{(b,\sigma)\in \mathcal{B}^\varepsilon}\mathbb{E}[\lvert\operatorname{II}^{b,\sigma}\lvert\lvert]\leq C_{\operatorname{II}} \varepsilon^2$.

	\vspace{0.5em}
	\noindent
	{\it Step 2c. estimates on $\mathrm{III}^{b,\sigma}$} : Since $\lVert D^2_sl (\cdot) \rVert_{\operatorname{F}}\leq C_l$ and $\lvert \partial_{y}f(\omega,x,y) \rvert\leq  \widetilde{C} \left(1  +\lvert {x} \rvert^{r+1} +\lvert y \rvert^{r+1} \right)$ for every $\omega\in\Omega$ and $x,y\in \mathbb{R}$ (see Remark \ref{rem:objective}), it follows that for every $\varepsilon\in[0,1]$ and $(b,\sigma )\in {\cal B}^\varepsilon$,  
	\begin{align}\label{eq:estim_I0}
		\begin{aligned}
			\mathbb{E} \left[ \lvert\operatorname{III}^{b,\sigma}\rvert \right]
			&\leq C_l \widetilde{C} \mathbb{E} \left[ \big(1+ \lvert X_T^{H^*;o} \rvert^{r+1}+\lvert l(S_T^o) \rvert^{r+1}\big) \big\lvert \Delta S_T^{b,\sigma}  \big\rvert^2 \right]\\
			&\leq C_l\widetilde{C}  \Big\lVert 1+ \lvert X_T^{H^*;o} \rvert^{r+1}+\lvert l(S_T^o) \rvert^{r+1} \Big\lVert_{L^{\frac{p}{p-2}}} \big\lVert \Delta S_T^{b,\sigma}  \big\rVert_{L^p}^{2},
		\end{aligned}
	\end{align}
	where the second inequality follows from  H\"older's inequality (with exponent $\frac{p}{2}>1$).
    
    By Lemma~\ref{lem:p_integrability}\;(i) we have that  for every $\varepsilon\in[0,1]$ and $(b,\sigma)\in \mathcal{B}^\varepsilon$, $\lVert \Delta S_T^{b,\sigma}  \rVert_{L^p}^{2}
    	\leq  C_{2}^{\frac{2}{p}} \varepsilon^2$.
    	Moreover, by the power triangle inequality,
	\begin{align*}
		\mathbb{E}\left[ \left(1+ \lvert X_T^{H^*;o} \rvert^{r+1}+\lvert l(S_T^o) \rvert^{r+1} \right)^{\frac{p}{p-2}}\right]
		\leq 3^{\frac{p}{p-2}} \left(1+ \mathbb{E}\Big[ \lvert X_T^{H^*;o} \rvert^{\frac{(r+1)p}{p-2}}\Big]+ \mathbb{E}\Big[ \lvert l(S_T^o) \rvert^{\frac{(r+1)p}{p-2}}\Big]\right)<\infty,
	\end{align*}
    where the last inequality follows from Lemma \ref{lem:p_integrability}\;(ii) (noting that $\frac{(r+1)p}{p-2}\leq p$).  
    
      Combined with \eqref{eq:estim_I0}, we conclude that there is some $C_{\operatorname{I}}>0$ such that  for every $\varepsilon\in[0,1]$, $\sup_{(b,\sigma)\in \mathcal{B}^\varepsilon}\mathbb{E}[\lvert\operatorname{III}^{b,\sigma}\rvert] \leq C_{\operatorname{III}}\varepsilon^2$.
	

	\vspace{0.5em}
	\noindent
	{\it Step 3.}
	For any $\varepsilon\in[0,1]$, set 
	\[
	\Phi(\varepsilon):= \sup_{(b,\sigma)\in\mathcal{B}^\varepsilon} \mathbb{E}\bigg[\int_0^T R_t^*\Delta X_t^{*;b,\sigma} dt + A^*  \Delta X_T^{*;b,\sigma}+\big(B^* \big)^\top  \Delta S_T^{b,\sigma} \bigg].
	\]
	Then, by \eqref{eq:taylor_1} and \eqref{eq:estimate_I}, for  every $\varepsilon\in[0,1]$,
	\begin{align}\label{eq:taylor1}
			0\leq V(\varepsilon)-V(0) 
			\leq \Phi(\varepsilon)+ (T C_{\operatorname{I}}+C_{\operatorname{II}}+C_{\operatorname{III}})\varepsilon^2.
	\end{align}

	It remains to show that for every $\varepsilon\in[0,1]$,
	\begin{align*}
		\Phi(\varepsilon)\leq \varepsilon \left( \gamma \lVert Y^*H^*+ {{\cal Y}}^* \rVert_{\mathbb{L}^q} + \eta \lVert Z^*(H^*)^\top + {{\cal Z}}^*\rVert_{\mathbb{H}^q}  \right).
	\end{align*}
	To that end, we note that by Lemma \ref{lem:inner_prod},
	\begin{align}\label{eq:up_bound2}
		\Phi(\varepsilon)= \sup_{(b,\sigma)\in\mathcal{B}^\varepsilon}  \Big( \langle Y^* H^*+{{\cal Y}}^*, b-b^o \rangle_{\mathbb{P}\otimes dt}+ \langle Z^*(H^*)^\top+{{\cal Z}}^*, \sigma-\sigma^o \rangle_{ \mathbb{P}\otimes dt,\operatorname{F}}  \Big).
	\end{align}
	Set $\mathcal{B}^{\varepsilon}_1:= \{b:(b,\sigma)\in \mathcal{B}^\varepsilon\}$ and $\mathcal{B}^{\varepsilon}_2:= \{\sigma:(b,\sigma)\in \mathcal{B}^\varepsilon\}$ so that $\mathcal{B}^{\varepsilon}=\mathcal{B}^{\varepsilon}_1\times \mathcal{B}^{\varepsilon}_2$  by the definition of $\mathcal{B}^\varepsilon$.
	It follows from the Cauchy-Schwartz inequality and H\"older's inequality (with exponent $p>3$) that for every $\varepsilon \in [0,1]$ and $b\in \mathcal{B}^{\varepsilon}_1$, 
	\begin{eqnarray*}\label{eq:dual_inner0}
		\begin{aligned}
			\langle Y^*H^*+{{\cal Y}}^*, b-b^o \rangle_{\mathbb{P}\otimes dt} 
			&\leq  \mathbb{E}\Bigg[\int_0^T  \rvert b_t-b_t^o \lvert  \lvert Y_{t}^*H_t^*+{{\cal Y}}_t^* \rvert dt\Bigg]\\
			&\leq  \lVert  b-b^o \rVert_{\mathbb{L}^p} \lVert Y^*H^*+{{\cal Y}}^* \rVert_{\mathbb{L}^q}.
		\end{aligned}
	\end{eqnarray*}
    Hence for every $\varepsilon \in [0,1]$ we have that
    \begin{align}\label{eq:dual_inner1}
    	\begin{aligned}
    		\sup_{b\in \mathcal{B}^{\varepsilon}_1} \langle Y^*H^*+{{\cal Y}}^*, b-b^o \rangle_{\mathbb{P}\otimes dt} 
    		&\leq \varepsilon \gamma \lVert Y^*H^*+{{\cal Y}}^* \rVert_{\mathbb{L}^q}.
    	\end{aligned}
    \end{align}
	
	Similarly, using Cauchy-Schwartz inequality for $\| \cdot \|_{\operatorname{F}}$ and H\"older's inequality (with exponent~2), it follows that for every $\varepsilon \in [0,1]$ and $\sigma\in \mathcal{B}^{\varepsilon}_2$, 
	\begin{align*}
		\begin{aligned}
			\langle Z^*(H^*)^\top+{{\cal Z}}^*, \sigma-\sigma^o \rangle_{ \mathbb{P}\otimes dt,\operatorname{F}}
			&\leq \mathbb{E}\Bigg[\int_0^T  \lVert \sigma_t-\sigma_t^o \rVert_{\operatorname{F}}  \lVert Z_{t}^*(H_t^*)^{\top} +{{\cal Z}}_t^*\rVert_{\operatorname{F}} dt\Bigg]\\
			&\leq  \mathbb{E}\Bigg[\bigg(\int_0^T  \lVert \sigma_t-\sigma_t^o \rVert_{\operatorname{F}}^2 dt \bigg)^{\frac{1}{2}} \bigg(\int_0^T \lVert Z_{t}^*(H_t^*)^{\top}+{{\cal Z}}_t^* \rVert_{\operatorname{F}}^2 dt \bigg)^{\frac{1}{2}}\Bigg].\\
		\end{aligned}
	\end{align*}
    Therefore, another application of H\"older's inequality shows that for every $\varepsilon \in [0,1]$,
    \begin{eqnarray}\label{eq:dual_inner3}
    	\begin{aligned}
    		\sup_{\sigma \in \mathcal{B}^{\varepsilon}_2}\langle Z^*(H^*)^\top+{{\cal Z}}^*, \sigma-\sigma^o \rangle_{ \mathbb{P}\otimes dt,\operatorname{F}}
    		&\leq \sup_{\sigma \in \mathcal{B}^{\varepsilon}_2}\lVert \sigma-\sigma^o \rVert_{\mathbb{H}^p}\lVert Z^*(H^*)^\top +{{\cal Z}}^*\rVert_{\mathbb{H}^q}\\
    		&\leq \varepsilon \eta \lVert Z^*(H^*)^\top +{{\cal Z}}^*\rVert_{\mathbb{H}^q}.
    	\end{aligned}
    \end{eqnarray}
	Combining \eqref{eq:dual_inner1} and \eqref{eq:dual_inner3} with \eqref{eq:taylor1}  concludes the proof.
\end{proof}

\begin{rem}\label{rem:auxiliary}
	Recall $V^*(\varepsilon)$ defined in \eqref{eq:value_worst}. 
	The proof of Lemma \ref{lem:up_bound} actually shows that under the same assumptions as therein, for any $\varepsilon\in[0,1]$,
	\begin{align*}
	0\leq V^*(\varepsilon) - V^*(0) \leq \varepsilon \left(\gamma  \lVert Y^*H^*+ {{\cal Y}}^* \rVert_{\mathbb{L}^q} + \eta \lVert Z^*(H^*)^\top + {{\cal Z}}^*\rVert_{\mathbb{H}^q}  \right)+C \varepsilon^2.
\end{align*} 	
This will be used later in the proof of Theorem \ref{thm:main2}.
\end{rem}

From Lemma \ref{lem:up_bound}, we can deduce the upper bound 
\begin{align}
\label{eq:upper.bound.derivative}
 V'(0) 
\leq \gamma  \lVert Y^*H^*+{{\cal Y}}^* \rVert_{\mathbb{L}^q} + \eta \lVert Z^*(H^*)^\top +{{\cal Z}}^* \rVert_{\mathbb{H}^q},
\end{align}
 where $H^*$ is the unique optimizer for $V(0)$ and  $(Y^*,Z^*,{{\cal Y}}^*,{{\cal Z}}^*)$ are defined in Proposition \ref{pro:main0}. 

To prove the corresponding lower bound in \eqref{eq:upper.bound.derivative}, let us consider  \emph{$\varepsilon^2$-optimizers} $H^{\varepsilon}\in {\cal A}$ of $V(\varepsilon)$, i.e.,
\begin{align}\label{eq:en_optimizer}
	\begin{aligned}
	V(\varepsilon) = \inf_{H\in {\cal A}} {\cal V}(H,\varepsilon)&>{\cal V}(H^{\varepsilon},\varepsilon)-\varepsilon^2   \\
	& = \sup_{(b,\sigma) \in \mathcal{B}^\varepsilon} \mathbb{E}\bigg[\int_0^T g\big(t,X_t^{H^{\varepsilon};b,\sigma},H^{\varepsilon}_t\big)dt + f \big(X_T^{H^{\varepsilon};b,\sigma},l(S_T^{b,\sigma})\big) \bigg]-\varepsilon^2.
	\end{aligned}
\end{align}
The lower bound follows from the following lemma, together with an additional result which implies that $H^\varepsilon$ converges to $H^\ast$ in a suitable sense.

\begin{lem}\label{lem:low_bound}
Suppose that Assumptions \ref{as:posterior_refer}, \ref{dfn:control}, and \ref{as:objective} are satisfied. 
For any $\varepsilon\in (0,1]$,  let $H^\varepsilon \in {\cal A}$ be an $\varepsilon^2$-optimizer of $V(\varepsilon)$. Set 
	\begin{align*}
	&A^{\varepsilon}:=A^{H^\varepsilon}=\partial_{x}f(X_T^{H^\varepsilon;o},h(S_T^o)),\qquad R^\varepsilon:=R^{H^\varepsilon}=(\partial_x g(t,X_t^{H^\varepsilon;o},H_t^\varepsilon))_{t\in[0,T]},\\
	&B^{\varepsilon}:=B^{H^\varepsilon}=\partial_{y}f(X_T^{H^\varepsilon;o},l(S_T^o))\nabla_s l(S_T^o),
	\end{align*}
	and let 
	\begin{align}
	\label{eq:BSDE_opti_ep_x}
		(Y^\varepsilon,Z^\varepsilon,L^\varepsilon)
		&\in\mathscr{S}^{2}(\mathbb{R}) \times \mathscr{H}^{2}(\mathbb{R}^d) \times \mathscr{M}^2(\mathbb{R}),
		\\
		\label{eq:BSDE_opti_ep_y2}
		({\cal Y}^\varepsilon,{\cal Z}^\varepsilon,{\cal L}^\varepsilon)
		&\in(\mathscr{S}^{2}(\mathbb{R}))^d \times (\mathscr{H}^{2}(\mathbb{R}^d))^{d} \times (\mathscr{M}^2(\mathbb{R}))^{d}
	\end{align}
	be the unique solutions of \eqref{eq:BSDE_univ1} with the terminal condition and generator $(A^{\varepsilon},(R_t^\varepsilon)_{t\in[0,T]})$ and  \eqref{eq:BSDE_univ3} with the terminal condition $B^{\varepsilon},$ respectively (ensured by Proposition \ref{pro:BSDE}).
	Then, 
	\[
	V(\varepsilon)-V(0) \geq \varepsilon \Big( \gamma \lVert Y^\varepsilon H^\varepsilon +{{\cal Y}}^\varepsilon \rVert_{\mathbb{L}^q} +\eta \lVert Z^\varepsilon(H^\varepsilon)^\top +  {{\cal Z}}^\varepsilon\rVert_{\mathbb{H}^q}\Big) -{C}_{\operatorname{res}}\varepsilon^2,
	\]
	where ${C}_{\operatorname{res}}:=TC_{\operatorname{I}}+C_{\operatorname{II}}+C_{\operatorname{III}}+3 $ and $C_{\operatorname{I}},C_{\operatorname{II}},,C_{\operatorname{III}}>0$ are given in \eqref{eq:estimate_I}.
\end{lem}
\begin{proof}
	It is obvious that 
	\begin{align}\label{eq:ep_optimizer0}
		V(\varepsilon)-V(0)\geq {\cal V}(H^\varepsilon,\varepsilon)- {\cal V}(H^\varepsilon,0)-\varepsilon^2.
	\end{align}
	Furthermore, from the second-order Taylor expansions of $g$ and $f$ given in \eqref{eq:taylor_01} and \eqref{eq:taylor_1} and  the a priori estimates given in \eqref{eq:estimate_I} with $C_{\operatorname{I}},C_{\operatorname{II}},C_{\operatorname{III}}>0$ (replacing $H^*$ by $H^\varepsilon$, see the proof of Lemma \ref{lem:up_bound}), it follows that for every $\varepsilon\in(0,1]$
	\begin{align}\label{eq:ep_optimizer} 
		\begin{aligned}
			 {\cal V}(H^\varepsilon,\varepsilon)- {\cal V}(H^\varepsilon,0)
			 &=\sup_{(b,\sigma)\in\mathcal{B}^\varepsilon} \mathbb{E}\Bigg[ \int_0^T \Big(g\big(t,X_t^{H^{\varepsilon};b,\sigma},H^{\varepsilon}_t\big)dt-g\big(t,X_t^{H^{\varepsilon};o},H^{\varepsilon}_t\big)dt \Big) dt  \\ 
			 &\hspace{6.em}+f\big(X_T^{H^\varepsilon;b,\sigma},h(S_T^o)\big)- f\big(X_T^{H^\varepsilon;o},h(S_T^o)\big) \Bigg] \\
			&\geq \sup_{(b,\sigma)\in\mathcal{B}^\varepsilon} \mathbb{E}\Bigg[\int_0^T R_t^\varepsilon (X_t^{H^\varepsilon;b,\sigma}-X_t^{H^\varepsilon;o})dt+A^{\varepsilon} \big(X_T^{H^\varepsilon;b,\sigma} - X_T^{H^\varepsilon;o}\big) \\
			&\hspace{6.em}+ \big( B^{\varepsilon} \big)^\top \big(S_T^{b,\sigma} - S_T^{o}\big) \Bigg]-(TC_{\operatorname{I}}+C_{\operatorname{II}}+C_{\operatorname{III}})\varepsilon^2.
		\end{aligned}
	\end{align}
	
	Next, note that for $(Y^\varepsilon,Z^\varepsilon)$ and $({{\cal Y}}^\varepsilon,{{\cal Z}}^\varepsilon)$ given in \eqref{eq:BSDE_opti_ep_x} and \eqref{eq:BSDE_opti_ep_y2}, Lemma \ref{lem:inner_prod} implies that
	\begin{align}\label{eq:ep_optimizer1}
		\begin{aligned}
			&\sup_{(b,\sigma)\in\mathcal{B}^\varepsilon} \mathbb{E}\Bigg[\int_0^T R_t^\varepsilon (X_t^{H^\varepsilon;b,\sigma}-X_t^{H^\varepsilon;o})dt+ A^\varepsilon \big(X_T^{H^\varepsilon;b,\sigma} - X_T^{H^\varepsilon;o}\big)+ (B^\varepsilon)^\top \big(S_T^{b,\sigma} - S_T^{o}\big) \Bigg]\\
			&\quad =\sup_{(b,\sigma) \in \mathcal{B}^\varepsilon} \Big( \langle Y^\varepsilon H^\varepsilon+{{\cal Y}}^\varepsilon, b-b^o \rangle_{\mathbb{P}\otimes dt}+ \langle Z^\varepsilon (H^\varepsilon)^\top+{{\cal Z}}^\varepsilon, \sigma-\sigma^o \rangle_{ \mathbb{P}\otimes dt,\operatorname{F}} \Big).
		\end{aligned}
	\end{align}

	\noindent By \cite[Remark 5.3, p.137]{HWY2019}, the left limit process $Y_-^\varepsilon:=(Y^\varepsilon_{t-})_{t\in[0,T]}$ defined by $Y_{t-}^\varepsilon:= \lim_{s\uparrow t}Y_s^\varepsilon$, $t\in(0,T]$ and $Y_{0-}^\varepsilon:=Y_{0}^\varepsilon$, is the $\mathbb{F}$-predictable projection of $Y^\varepsilon$. Moreover, as $Y^\varepsilon$ is c\`adl\`ag, $Y_t^\varepsilon=Y_{t-}^\varepsilon$ $\mathbb{P}\otimes dt$-a.e.. Using the same notation and arguments, it follows that ${\cal Y}_t^\varepsilon={\cal Y}_{t-}^\varepsilon$ $\mathbb{P}\otimes dt$-a.e.. Therefore, we can invoke the duality between $\mathbb{F}$-predictable spaces $\mathbb{L}^q(\mathbb{R}^d)$ and $\mathbb{L}^p(\mathbb{R}^d)$ 
	to obtain some $\widetilde{b}^{\varepsilon} \in \mathbb{L}^p(\mathbb{R}^d)$ that satisfies $\lVert \widetilde{b}^{\varepsilon} \rVert_{\mathbb{L}^p}=1$ and
	\begin{eqnarray}\label{eq:dual_Lp}
		\begin{aligned}
			\lVert Y^\varepsilon H^{\varepsilon}+{{\cal Y}}^\varepsilon  \rVert_{\mathbb{L}^q} &= \lVert Y_-^\varepsilon H^{\varepsilon}+{{\cal Y}}_-^\varepsilon  \rVert_{\mathbb{L}^q}\\ &=\sup_{\lVert \widetilde{b} \rVert_{\mathbb{L}^p} =1} \langle Y_-^\varepsilon H^\varepsilon+{{\cal Y}}_-^\varepsilon  , \widetilde{b} \rangle_{\mathbb{P}\otimes dt} \\
			&\leq  \langle Y_-^\varepsilon H^\varepsilon+{{\cal Y}}_-^\varepsilon, \widetilde{b}^{\varepsilon} \rangle_{\mathbb{P}\otimes dt} +\varepsilon= \langle Y^\varepsilon H^\varepsilon+{{\cal Y}}^\varepsilon, \widetilde{b}^{\varepsilon} \rangle_{\mathbb{P}\otimes dt} +\varepsilon.
		\end{aligned}
	\end{eqnarray}
	
	In a similar manner, since $Z^\varepsilon(H^\varepsilon)^\top+{{\cal Z}}^\varepsilon$ is $\mathbb{F}$-predictable (because $H^\varepsilon \in {\cal A}$, $Z^\varepsilon \in \mathscr{H}^2(\mathbb{R})$ and $\{{\cal Z}^{\varepsilon,i}\}_{i=1,\dots,d} \subseteq \mathscr{H}^2(\mathbb{R})$), we may invoke the duality between 	$\mathbb{H}^q (\mathbb{R}^{d\times d})$ and  $\mathbb{H}^{p}(\mathbb{R}^{d\times d})$ (see \cite[Theorems 1.3.10 \& 1.3.21]{Hytonen2016}) to obtain some $\widetilde{\sigma}^{\varepsilon}\in \mathbb{H}^p(\mathbb{R}^{d\times d})$ that satisfies $\lVert \widetilde{\sigma}^{\varepsilon} \rVert_{\mathbb{H}^p}=1$~and
	\begin{eqnarray}\label{eq:dual_Hp}
		\begin{aligned}
			\lVert Z^\varepsilon(H^\varepsilon)^\top +{{\cal Z}}^\varepsilon\rVert_{\mathbb{H}^q} &= \sup_{\lVert \widetilde{\sigma} \rVert_{\mathbb{H}^p} =1} \langle Z^\varepsilon(H^\varepsilon)^\top+{{\cal Z}}^\varepsilon, \widetilde{\sigma} \rangle_{\mathbb{P}\otimes dt,\operatorname{F}} \\
			&\leq  \langle Z^\varepsilon(H^\varepsilon)^\top+{{\cal Z}}^\varepsilon, \widetilde{\sigma}^{\varepsilon} \rangle_{\mathbb{P}\otimes dt,\operatorname{F}} +\varepsilon.
		\end{aligned}
	\end{eqnarray}
	
	Finally, define $(b^{\star,\varepsilon},\sigma^{\star,\varepsilon}) \in \mathcal{B}^\varepsilon$ by
	\[
	b_t^{\star,\varepsilon} := b^o_t + \varepsilon \gamma \widetilde{b}_t^\varepsilon,\quad \sigma_t^{\star,\varepsilon} := \sigma_t^o + \varepsilon \eta  \widetilde{\sigma}_t^\varepsilon,\quad t\in [0,T],
	\]
	and set $C_{\operatorname{res},1}:=(TC_{\operatorname{I}}+C_{\operatorname{II}}+C_{\operatorname{III}})+1$. Then 
	\eqref{eq:ep_optimizer0}-\eqref{eq:dual_Hp} imply that
	\begin{align*}
		\begin{aligned}
			V(\varepsilon) -V(0)
			&\geq \Big( \langle Y^\varepsilon H^\varepsilon+{{\cal Y}}^\varepsilon, b^{\star,\varepsilon}-b^o \rangle_{\mathbb{P}\otimes dt}+ \langle Z^\varepsilon (H^\varepsilon)^\top+{{\cal Z}}^\varepsilon, \sigma^{\star,\varepsilon}-\sigma^o \rangle_{ \mathbb{P}\otimes dt,\operatorname{F}} \Big)-C_{\operatorname{res},1} \varepsilon^2\\
			&= \varepsilon \Big( \gamma \langle Y^\varepsilon H^\varepsilon+{{\cal Y}}^\varepsilon, \widetilde{b}^\varepsilon \rangle_{\mathbb{P}\otimes dt} + \eta \langle Z^\varepsilon (H^\varepsilon)^\top+{{\cal Z}}^\varepsilon, \widetilde{\sigma}^\varepsilon \rangle_{\mathbb{P}\otimes dt,\operatorname{F}} \Big)-C_{\operatorname{res},1} \varepsilon^2 \\
			&\geq \varepsilon \Big( \gamma \lVert Y^\varepsilon H^\varepsilon+ {{\cal Y}}^\varepsilon \rVert_{\mathbb{L}^q} +\eta \lVert Z^\varepsilon (H^\varepsilon)^\top+{{\cal Z}}^\varepsilon \rVert_{\mathbb{H}^q}  \Big) -(C_{\operatorname{res},1}+\gamma+\eta)\varepsilon^2.
		\end{aligned}
	\end{align*}
    This completes the proof. 
\end{proof}

The final ingredient in the proof of Theorem \ref{thm:main} is the following stability result.

\begin{lem}\label{lem:strong_H}
	Suppose that Assumptions \ref{as:posterior_refer}, \ref{dfn:control}, and \ref{as:objective} are satisfied.
	Let $(\varepsilon_n)_{n\in \mathbb{N}}\subseteq (0,1]$ with $\lim_{n }\varepsilon_n=0$ be such that $\lim_{n} \frac{V(\varepsilon_n)-V(0)}{\varepsilon_n}=\liminf_{\varepsilon \downarrow 0} \frac{V(\varepsilon)-V(0)}{\varepsilon}$. Moreover, let $H^*$ be the optimizer for $V(0)$ (see Proposition \ref{pro:main0}\;(i)). Then for any sequence of $\varepsilon_n^2$-optimizers $H^{\varepsilon_n}$, the following hold:
	\begin{itemize}
		\item [(i)] For every $\beta \geq 1$,
		\begin{eqnarray*}\label{eq:strong_H_eps}
			\lVert H_t^{\varepsilon_n}-H_t^*\rVert_{\mathbb{L}^\beta} \rightarrow 0\quad \mbox{\mbox{as $n\rightarrow \infty$}}.
		\end{eqnarray*}
		\item [(ii)] Denote for each $n\in\mathbb{N}$ by $(Y^{\varepsilon_n},Z^{\varepsilon_n},L^{\varepsilon_n})$ the unique solution of \eqref{eq:BSDE_univ1} with the~terminal condition and generator $(A^{H^{\varepsilon_n}},R^{H^{\varepsilon_n}})=(\partial_{x}f(X_T^{H^{\varepsilon_n};o},l(S_T^o)),(\partial_x g(t,X_t^{H^{\varepsilon_n}},H^{\varepsilon_n}_t))_{t\in[0,T]})$, and by $(Y^{*},Z^{*},L^{*})$ the unique solution of \eqref{eq:BSDE_re_x}. Then, as $n\rightarrow \infty$,
		\begin{eqnarray*}\label{eq:strong_BSDE_eps_re_x}
			\lVert Y^{\varepsilon_n} - Y^* \rVert_{\mathscr{S}^2} + \lVert Z^{\varepsilon_n} - Z^* \rVert_{\mathscr{H}^2}+\lVert L^{\varepsilon_n}- L^* \rVert_{\mathscr{M}^2} \rightarrow 0.     	
		\end{eqnarray*}
		\item [(iii)]
		Denote for each $n\in\mathbb{N}$ by $
		({\cal Y}^{\varepsilon_n},{\cal Z}^{\varepsilon_n},{\cal L}^{\varepsilon_n})$
		the unique solution of \eqref{eq:BSDE_univ3} with the terminal condition $B^{H^{\varepsilon_n}}=\partial_{y}f(X_T^{H^{\varepsilon_n};o},h(S_T^o))\nabla_s l(S_T^o)$, and by $({\cal Y}^*,{\cal Z}^* ,{\cal L}^*)$ the unique solution of \eqref{eq:BSDE_re_y}. Then, for every $i=1,\dots,d$, as $n\rightarrow \infty$,
		\begin{eqnarray*}\label{eq:strong_BSDE_eps_re_y}
			\quad\quad\lVert {\cal Y}^{\varepsilon_n,i} - {\cal Y}^{*,i} \rVert_{\mathscr{S}^2} + \lVert {\cal Z}^{\varepsilon_n,i} - {\cal Z}^{*,i} \rVert_{\mathscr{H}^2}+\lVert {\cal L}^{\varepsilon_n,i}- {\cal L}^{*,i} \rVert_{\mathscr{M}^2} \rightarrow 0. 
		\end{eqnarray*}
	\end{itemize}
\end{lem}
\begin{proof}
	We start by proving (i). Since $H^*$ is optimal for $V(0)$ and $H^{{\varepsilon}_n}$ is $\varepsilon_n^2$-optimal for $V(\varepsilon_n)$ (see \eqref{eq:en_optimizer}). Then for every $t\in[0,T]$ set
	\begin{align*}
		\Delta X_t^{\varepsilon_n}:=X_t^{H^{\varepsilon_n};o}-X_t^{H^{*};o},\qquad \Delta H^{\varepsilon_n}_t :=H^{\varepsilon_n}_t-H^{*}_t.
	\end{align*}
	The second-order Taylor expansions of $g(t,\,\cdot\,,\,\cdot)$ around $(X_t^{H^*;o},H_t^*)$ for each $t\in[0,T]$ and of $f(\cdot\,,l(S_T^o))$ around $X_T^{H^*;o}$ show that
	\begin{align*}
		\begin{aligned}
			&V(\varepsilon_n) - V(0) + \varepsilon_n^2 \\
			&\;  \geq \mathbb{E}\bigg[\int_0^T\Big( g(t,X_t^{H^{\varepsilon_n};o},H^{\varepsilon_n}_t)-g(t,X_t^{H^{*};o},H^{*}_t)\Big)dt+f\big(X_T^{H^{\varepsilon_n};o},l(S_T^o)\big)-f\big(X_T^{H^*;o},l(S_T^o)\big) \bigg] \\
			&\; =   \mathbb{E}\bigg[\int_0^T (\nabla_{x,h}g(t,X_t^{H^{*};o},H^{*}_t))^\top 
			\begin{pmatrix}
				\Delta X_t^{\varepsilon_n}\\
				 \Delta H^{\varepsilon_n}_t
			\end{pmatrix}dt  + \partial_{x}f\big(X_T^{H^*;o},l(S_T^o)\big) \Delta X_T^{\varepsilon_n} + \int_0^T \operatorname{IV}_tdt  +\operatorname{V} \bigg],
		\end{aligned}
	\end{align*}
	where for each $t\in[0,T]$ $\operatorname{IV}_t$ and $\operatorname{V}$ are given by 
	\begin{align*}
		\operatorname{IV}_t&:= (\Delta X_t^{\varepsilon_n},(\Delta H^{\varepsilon_n}_t)^\top)\int_0^1 (1-\theta) D_{x,h}^2g\big(t,X_t^{H^*;o}+\theta \Delta X_t^{\varepsilon_n} , H_t^*+\theta \Delta H_t^{\varepsilon_n}\big) d\theta \begin{pmatrix}
			\Delta X_t^{\varepsilon_n}\\
			\Delta H^{\varepsilon_n}_t
		\end{pmatrix},\\
		\operatorname{V}&:= \big\lvert \Delta X_T^{\varepsilon_n}\big\rvert^2 \int_0^1 (1-\theta) \partial_{xx} f\left(X_T^{H^*;o}+\theta \Delta X_T^{\varepsilon_n},l(S_T^o) \right)d\theta.
	\end{align*}
	By the convexity of $g$ in $(x,h)$ and the strong convexity of $f$ in $x$ (see Assumption \ref{as:objective}\;(iv)), 
	it holds that $\int_0^T \operatorname{IV}_tdt \geq 0$ and $\operatorname{V}\geq  \frac{1}{2}\underline C_{2}\big\lvert \Delta X_T^{\varepsilon_n}\big\rvert^2$.
	
	Moreover,  using the first-order optimality of $H^*$ (see Lemma \ref{lem:FOC}),
	\[
	V(\varepsilon_n) - V(0)\geq \frac{1}{2}\underline C_{2}\mathbb{E}\Big[ \lvert \Delta X_T^{\varepsilon_n} \rvert^2 \Big] -\varepsilon_n^2.
	\]
	
	Since $V(\varepsilon_n)\rightarrow V(0)$ as $n\rightarrow \infty$ (by Lemma \ref{lem:up_bound}), we have that $\mathbb{E}[ \lvert  \Delta X_T^{\varepsilon_n} \rvert^2]\rightarrow 0$ as $n\rightarrow \infty$; in particular
	\[
		\big\lvert  \Delta X_T^{\varepsilon_n}  \big\rvert=\big\lvert X_T^{H^{\varepsilon_n};o} - X_T^{H^*;o} \big\rvert\xrightarrow[]{\mathbb{P}} 0\quad  \mbox{as $n\rightarrow \infty$}.
	\]
	Thus, an application of Lemma \ref{lem:stability} implies (i). 
	
	Finally, Lemma \ref{lem:stability_1} ensures that (ii) and (iii) hold, which completes the proof.
\end{proof}

\begin{proof}[Proof of Theorem \ref{thm:main}] 
	First note that Lemma \ref{lem:up_bound} immediately implies the upper bound
	\[
	\limsup_{\varepsilon\downarrow 0} \frac{V(\varepsilon)-V(0)}{\varepsilon}\leq \gamma \lVert Y^*H^*+ {{\cal Y}}^* \rVert_{\mathbb{L}^q} + \eta \lVert Z^*(H^*)^\top +  {{\cal Z}}^*\rVert_{\mathbb{H}^q}.
	\]
	Thus, all that is left to do is to prove the corresponding lower bound.
	That will be achieved in three steps.
	
	\vspace{0.5em}
	\noindent
	\emph{Step 1:}	
	Let $(\varepsilon_n)_{n\in \mathbb{N}} \subseteq (0,1]$ satisfy that $\lim_{n}\varepsilon_n=0$ and $\lim_{n} \frac{V(\varepsilon_n)-V(0)}{\varepsilon_n}=\liminf_{\varepsilon \downarrow 0} \frac{V(\varepsilon)-V(0)}{\varepsilon}$.  Moreover, for each $n\in\mathbb{N}$, let $H^{\varepsilon_n}$ be an $\varepsilon_n^2$-optimizer of $V(\varepsilon_n)$. Let $(Y^{\varepsilon_n},Z^{\varepsilon_n},L^{\varepsilon_n})$ and $({\cal Y}^{\varepsilon_n},{\cal Z}^{\varepsilon_n},{\cal L}^{\varepsilon_n})$ be the unique solutions of \eqref{eq:BSDE_univ1} and \eqref{eq:BSDE_univ3}, 	respectively, with $(A^{H^{\varepsilon_n}},R^{H^{\varepsilon_n}})=(\partial_{x}f(X_T^{H^{\varepsilon_n};o},l(S_T^o)),(\partial_x g(t,X_t^{H^{\varepsilon_n};o},H_t^{\varepsilon_n}))_{t\in[0,T]})$ and $B^{H^{\varepsilon_n}}=\partial_{y}f(X_T^{H^{\varepsilon_n};o},l(S_T^o))\nabla_s l(S_T^o)$.
	
	Then, by Lemma \ref{lem:low_bound},
	\begin{align*}
		\begin{aligned}
			\lim_{n\rightarrow \infty} \frac{V(\varepsilon_n)-V(0)}{\varepsilon_n}
			&\geq \liminf_{n\rightarrow \infty}\left(\gamma \lVert Y^{\varepsilon_n} H^{\varepsilon_n}+{{\cal Y}}^{\varepsilon_n} \rVert_{\mathbb{L}^q} +\eta \lVert Z^{\varepsilon_n} (H^{\varepsilon_n})^\top+{{\cal Z}}^{\varepsilon_n} \rVert_{\mathbb{H}^q} \right)\\
			&\geq \gamma \lVert Y^{*} H^{*} +{{\cal Y}}^{*}\rVert_{\mathbb{L}^q} +\eta \lVert Z^{*} (H^{*})^\top +{{\cal Z}}^{*}\rVert_{\mathbb{H}^q}\\
			&\quad - \limsup_{n\rightarrow \infty}\Big(\gamma \lVert  Y^{\varepsilon_n} H^{\varepsilon_n} -Y^{*} H^{*} \rVert_{\mathbb{L}^q} + \eta \lVert Z^{\varepsilon_n} (H^{\varepsilon_n})^\top-Z^{*} (H^{*})^\top \rVert_{\mathbb{H}^q} \Big.\\
			&\hspace{6.em}+\Big.\gamma \lVert {{\cal Y}}^{\varepsilon_n}-{{\cal Y}}^{*}\rVert_{\mathbb{L}^q} + \eta \lVert {{\cal Z}}^{\varepsilon_n}-{{\cal Z}}^{*} \rVert_{\mathbb{H}^q} \Big).
		\end{aligned}
	\end{align*}
	Therefore, it remains to show that 
	\begin{align}\label{eq:sup_conv_zero_x}
		\limsup_{n\rightarrow \infty}\Big(\gamma \lVert  Y^{\varepsilon_n} H^{\varepsilon_n} -Y^{*} H^{*} \rVert_{\mathbb{L}^q} + \eta \lVert Z^{\varepsilon_n} (H^{\varepsilon_n})^\top-Z^{*} (H^{*})^\top \rVert_{\mathbb{H}^q} \Big)&=0,\\
		\label{eq:sup_conv_zero_y}
		\limsup_{n\rightarrow \infty}\left(\gamma \lVert {{\cal Y}}^{\varepsilon_n}-{{\cal Y}}^{*}\rVert_{\mathbb{L}^q} + \eta \lVert {{\cal Z}}^{\varepsilon_n}-{{\cal Z}}^{*} \rVert_{\mathbb{H}^q} \right)&=0.
	\end{align}
	
	\vspace{0.5em}
	\noindent \emph{Step 2: proof of \eqref{eq:sup_conv_zero_x}.}
	Recall that $\lvert H_t^* \rvert, \lvert H_t^{\varepsilon_n}\rvert  \leq K$ $\mathbb{P}\otimes dt$-a.e..\
	Hence, by the triangle inequality (and using that $\gamma,\eta\leq 1$),
	\begin{align*}
			&\gamma \lVert Y^{\varepsilon_n} H^{\varepsilon_n}- {Y}^* {H}^* \rVert_{\mathbb{L}^q} +\eta \lVert Z^{\varepsilon_n} (H^{\varepsilon_n})^\top - Z^*(H^*)^\top \rVert_{\mathbb{H}^q} \\
			&\leq \lVert  ( Y^{\varepsilon_n } -{Y}^*) {H}^*  \rVert_{\mathbb{L}^q}+ \lVert Y^{\varepsilon_n}  (H^{\varepsilon_n}- {H}^*) \rVert_{\mathbb{L}^q} +\lVert  ( Z^{\varepsilon_n } -{Z}^*) ({H}^*)^\top  \rVert_{\mathbb{H}^q}+ \lVert Z^{\varepsilon_n}  (H^{\varepsilon_n}- {H}^*)^\top \rVert_{\mathbb{H}^q} \\
			& \leq  K \Big( \lVert Y^{\varepsilon_n}- Y^* \rVert_{\mathscr{S}^q}+  \lVert Z^{\varepsilon_n } -{Z}^* \rVert_{\mathscr{H}^q}\Big)+ \mathbb{E}\left[ \sup_{t\in [0,T]} \lvert Y_t^{\varepsilon_n} \rvert^q \int_0^T \lvert H_t^{\varepsilon_n} - {H}_t^* \rvert^q dt \right]^{\frac{1}{q}}+ \\
			&\quad + \lVert Z^{\varepsilon_n}  (H^{\varepsilon_n}- {H}^*)^\top \rVert_{\mathbb{H}^q}
	\\
	&=: \mathrm{I}^n + \mathrm{ II}^n + \mathrm{III}^n.
	\end{align*}
	We will show that $\mathrm{I}^n, \mathrm{II}^n, \mathrm{III}^n$ vanish as $n\rightarrow \infty$. 
	
	\vspace{0.5em}
		\noindent
\emph{Step 2, limit of $\mathrm{I}^n$} : 
	Note that $1<q=\frac{p}{p-1}<2$ since $p>3$.
	Hence  Lemma \ref{lem:strong_H}\;(ii)  implies that $\mathrm{I}^n\to 0$ as $n\rightarrow \infty$, as claimed.
	
	\vspace{0.5em}
		\noindent
\emph{Step 2, limit of $\mathrm{II}^n$} :
	Let $v>1$ satisfy $1<vq<2$. 
	Then by H\"older's inequality (with exponent $v>1$) and Jensen's inequality (noting that $x\rightarrow |x|^{\frac{v}{v-1}}$ is convex), 
	\begin{align*}
		\begin{aligned}
			\mathrm{II}^n &\leq  \left\| Y^{\varepsilon_n} \right\|_{\mathscr{S}^{vq}} \mathbb{E} \Bigg[\Big(\int_0^T \lvert H_t^{\varepsilon_n} - {H}_t^* \rvert^{q} dt \Big)^{\frac{v}{v-1}}\Bigg]^{\frac{v-1}{qv}}\\
			&\leq  \left\| Y^{\varepsilon_n} \right\|_{\mathscr{S}^{vq}} T^{\frac{1}{qv}} \lVert H^{\varepsilon_n} - {H}^* \rVert_{\mathbb{L}^{\frac{qv}{v-1}}}.
		\end{aligned}
	\end{align*}
    Lemma \ref{lem:strong_H}\;(i) ensures that $\lVert H^{\varepsilon_n} - {H}^* \rVert_{\mathbb{L}^{\frac{qv}{v-1}}}\rightarrow 0$ $\mbox{as $n\rightarrow \infty$}.$     
    Furthermore, by the triangle inequality, for every $n\in\mathbb{N}$,
	\[
	\left\| Y^{\varepsilon_n} \right\|_{\mathscr{S}^{vq}}\leq \left\| Y^{\varepsilon_n}-Y^* \right\|_{\mathscr{S}^{vq}}+	\left\| Y^* \right\|_{\mathscr{S}^{vq}}.
	\]
	Note that $	\left\| Y^* \right\|_{\mathscr{S}^{vq}}<\infty$ since $Y^*\in \mathscr{S}^2(\mathbb{R})$ and $vq<2$. 
	Moreover, by Lemma~\ref{lem:strong_H}\;(ii) and H\"older's inequality (with exponent $\frac{2}{vq}>1$), we have that 
	\[
	0 \leq  \lim_{n\rightarrow \infty}\left\| Y^{\varepsilon_n}-Y^* \right\|_{\mathscr{S}^{vq}} \leq  \lim_{n\rightarrow \infty}\lVert Y^{\varepsilon_n}-Y^* \rVert_{\mathscr{S}^2}=0.
	\]
	Therefore, we conclude that  $\mathrm{ II}^n\to 0$ as $n\to\infty$. 
	

	\vspace{0.5em}
		\noindent
\emph{Step 2. limit of $\mathrm{III}^n$} : 
	By the triangle inequality and Jensen's inequality (noting that $x\rightarrow |x|^{\frac{2}{q}}$ is convex), 
	\begin{align}\label{eq:estimate_6}
	   \begin{aligned}
	   		\mathrm{III}^n
	   		&\leq  \lVert (Z^{\varepsilon_n}-Z^*)  (H^{\varepsilon_n}- {H}^*)^\top \rVert_{\mathbb{H}^q} +  \lVert Z^{*}  (H^{\varepsilon_n}- {H}^*)^\top \rVert_{\mathbb{H}^q}\\
	   		& \leq  \mathbb{E}\Bigg[\int_0^T \lvert  Z_t^{\varepsilon_n }- Z_t^*\rvert^2 \lvert H_t^{\varepsilon_n} - H_t^* \rvert^2 dt  \Bigg]^{\frac{1}{2}} +  \mathbb{E}\Bigg[\int_0^T \lvert  Z_t^{*}\rvert^2 \lvert H_t^{\varepsilon_n} - H_t^* \rvert^2 dt \Bigg]^{\frac{1}{2}}.
	   \end{aligned}
   \end{align}
   Moreover, since  $\lvert H_t^* \rvert, \lvert H_t^{\varepsilon_n}\rvert  \leq K$ $\mathbb{P}\otimes dt$-a.e., it follows that
	\begin{align*}
		\begin{aligned}
			\mathbb{E}\left[\int_0^T \lvert  Z_t^{\varepsilon_n }- Z_t^*\rvert^2 \lvert H_t^{\varepsilon_n} - H_t^* \rvert^2 dt  \right]^{\frac{1}{2}}
			&\leq  2K \lVert Z^{\varepsilon_n}-Z^*\rVert_{\mathscr{H}^2}
			\to 0,
		\end{aligned}
	\end{align*}
	as $n\rightarrow\infty$, where we used Lemma \ref{lem:strong_H}\;(ii) in the last step.
	Therefore, it remains to show that the last term in \eqref{eq:estimate_6} vanishes when $n\rightarrow \infty$. 
	To that end, we note that since $\lVert H^{\varepsilon_n} - H^* \rVert_{\mathbb{L}^2}\rightarrow 0$ as $n\rightarrow \infty$ (by Lemma \ref{lem:strong_H}\;(i)), the continuous mapping theorem implies that
	$$
	\lvert Z_t^*\rvert^2 \lvert H_t^{\varepsilon_n} - H_t^* \rvert^2  \xrightarrow[]{\mathbb{P}\otimes dt} 0\quad \mbox{as $n\rightarrow \infty$}.
	$$
	Finally since $Z^*\in \mathscr{H}^2(\mathbb{R}^d)$ and $\lvert H^{\varepsilon_n}_t-H^*_t \rvert \leq 2K$ $\mathbb{P}\otimes dt$-a.e., the dominated convergence theorem guarantees that $\mathbb{E}[\int_0^T \lvert  Z_t^{*}\rvert^2 \lvert H_t^{\varepsilon_n} - H_t^* \rvert^2 dt ]^{\frac{1}{2}}\rightarrow 0$ as $n\rightarrow \infty$. 
	
	\vspace{0.5em}
		\noindent
\emph{Step 3: proof of \eqref{eq:sup_conv_zero_y}.}
    By H\"older's inequality (with exponent $\frac{2}{q}>1$), 
	\begin{align*}
		\begin{aligned}
			& \lVert {{\cal Y}}^{\varepsilon_n}-{{\cal Y}}^{*}\rVert_{\mathbb{L}^q}+ \lVert {{\cal Z}}^{\varepsilon_n}-{{\cal Z}}^{*} \rVert_{\mathbb{H}^q} \\
			&\quad = \mathbb{E}\Bigg[\int_0^T \Big(\sum_{i=1}^d\big({\cal Y}_t^{\varepsilon_n,i}- {\cal Y}_t^{*,i} \big)^2\Big)^{\frac{q}{2}} dt  \Bigg]^{\frac{1}{q}}+ \mathbb{E}\Bigg[ \Big(\int_0^T\sum_{i=1}^d\big\lvert  {\cal Z}_t^{\varepsilon_n,i }- {\cal Z}_t^{*,i} \big \rvert^2 dt \Big)^{\frac{q}{2}} \Bigg]^{\frac{1}{q}} \\
			&\quad \leq  T^{\frac{1}{q}-\frac{1}{2}} \mathbb{E}\Bigg[\sum_{i=1}^d \int_0^T \big({\cal Y}_t^{\varepsilon_n,i}- {\cal Y}_t^{*,i} \big)^2 dt  \Bigg]^{\frac{1}{2}}+ \mathbb{E}\Bigg[ \sum_{i=1}^d\int_0^T\big\lvert {\cal Z}_t^{\varepsilon_n,i}- {\cal Z}_t^{*,i} \big \rvert^2 dt  \Bigg]^{\frac{1}{2}}
			=:\mathrm{IV}^n. 
		\end{aligned}
	\end{align*}
	Moreover, it follows from the power triangle inequality that
	\begin{eqnarray*}
		\begin{aligned}
			\mathrm{IV}^n&\leq T^{\frac{1}{q}-\frac{1}{2}} d^{\frac{1}{2}} \sum_{i=1}^d \Bigg(  \mathbb{E}\Bigg[ \int_0^T \big( {\cal Y}_t^{\varepsilon_n,i}- {\cal Y}_t^{*,i} \big)^2 dt   \Bigg]^{\frac{1}{2}}+  \lVert {\cal Z}^{\varepsilon_n,i}- {\cal Z}^{*,i} \rVert_{\mathscr{H}^2}\Bigg)\\
			&\leq T^{\frac{1}{q}-\frac{1}{2}} d^{\frac{1}{2}}\sum_{i=1}^d  \left( T^{\frac{1}{2}}\lVert {\cal Y}^{\varepsilon_n,i}- {\cal Y}^{*,i}  \rVert_{\mathscr{S}^2} +  \lVert {\cal Z}^{\varepsilon_n,i}- {\cal Z}^{*,i} \rVert_{\mathscr{H}^2} \right).  
		\end{aligned}
	\end{eqnarray*}
	Combined with Lemma \ref{lem:strong_H}\;(iii), this estimate ensures that  \eqref{eq:sup_conv_zero_y} holds. 
	The proof is complete.
\end{proof}

\section{Remaining proofs}
\label{sec:remaining.proofs}

\begin{proof}[Proof of Lemma \ref{lem:posterior_refer2}]
	We start by proving \textrm{(i).} Since $\lvert b_t^o \rvert +\lVert \sigma_t^o \rVert_{\operatorname{F}} \leq C_{{b},{\sigma}}$ $\mathbb{P}\otimes dt$-a.e., 
	\[
	\lVert b^o \rVert_{\mathbb{L}^p}^p = \mathbb{E}\left[\int_0^T \lvert b_t^o \rvert^p dt\right] \leq (C_{b,\sigma})^p T, \quad \lVert \sigma^o \rVert_{\mathbb{H}^p}^p = \mathbb{E}\Bigg[\Bigg(\int_0^T\lVert \sigma_t^o \rVert_{\operatorname{F}}^2dt\Bigg)^{\frac{p}{2}}\Bigg] \leq (C_{b,\sigma})^p T^{\frac{p}{2}}.
	\] 
	This ensures that Assumption \ref{as:posterior_refer}\;(i) holds. 
	
	Next, we claim that Assumption \ref{as:posterior_refer}\;(ii) holds. \;Note that from the uniform ellipticity condition on $(\sigma^o)^\top \sigma^o$, it follows that there exists some constant $C_{{c}}>0$ such that $y^\top (\sigma_t^o)^\top \sigma_t^o y \geq \frac{1}{C_{{c}}}\lvert y \rvert^2$ $\mathbb{P}\otimes dt$-a.e. for every $y\in \mathbb{R}^d$, hence
	$$
	\lvert (\sigma_t^o)^{-1} y \rvert \leq  \sqrt{C_{{c}}} \lvert y \rvert\;\;\mbox{$\mathbb{P}\otimes dt$-a.e.},\;\; \mbox{for every $y\in\mathbb{R}^d.$}
	$$
	In particular, using the uniform boundedness of $b^o$, it follows that
	\begin{align}\label{eq:bdd_novikov}
		\frac{1}{2}\int_0^T \lvert (\sigma^o_u)^{-1}{b}_u^o \rvert^2 du\leq \frac{1}{2}\int_0^T C_{{c}} \lvert {b}_u^o\rvert^2 du \leq \frac{T}{2} C_{{b},\sigma}^2 C_{{c}}<\infty\;\;\mbox{$\mathbb{P}$-a.s.}.
	\end{align}
	Thus, $\int_0^\cdot((\sigma^o_u)^{-1}{b}_u^o )^\top d{W}_u$ is well-defined and ${\cal D}$ satisfies ${\cal D}_\cdot=1-\int_0^\cdot{\cal D}_u(( \sigma^o_u)^{-1}{b}_u^o )^\top d{W}_u$ showing that ${\cal D}$ is a continuous, $(\mathbb{F},\mathbb{P})$-local martingale.
	
	Moreover, \eqref{eq:bdd_novikov} clearly implies that $\mathbb{E}[\exp(\frac{1}{2}\int_0^T \lvert (\sigma^o_u)^{-1}{b}_u^o \rvert^2 du)]<\infty.$ 
	Hence, Novikov's condition in \cite[Proposition 3.5.12, p.198]{KS1991} ensures that ${\cal D}$ is a strictly positive $(\mathbb{F},\mathbb{P})$-martingale.
	
	We proceed to show that ${\cal D}_T\in L^\beta({\cal F}_T;\mathbb{R})$ for every $\beta \geq 1$. 
	Indeed, using \eqref{eq:bdd_novikov}, we have that for every $\beta \geq1$, 
	\begin{align}\label{eq:holder_exp}
		\begin{aligned}
			\mathbb{E}[{\cal D}_T ^\beta] &= \mathbb{E}\Bigg[{\cal E} \left(-\beta \big((\sigma^o)^{-1}{b}^o\big) \cdot {W}\right)_T\exp\bigg(\bigg(\frac{\beta^2}{2}-\frac{\beta}{2}\bigg)\int_0^T \lvert (\sigma^o_t)^{-1}{b}_t^o \rvert^2 dt  \bigg)\Bigg]\\
			&\leq \mathbb{E}\Bigg[{\cal E} \left(-\beta \big((\sigma^o)^{-1}{b}^o\big) \cdot {W}\right)_T\Bigg]\exp\bigg(\bigg(\frac{\beta^2}{2}-\frac{\beta}{2}\bigg){T} C_{{b},\sigma}^2 C_{{c}}  \bigg).
		\end{aligned}
	\end{align}
	Since ${\cal E} (-\beta ((\sigma^o)^{-1}{b}^o) \cdot {W})_T$ is a nonnegative continuous $(\mathbb{F},\mathbb{P})$-local martingale, it is an $(\mathbb{F},\mathbb{P})$-supermartingale and hence integrable.
		 This together with \eqref{eq:holder_exp} ensures that ${\cal D}_T\in L^\beta({\cal F}_T;\mathbb{R})$, for every $\beta \geq 1$. 
	
	\vspace{0.5em}
	Now let us prove \textrm{(ii).} Since $b_t^o= \widetilde{b}^o(t,S_t^o),$ $\sigma_t^o= \widetilde{\sigma}^o(t,S_t^o)$ $\mathbb{P}\otimes dt$-a.e., $S^o$ given in \eqref{eq:posterior_semi_ito} is driven by the following stochastic differential equation (SDE):
	\begin{align}\label{eq:posterior_SDE}
		S_t^{o}= s_0+ \int_0^t\widetilde{b}^o(u,S_u^o) du + \int_0^t \widetilde{\sigma}^o(u,S_u^o) dW_u\quad \mbox{$\mathbb{P}$-a.s., $t\in [0,T]$}.
	\end{align}
	Using the Lipschitz and linear growth conditions on $(\widetilde{b}^o,\widetilde{\sigma}^o)$, an application of \cite[Theorem~2.3.1]{Mao2007} shows that the SDE \eqref{eq:posterior_SDE} has a unique solution.  %
	
	Furthermore, since $S_0^o=s_0\in \mathbb{R}$, an application of \cite[Theorem~2.4.1]{Mao2007} shows that there is some $C>0$ (depending on $s_0\in\mathbb{R}$, $p> 3$ and $T>0$) which satisfies
	\[
	\mathbb{E} \left[\lvert S_t^o \rvert^p\right] \leq C, \quad \mbox{for every $t\in[0,T]$}.
	\]
	Therefore, the linear growth condition on $\widetilde{b}^o$ (with the constant $C_{\widetilde{b},\widetilde{\sigma}}>0$), and the power triangle inequality implies
	\[
	\lVert b^o \rVert_{\mathbb{L}^p}^p
	\leq (C_{\widetilde{b},\widetilde{\sigma}}2)^p\mathbb{E}\bigg[\int_0^T (1+\lvert S_t^o  \rvert^p) dt \bigg] \leq (C_{\widetilde{b},\widetilde{\sigma}}2)^p T  \left(1 + C\right)<\infty.
	\]
	
	Furthermore, using the same arguments as above and Jensen's inequality (noting that $x\rightarrow |x|^{\frac{p}{2}}$ is convex), we have
	\begin{align*}
			\quad \lVert \sigma^o \rVert_{\mathbb{H}^p}^p \leq (C_{\widetilde{b},\widetilde{\sigma}}2)^p \mathbb{E}\Bigg[\bigg(\int_0^T (1+\lvert S_t^o  \rvert^2) dt\bigg)^{\frac{p}{2}} \Bigg]\leq (C_{\widetilde{b},\widetilde{\sigma}}2)^p T^{\frac{p}{2}-1} 2^{\frac{p}{2}} \mathbb{E}\Bigg[\int_0^T (1+\lvert S_t^o  \rvert^p) dt \Bigg]<\infty.
	\end{align*}
	Hence Assumption \ref{as:posterior_refer}\;(i) holds. 
	
	Assumption \ref{as:posterior_refer}\;(ii) follows from the Bene\v{s} condition \cite{Benes1971}, i.e., $(\sigma^o_t)^{-1} {b}_t^o= \theta(t,{W})$ $\mathbb{P}\otimes dt$-a.e.. Indeed, by \cite[Corollary 3.5.16]{KS1991}, ${\cal D}$ is a strictly positive $(\mathbb{F},\mathbb{P})$-martingale. Moreover, the condition that ${\cal D}_T\in L^\beta({\cal F}_T;\mathbb{R})$, for every $\beta \geq 1$ follows from \cite[Corollary 2]{GM2003}. This completes the proof.
\end{proof}

\begin{proof}[Proof of Lemma \ref{lem:set_value_ex}]
	We start by proving (i). For every $(\omega,t,v)\in\Omega\times[0,T]\times \mathbb{R}^d$, set 
		\[
		\mathfrak{F}(\omega,t,v):= \big |c_t^{\operatorname{center}}(\omega)-v \big|- r^{\operatorname{radius}}_t(\omega).
		\]
		Then $\mathfrak{F}(\cdot,\cdot,v)$ is $\mathbb{F}$-predictable for every $v\in \mathbb{R}^d$ because $(c^{\operatorname{center}}_t)_{t\in[0,T]}$ and $(r^{\operatorname{radius}}_t)_{t\in[0,T]}$ are $\mathbb{F}$-predictable. Moreover, $\mathfrak{F}$ is continuous in $v\in \mathbb{R}^d$. Therefore, $\mathfrak{F}:\Omega\times[0,T]\times \mathbb{R}^d\to \mathbb{R}$ is a Carath\'eodory function. Hence, an application of \cite[Proposition~14.33]{rockafellar2009variational} ensures that ${\cal K}^{\operatorname{Ball}}$ is closed-valued and weakly $\mathbb{F}$-{predictably} measurable.
		
		The convexity of the values of ${\cal K}^{\operatorname{Ball}}$ is obvious by its definition and the uniformly-boundedness follows directly from the fact that  $(c^{\operatorname{center}}_t)_{t\in[0,T]}$ and $(r^{\operatorname{radius}}_t)_{t\in[0,T]}$ have bounded values.
		
		\vspace{0.5em}
		\noindent Now let us prove (ii). For every $i=1,\dots,d$, define ${\cal LB}^i: \Omega \times [0,T]\ni (\omega,t)\twoheadrightarrow {\cal LB}^i(\omega,t)\subseteq\mathbb{R}^d$~by
		\begin{align*}
			{\cal LB}^i(\omega,t):= \left\{v=(v^1,\dots,v^d)^\top\in\mathbb{R}^d\,\big|\, \underline{a}^i_t(\omega)-v^i\leq 0\right\}.
		\end{align*}
		Similarly, define by  ${\cal UB}^i: \Omega \times [0,T]\ni (\omega,t)\twoheadrightarrow {\cal UB}^i(\omega,t)=\{v\in\mathbb{R}^d| v^i-\overline{a}^i_t(\omega)\leq 0\}\subseteq\mathbb{R}^d$. Then we have for every $(\omega,t)\in \Omega\times [0,T]$ that
		\begin{align}\label{eq:intersect_Cara}
			\big(\cap^d_{i=1}{\cal LB}^i(\omega,t)\big) \cap \big(\cap^d_{i=1}{\cal UB}^i(\omega,t)\big) ={\cal K}^{\operatorname{Box}}(\omega,t).
		\end{align}
		
		We note that by the $\mathbb{F}$-predictability of $(\underline a_t^i)_{t\in[0,T]}$ and $(\overline a_t^i)_{t\in[0,T]}$, the same arguments as presented for the proof of (i) can be used to show that for every $i=1,\dots,d$ both ${\cal LB}^i$ and ${\cal UB}^i$ are closed-valued and weakly $\mathbb{F}$-{predictably} measurable correspondences.
		
		Moreover, by \eqref{eq:intersect_Cara}, an application of \cite[Proposition 14.11\;(i)]{rockafellar2009variational} ensures that ${\cal K}^{\operatorname{Box}}$ is closed-valued and weakly $\mathbb{F}$-{predictably} measurable. Both the convex-valuedness and uniformly-boundedness of ${\cal K}^{\operatorname{Box}}$  follow from its definition and the bounded values of $(\underline a_t^i)_{t\in[0,T]}$,~$(\overline a_t^i)_{t\in[0,T]}$.
\end{proof}

\begin{proof}[Proof of Lemma \ref{lem:explicit.Y.Z}]
	The expressions for $Y^*$ and ${\cal Y}^*$ follow from the definition of the Galtchouk-Kunita-Watanabe decompositions given in \eqref{eq:BSDE_re_x} and \eqref{eq:BSDE_re_y} by taking conditional expectations.
	
	As for the expressions of $Z^\ast$, recall that   $L^*\in \mathscr{M}^2(\mathbb{R})$ is strongly orthogonal with $W^{i}$ for every $i=1,\dots,d$, see  Proposition \ref{pro:main0}\;(ii) and Lemma \ref{lem:BSDE}.
	Hence, for every $t\in[0,T]$, $\mathbb{P}$-a.s.,
	\begin{align*}
		\begin{aligned}
			d(\langle Y^*, W^{1} \rangle_t,\dots,\langle Y^*, W^{d} \rangle_t)^\top =  Z_{t}^{*} dt+ d(\langle L^*, W_{1}\rangle_t,\dots,\langle L^*, W_{d}\rangle_t)^\top=Z_{t}^* dt,
		\end{aligned}
	\end{align*}
	proving the claimed expression of $Z^*$.
	
	The proof for the expression of ${\cal Z}^{*,i}$ follows from the same arguments (replacing $Z^*$ with ${\cal Z}^{*,i}$ and $Y^*$ with ${\cal Y}^{*,i}$).
\end{proof}

\begin{proof}[Proof of Corollary \ref{cor:regularity}] 
By the regularity assumption on $J$, i.e., $J\in  \mathcal{C}^{1,2,2}$, an application of It\^o's formula ensures that for every $t\in[0,T]$,
	\begin{align*}
		\begin{aligned}
			J(t,X_t^{H^*;o},S_t^o)&=\partial_{x}f\big(X_T^{H^*;o},l(S_T^o)\big)-\int_t^T\big({\cal L}_{r}^1+{\cal L}_{r}^2+{\cal L}_{r}^3  \big)dr\\
			&\quad-\int_t^T\Big(\partial_xJ(r,X_r^{H^*;o},S_r^o)(H_r^*)^\top +(\nabla_{{s}}J)^\top(r,X_r^{H^*;o},S_r^o) \Big)  \sigma_r^o dW_r,
		\end{aligned}
	\end{align*}
	where 
	\begin{itemize}[leftmargin=2em]
		\item [$\;$] ${\cal L}^{1}_r:=\partial_t J(r,X_r^{H^*;o},S_r^o)+\frac{1}{2} \partial_{xx} J (r,X_r^{H^*;o},S_r^o) \lvert (\sigma_r^o)^\top H_r^*   \rvert^2+  \partial_x J(r,X_r^{H^*;o},S_r^o) (H_r^*)^\top b_r^o$;
		\item [$\;$] ${\cal L}^{2}_r:= \sum_{i=1}^d \partial_{x,s_i}J(r,X_r^{H^*;o},S_r^o) ((H_r^*)^\top \sigma_r^o)(\sigma_{r}^{o,i})^\top$;
		\item [$\;$] ${\cal L}^{3}_r:=\frac{1}{2} \operatorname{tr}\big((\sigma_{r}^o)(\sigma_{r}^o)^\top D^2_sJ(r,X_r^{H^*;o},S_r^o)  \big)+ (\nabla_sJ)^\top(r,X_r^{H^*;o},S_r^o) b_r^o $;
	\end{itemize} 
	with $(\sigma_{t}^{o,i})_{t\in[0,T]}$, $i=1,\dots,d$, denoting the $i$-th row vector process of $\sigma^o$, $D^2_sJ$ denoting the Hessian of $J$ with respect to $s$. 
	
	Hence, since $Y_t^* = J\big(t,X_t^{H^*;o},S_t^o\big)$ holds for every $t\in[0,T]$, $\mathbb{P}$-a.s.\;(see \eqref{eq:condi_forms}), by Lemma~\ref{lem:explicit.Y.Z} and the regularity on $J$, we have that for every $t\in[0,T)$, $\mathbb{P}$-a.s.,
	\begin{align*}
		\begin{aligned}
			Z_t^*&= \frac{d}{dt}\Big(\langle J(\;\cdot\;,X^{H^*;o},S^o),W^1\rangle_t,\cdots,\langle  J(\;\cdot\;,X^{H^*;o},S^o),W^{d}\rangle_t\Big)^\top\\
			&=(\sigma_t^o)^\top [\partial_xJ(t,X_t^{H^*;o},S_t^o)H_t^* + \nabla_s J(t,X_t^{H^*;o},S_t^o) ].
		\end{aligned}
	\end{align*}
	
	We can use the same arguments (replacing $Y^*=J(\cdot,X^{H^*;o},S^o)$ with ${Y}^{*,i}={\cal J}^i(\;\cdot\;,X^{H^*;o},S^o)$ and using the regularity assumption on ${\cal J}^i$) to show that the property for ${\cal Z}^{*,i}$ holds for every $i=1,\dots,d$. 
\end{proof}

	

\begin{rem}\label{rem:regularity}
	The regularity assumptions in Corollary \ref{cor:regularity} (see \eqref{eq:condi_forms}) are satisfied if $g$ and\;$f$ are deterministic (hence we can assume that $g$ and $f$ are defined on $[0,T]\times \mathbb{R}\times\mathbb{R}^d$ and~$\mathbb{R}^2$, respectively) and the following conditions (i),\,(ii)-a,\,(iii)-a or (i),\,(ii)-b,\,(iii)-b hold (see \cite[Theorem~3.2]{pardoux2005backward} and \cite[Theorem 5.7.6 and Remark~5.7.8]{KS1991}): 
	\begin{itemize}
		\item [(i)] $\mathbb{F}$ is the completion of the filtration generated by the Brownian motion $W$;
		\item [(ii)-a.] For every $t\in[0,T]$, $\partial_xg(t,\cdot,\cdot)$, $\partial_xf(\cdot,l(\cdot))$ and $\partial_yf(\cdot,l(\cdot))\partial_{s_i}l(\cdot)$, $i=1,\dots,d$, are three times continuously differentiable and their derivatives of order less than or equal to 3 grow at most like a polynomial function of the variable at infinity; 
		\item [(ii)-b.] $(t,x,h)\to \partial_xg(t,x,h)$ is continuous; 
		\item [(iii)-a.] Let $\hat{b}^o:\mathbb{R}^d\rightarrow \mathbb{R}^d$ and $\hat{\sigma}^o:\mathbb{R}^d\rightarrow \mathbb{R}^{d\times d}$ be three times continuously differentiable with bounded first, second, and third derivatives. Assume that $H^*$ given in Proposition~\ref{pro:main0}\;(i) is constant, and $b^o$ and $\sigma^o$ given in \eqref{eq:posterior_semi_ito} satisfy $\mathbb{P}\otimes dt$-a.e., 
		\begin{align*}
			\quad\quad b^o_t=\hat{b}^o(S_t^o),\quad \sigma^o_t=\hat{\sigma}^o(S_t^o);
		\end{align*}
		\item [(iii)-b.] Let $\bar{b}^o:[0,T]\times \mathbb{R}^d\rightarrow \mathbb{R}^d$ be bounded, H\"older-continuous. Moreover, let $\bar{\sigma}^o:[0,T]\times \mathbb{R}^d\rightarrow \mathbb{R}^{d\times d}$ be bounded and satisfy that $\bar{\sigma}^o(\bar{\sigma}^o)^\top:[0,T]\times \mathbb{R}^d\rightarrow \mathbb{R}^{d\times d}$ is H\"older-continuous and satisfies a uniform ellipticity condition, i.e., there is $C_\sigma>0$ such that for every $v\in \mathbb{R}^d$ and $(t,x)\in [0,T]\times \mathbb{R}^d$, $v^\top\bar{\sigma}^o(\bar{\sigma}^o)^\top(t,x)v\geq C_\sigma |v|^2.$ 
		Assume that $H^*$ is deterministic, H\"older-continuous in $[0,T]$, and satisfy 
		$|H_t^*|>0$ for every $t\in[0,T]$, and $b^o$ and $\sigma^o$ satisfy $\mathbb{P}\otimes dt$-a.e., 
		\[
		\quad\quad b^o_t=\bar{b}^o(t,S_t^o),\quad \sigma^o_t=\bar{\sigma}^o(t,S_t^o).
		\]
	\end{itemize}
\end{rem}

\begin{proof}[Proof of Theorem \ref{thm:main2}]
	Since $V^*(0)=V(0)$ and 
	\[V'(0) 
	= \gamma \lVert Y^*H^*+\mathcal{ Y}^* \rVert_{\mathbb{L}^q}   +  \eta \lVert Z^* (H^*)^\top+\mathcal{ Z}^* \rVert_{\mathbb{H}^q}\] (see Theorem \ref{thm:main}), the inequality 
	\[
	V^\ast(\varepsilon)\leq V(0)+\varepsilon V'(0)+O(\varepsilon^2)=V(\varepsilon)+O(\varepsilon^2)
	\]
	as $\varepsilon\downarrow 0$, follows directly from Remark \ref{rem:auxiliary}.~The reverse inequality follows by using the same arguments as for the proof of Lemma \ref{lem:low_bound}, but replacing $H^\varepsilon$ by $H^*$.
\end{proof}



\bibliographystyle{abbrv}
\bibliography{references}

\begin{thebibliography}{100}

\bibitem{CharalambosKim2006infinite}
C.~D. Aliprantis and K.~C. Border.
\newblock {\em Infinite dimensional analysis: A Hitchhiker's Guide}.
\newblock Springer, 2006.

\bibitem{backhoff2020adapted}
J.~Backhoff-Veraguas, D.~Bartl, M.~Beiglb{\"o}ck, and M.~Eder.
\newblock Adapted wasserstein distances and stability in mathematical finance.
\newblock {\em Finance Stoch.}, 24(3):601--632, 2020.

\bibitem{bartl2021sensitivity}
D.~Bartl, S.~Drapeau, J.~Ob{\l}{\'o}j, and J.~Wiesel.
\newblock Sensitivity analysis of wasserstein distributionally robust
  optimization problems.
\newblock {\em Proc. R. Soc. A}, 477(2256):20210176, 2021.

\bibitem{bartl2021limits}
D.~Bartl, S.~Eckstein, and M.~Kupper.
\newblock Limits of random walks with distributionally robust transition
  probabilities.
\newblock {\em Electron. Commun. Probab.}, (26):1--13, 2021.

\bibitem{BKN21}
D.~Bartl, M.~Kupper, and A.~Neufeld.
\newblock {Duality Theory for Robust Utility Maximization}.
\newblock {\em Finance Stoch.}, 25:469--503, 2021.

\bibitem{bartl2024numerical}
D.~Bartl, A.~Neufeld, and K.~Park.
\newblock Numerical method for nonlinear \textrm{Kolmogorov PDEs} via
  sensitivity analysis.
\newblock {\em arXiv preprint arXiv:2403.11910}, 2024.

\bibitem{bartl2023sensitivity}
D.~Bartl and J.~Wiesel.
\newblock Sensitivity of multiperiod optimization problems with respect to the
  adapted wasserstein distance.
\newblock {\em SIAM J. Financial Math.}, 14(2):704--720, 2023.

\bibitem{bayraktar2013multidimensional}
E.~Bayraktar and Y.-J. Huang.
\newblock On the multidimensional controller-and-stopper games.
\newblock {\em SIAM J. Control Optim.}, 51(2):1263--1297, 2013.

\bibitem{bayraktar2013stability}
E.~Bayraktar and R.~Kravitz.
\newblock Stability of exponential utility maximization with respect to market
  perturbations.
\newblock {\em Stoch. Process. Appl.}, 123(5):1671--1690, 2013.

\bibitem{bayraktar2011optimal_a}
E.~Bayraktar and S.~Yao.
\newblock Optimal stopping for non-linear expectations—part i.
\newblock {\em Stoch. Process. Appl.}, 121(2):185--211, 2011.

\bibitem{bayraktar2011optimal_b}
E.~Bayraktar and S.~Yao.
\newblock Optimal stopping for non-linear expectations—part ii.
\newblock {\em Stoch. Process. Appl.}, 121(2):212--264, 2011.

\bibitem{Benes1971}
V.~Bene\u{s}.
\newblock {Existence of optimal stochastic control laws}.
\newblock {\em SIAM J. Control}, 9:446--475, 1971.

\bibitem{biagini2019robust}
F.~Biagini, J.~Mancin, and T.~M. Brandis.
\newblock Robust mean--variance hedging via \textrm{$G$-expectation}.
\newblock {\em Stoch. Process. Appl.}, 129(4):1287--1325, 2019.

\bibitem{biagini2017robust}
S.~Biagini and M.~{\c{C}}. P{\i}nar.
\newblock The robust merton problem of an ambiguity averse investor.
\newblock {\em Math. Financ. Econom.}, 11(1):1--24, 2017.

\bibitem{blanchet2019quantifying}
J.~Blanchet and K.~Murthy.
\newblock Quantifying distributional model risk via optimal transport.
\newblock {\em Math. Oper. Res.}, 44(2):565--600, 2019.

\bibitem{bonnans2013perturbation}
J.~F. Bonnans and A.~Shapiro.
\newblock {\em Perturbation analysis of optimization problems}.
\newblock Springer Science \& Business Media, 2013.

\bibitem{cadenillas1995stochastic}
A.~Cadenillas and I.~Karatzas.
\newblock The stochastic maximum principle for linear, convex optimal control
  with random coefficients.
\newblock {\em SIAM J. Control Optim.}, 33(2):590--624, 1995.

\bibitem{calafiore2007ambiguous}
G.~C. Calafiore.
\newblock Ambiguous risk measures and optimal robust portfolios.
\newblock {\em SIAM J. Optim.}, 18(3):853--877, 2007.

\bibitem{chen2002ambiguity}
Z.~Chen and L.~Epstein.
\newblock Ambiguity, risk, and asset returns in continuous time.
\newblock {\em Econometrica}, 70(4):1403--1443, 2002.

\bibitem{cheridito2007second}
P.~Cheridito, H.~M. Soner, N.~Touzi, and N.~Victoir.
\newblock Second-order backward stochastic differential equations and fully
  nonlinear parabolic \textrm{PDEs}.
\newblock {\em Comm. Pure Appl. Math.}, 60(7):1081--1110, 2007.

\bibitem{CHMP2002}
F.~Coquet, Y.~Hu, J.~M\'emin, and S.~Peng.
\newblock {Filtration-consistent nonlinear expectations and related
  $g$-expectations.}
\newblock {\em Probab. Theory Relat. Fields}, 123(1):1--27, 2002.

\bibitem{cvitanic1992convex}
J.~Cvitani{\'c} and I.~Karatzas.
\newblock Convex duality in constrained portfolio optimization.
\newblock {\em Ann. Appl. Probab.}, pages 767--818, 1992.

\bibitem{cvitanic2001utility}
J.~Cvitani{\'c}, W.~Schachermayer, and H.~Wang.
\newblock Utility maximization in incomplete markets with random endowment.
\newblock {\em Finance Stoch.}, 5(2):259--272, 2001.

\bibitem{czichowsky2012convex}
C.~Czichowsky and M.~Schweizer.
\newblock Convex duality in mean-variance hedging under convex trading
  constraints.
\newblock {\em Adv. Appl. Prob.}, 44(4):1084--1112, 2012.

\bibitem{vega2023duality}
E.~J.~C. {Dela Vega} and H.~Zheng.
\newblock Duality method for multidimensional nonsmooth constrained linear
  convex stochastic control.
\newblock {\em J. Optim. Theory Appl.}, 199(1):80--111, 2023.

\bibitem{DelbaenSchacher06}
F.~Delbaen and W.~Schachermayer.
\newblock {\em {The mathematics of arbitrage}}.
\newblock Springer Science \& Business Media, 2006.

\bibitem{dellacherieprobabilitiesA}
C.~Dellacherie and P.~Meyer.
\newblock Probabilities and potential {A} theory of martingales.

\bibitem{dellacherieprobabilities}
C.~Dellacherie and P.~Meyer.
\newblock Probabilities and potential {B} theory of martingales.

\bibitem{DenisHuPeng2011}
L.~Denis, M.~Hu, and S.~Peng.
\newblock {Function spaces and capacity related to a sublinear expectation:
  application to $G$-Brownian motion paths}.
\newblock {\em Potential Anal.}, 34:139--161, 2011.

\bibitem{denis2013optimal}
L.~Denis and M.~Kervarec.
\newblock Optimal investment under model uncertainty in nondominated models.
\newblock {\em SIAM J. Control Optim.}, 51(3):1803--1822, 2013.

\bibitem{dolinsky2014martingale}
Y.~Dolinsky and H.~M. Soner.
\newblock Martingale optimal transport and robust hedging in continuous time.
\newblock {\em Probab. Theory Relat. Fields}, 160(1-2):391--427, 2014.

\bibitem{dow1992uncertainty}
J.~Dow and S.~R. da~Costa~Werlang.
\newblock Uncertainty aversion, risk aversion, and the optimal choice of
  portfolio.
\newblock {\em Econometrica}, pages 197--204, 1992.

\bibitem{DR1991}
D.~Duffie and H.~Richardson.
\newblock {Mean-variance hedging in continuous time}.
\newblock {\em Ann. Appl. Probab.}, pages 1--15, 1991.

\bibitem{Fouque_et_al16}
J.~P. Fouque, C.~S. Pun, and H.~Y. Wong.
\newblock {Portfolio optimization with ambiguous correlation and stochastic
  volatilities}.
\newblock {\em SIAM J. Control Optim.}, 54(5):2309--2338., 2016.

\bibitem{fuhrmann2023wasserstein}
S.~Fuhrmann, M.~Kupper, and M.~Nendel.
\newblock Wasserstein perturbations of markovian transition semigroups.
\newblock In {\em Annales de l'Institut Henri Poincare (B) Probabilites et
  statistiques}, volume~59, pages 904--932. Institut Henri Poincar{\'e}, 2023.

\bibitem{gao2022wasserstein}
R.~Gao, X.~Chen, and A.~J. Kleywegt.
\newblock Wasserstein distributionally robust optimization and variation
  regularization.
\newblock {\em Oper. Res.}, 2022.

\bibitem{gao2023distributionally}
R.~Gao and A.~Kleywegt.
\newblock Distributionally robust stochastic optimization with wasserstein
  distance.
\newblock {\em Math. Oper. Res.}, 48(2):603--655, 2023.

\bibitem{gilboa1989maxmin}
I.~Gilboa and D.~Schmeidler.
\newblock Maxmin expected utility with non-unique prior.
\newblock {\em J. Math. Econ.}, 18:141--153, 1989.

\bibitem{GM2003}
B.~Grigelionis and V.~Mackevi\v{c}ius.
\newblock {The finiteness of moments of a stochastic exponential}.
\newblock {\em Statistics \& probability letters}, 64(3):243--248, 2003.

\bibitem{HWY2019}
S.~W. He, J.~G. Wang, and J.~A. Yan.
\newblock {\em {Semimartingale theory and stochastic calculus}}.
\newblock Routledge, 2019.

\bibitem{hernandez2007control}
D.~Hern{\'a}ndez-Hern{\'a}ndez and A.~Schied.
\newblock A control approach to robust utility maximization with logarithmic
  utility and time-consistent penalties.
\newblock {\em Stoch. Process. Appl.}, 117(8):980--1000, 2007.

\bibitem{herrmann2017model}
S.~Herrmann and J.~Muhle-Karbe.
\newblock Model uncertainty, recalibration, and the emergence of delta--vega
  hedging.
\newblock {\em Finance Stoch.}, 21:873--930, 2017.

\bibitem{herrmann2017hedging}
S.~Herrmann, J.~Muhle-Karbe, and F.~T. Seifried.
\newblock Hedging with small uncertainty aversion.
\newblock {\em Finance Stoch.}, 21:1--64, 2017.

\bibitem{horst2014forward}
U.~Horst, Y.~Hu, P.~Imkeller, A.~R{\'e}veillac, and J.~Zhang.
\newblock Forward--backward systems for expected utility maximization.
\newblock {\em Stoch. Process. Appl.}, 124(5):1813--1848, 2014.

\bibitem{hu2005utility}
Y.~Hu, P.~Imkeller, and M.~M{\"u}ller.
\newblock Utility maximization in incomplete markets1.
\newblock {\em Ann. Appl. Probab.}, 15(3):1691--1712, 2005.

\bibitem{hugonnier2004optimal}
J.~Hugonnier and D.~Kramkov.
\newblock Optimal investment with random endowments in incomplete markets.
\newblock {\em Ann. Appl. Probab.}, 14(2):845--864, 2004.

\bibitem{Hytonen2016}
T.~Hyt\"onen, J.~Van~Neerven, M.~Veraar, and L.~Weis.
\newblock {\em {Analysis in Banach spaces}}.
\newblock Vol 12., Springer, Berlin, 2016.

\bibitem{Jacod1979}
J.~Jacod.
\newblock {\em {Calcul stochastique et probl\`emes de martingales}}.
\newblock Vol. 714. Springer, Berlin, 1979.

\bibitem{JacodShirayev2013}
J.~Jacod and A.~Shiryaev.
\newblock {\em {Limit theorems for stochastic processes}}.
\newblock Vol. 288. Springer, Berlin, 2013.

\bibitem{yifan}
Y.~Jiang.
\newblock Wasserstein distributional sensitivity to model uncertainty in a
  dynamic context. {DP}hil {T}ransfer of {S}tatus {T}hesis. {U}niversity of
  {O}xford, {J}anuary 2023.
\newblock {\em Private communication}.

\bibitem{jiang2024sensitivity}
Y.~Jiang and J.~Obloj.
\newblock Sensitivity of causal distributionally robust optimization.
\newblock {\em arXiv preprint arXiv:2408.17109}, 2024.

\bibitem{karatzas2007numeraire}
I.~Karatzas and C.~Kardaras.
\newblock The num{\'e}raire portfolio in semimartingale financial models.
\newblock {\em Finance Stoch.}, 11:447--493, 2007.

\bibitem{KLS}
I.~Karatzas, J.~Lehoczky, and S.~Shreve.
\newblock {Optimal portfolio and consumption decisions for a "small investor"
  on a finite horizon.}
\newblock {\em SIAM J. Control Optim.}, 25(6):1557--1586, 1987.

\bibitem{karatzas1991martingale}
I.~Karatzas, J.~P. Lehoczky, S.~E. Shreve, and G.-L. Xu.
\newblock Martingale and duality methods for utility maximization in an
  incomplete market.
\newblock {\em SIAM J. Control Optim.}, 29(3):702--730, 1991.

\bibitem{KS1991}
I.~Karatzas and S.~Shreve.
\newblock {\em {Brownian motion and stochastic calculus}}.
\newblock Vol. 113., Springer, 1991.

\bibitem{KW2000}
I.~Karatzas and H.~Wang.
\newblock {Utility maximization with discretionary stopping}.
\newblock {\em SIAM J. Control Optim.}, 39(1):306--329, 2000.

\bibitem{Klenke14}
A.~Klenke.
\newblock {\em {Probability Theory: A Comprehensive Course}}.
\newblock Universitext. Springer London, London, 2nd edition, 2014.

\bibitem{kramkov1999asymptotic}
D.~Kramkov and W.~Schachermayer.
\newblock The asymptotic elasticity of utility functions and optimal investment
  in incomplete markets.
\newblock {\em Ann. Appl. Probab.}, pages 904--950, 1999.

\bibitem{KunitaWatanabe67}
H.~Kunita and S.~Watanabe.
\newblock {On square integrable martingales}.
\newblock {\em Nagoya Mathematical Journal}, 30:209--245, 1967.

\bibitem{LabHeu2007}
C.~Labb\'e and A.~J. Heunis.
\newblock {Convex duality in constrained mean-variance portfolio optimization}.
\newblock {\em Adv. Appl. Prob.}, 39:77--104, 2007.

\bibitem{lam2016robust}
H.~Lam.
\newblock Robust sensitivity analysis for stochastic systems.
\newblock {\em Math. Oper. Res.}, 41(4):1248--1275, 2016.

\bibitem{larsen2007stability}
K.~Larsen and G.~{\v{Z}}itkovi{\'c}.
\newblock Stability of utility-maximization in incomplete markets.
\newblock {\em Stoch. Process. Appl.}, 117(11):1642--1662, 2007.

\bibitem{li2018constrained}
Y.~Li and H.~Zheng.
\newblock Constrained quadratic risk minimization via forward and backward
  stochastic differential equations.
\newblock {\em SIAM J. Control Optim.}, 56(2):1130--1153, 2018.

\bibitem{lin2021optimal}
Q.~Lin and F.~Riedel.
\newblock Optimal consumption and portfolio choice with ambiguous interest
  rates and volatility.
\newblock {\em Econom. Theory}, 71(3):1189--1202, 2021.

\bibitem{Mao2007}
X.~Mao.
\newblock {\em {Stochastic differential equations and applications}}.
\newblock Elsevier, 2007.

\bibitem{matoussi2015robust}
A.~Matoussi, D.~Possama{\"\i}, and C.~Zhou.
\newblock Robust utility maximization in nondominated models with
  \textrm{2BSDE}: the uncertain volatility model.
\newblock {\em Math. Finance}, 25(2):258--287, 2015.

\bibitem{mohajerin2018data}
P.~Mohajerin~Esfahani and D.~Kuhn.
\newblock Data-driven distributionally robust optimization using the
  wasserstein metric: performance guarantees and tractable reformulations.
\newblock {\em Math. Programming}, 171(1-2):115--166, 2018.

\bibitem{mostovyi2019sensitivity}
O.~Mostovyi and M.~S{\^\i}rbu.
\newblock Sensitivity analysis of the utility maximisation problem with respect
  to model perturbations.
\newblock {\em Finance Stoch.}, 23:595--640, 2019.

\bibitem{mostovyi2024quadratic}
O.~Mostovyi and M.~S{\^\i}rbu.
\newblock Quadratic expansions in optimal investment with respect to
  perturbations of the semimartingale model.
\newblock {\em Finance Stoch.}, to appear, 2024.

\bibitem{nendel2022parametric}
M.~Nendel and A.~Sgarabottolo.
\newblock A parametric approach to the estimation of convex risk functionals
  based on wasserstein distance.
\newblock {\em arXiv preprint arXiv:2210.14340}, 2022.

\bibitem{NeufeldNutz2014}
A.~Neufeld and M.~Nutz.
\newblock {Measurability of semimartingale characteristics with respect to the
  probability law}.
\newblock {\em Stoch. Process. Appl.}, 124(11):3819--3845, 2014.

\bibitem{neufeld2017nonlinear}
A.~Neufeld and M.~Nutz.
\newblock Nonlinear \textrm{L{\'e}vy} processes and their characteristics.
\newblock {\em Trans. Am. Math. Soc.}, 369(1):69--95, 2017.

\bibitem{neufeld2018robust}
A.~Neufeld and M.~Nutz.
\newblock Robust utility maximization with \textrm{L{\'e}vy} processes.
\newblock {\em Math. Finance}, 28(1):82--105, 2018.

\bibitem{nutz2013random}
M.~Nutz.
\newblock Random \textrm{$G$-expectations}.
\newblock {\em Ann. Appl. Probab.}, pages 1755--1777, 2013.

\bibitem{nutz2015robust}
M.~Nutz.
\newblock Robust superhedging with jumps and diffusion.
\newblock {\em Stoch. Process. Appl.}, 125(12):4543--4555, 2015.

\bibitem{obloj2021distributionally}
J.~Ob{\l}{\'o}j and J.~Wiesel.
\newblock Distributionally robust portfolio maximization and marginal utility
  pricing in one period financial markets.
\newblock {\em Math. Finance}, 31(4):1454--1493, 2021.

\bibitem{pardoux2005backward}
E.~Pardoux and S.~Peng.
\newblock Backward stochastic differential equations and quasilinear parabolic
  partial differential equations.
\newblock In {\em Stochastic Partial Differential Equations and Their
  Applications: Proceedings of IFIP WG 7/1 International Conference University
  of North Carolina at Charlotte, NC June 6--8, 1991}, pages 200--217.
  Springer, 2005.

\bibitem{PW22}
K.~Park and H.~Y. Wong.
\newblock {Robust consumption-investment with return ambiguity: A dual approach
  with volatility ambiguity}.
\newblock {\em SIAM J. Financial Math.}, 13(3):802--843, 2022.

\bibitem{PW23}
K.~Park and H.~Y. Wong.
\newblock {Robust Retirement with Return Ambiguity: Optimal $G$-Stopping Time
  in Dual Space}.
\newblock {\em SIAM J. Control Optim.}, 61(3):1009--1037, 2023.

\bibitem{park2023robust}
K.~Park, H.~Y. Wong, and T.~Yan.
\newblock Robust retirement and life insurance with inflation risk and model
  ambiguity.
\newblock {\em Insurance Math. Econom.}, 110:1--30, 2023.

\bibitem{PengS_gexp1997}
S.~Peng.
\newblock {\em {BSDE and related $g$-expectations. In \textit{Backward
  Stochastic Differential Equations}, ed. N. El Karoui and L. Mazliak}}.
\newblock Longman, Harlow, pp.141--159, 1997.

\bibitem{Peng07}
S.~Peng.
\newblock {\em {$G$-expectation, $G$-Brownian motion and related stochastic
  calculus of It\^o type.}}
\newblock 2007.

\bibitem{Peng08}
S.~Peng.
\newblock {Multi-dimensional $G$-Brownian motion and related stochastic
  calculus under $G$-expectation.}
\newblock {\em Stochastic Process. Appl.}, 118(12):2223--2253., 2008.

\bibitem{pflug2007ambiguity}
G.~Pflug and D.~Wozabal.
\newblock Ambiguity in portfolio selection.
\newblock {\em Quant. Finance}, 7(4):435--442, 2007.

\bibitem{pflug2010version}
G.~C. Pflug.
\newblock Version-independence and nested distributions in multistage
  stochastic optimization.
\newblock {\em SIAM J. Optim.}, 20(3):1406--1420, 2010.

\bibitem{pflug2012distance}
G.~C. Pflug and A.~Pichler.
\newblock A distance for multistage stochastic optimization models.
\newblock {\em SIAM J. Optim.}, 22(1):1--23, 2012.

\bibitem{pham2009continuous}
H.~Pham.
\newblock {\em Continuous-time stochastic control and optimization with
  financial applications}, volume~61.
\newblock Springer Science \& Business Media, 2009.

\bibitem{pham2022portfolio}
H.~Pham, X.~Wei, and C.~Zhou.
\newblock Portfolio diversification and model uncertainty: A robust dynamic
  mean-variance approach.
\newblock {\em Math. Finance}, 32(1):349--404, 2022.

\bibitem{Protter2005}
P.~Protter.
\newblock {\em {Stochastic differential equations}}.
\newblock Springer, Berlin, 2005.

\bibitem{rockafellar2009variational}
R.~T. Rockafellar and R.~J.-B. Wets.
\newblock {\em Variational analysis}, volume 317.
\newblock Springer Science \& Business Media, 2009.

\bibitem{sauldubois2024first}
N.~Sauldubois and N.~Touzi.
\newblock First order \textrm{Martingale} model risk and semi-static hedging.
\newblock {\em arXiv preprint arXiv:2410.06906}, 2024.

\bibitem{schied2007optimal}
A.~Schied.
\newblock Optimal investments for risk-and ambiguity-averse preferences: a
  duality approach.
\newblock {\em Finance Stoch.}, 11(1):107--129, 2007.

\bibitem{Schweizer1992}
M.~Schweizer.
\newblock {Mean-variance hedging for general claims}.
\newblock {\em Ann. Appl. Probab.}, pages 171--179, 1992.

\bibitem{soner2012wellposedness}
H.~M. Soner, N.~Touzi, and J.~Zhang.
\newblock Wellposedness of second order backward \textrm{SDEs}.
\newblock {\em Probab. Theory Related Fields}, 153(1-2):149--190, 2012.

\bibitem{soner2013dual}
H.~M. Soner, N.~Touzi, and J.~Zhang.
\newblock Dual formulation of second order target problems.
\newblock {\em Ann. Appl. Probab.}, 23(1):308--347, 2013.

\bibitem{SonerTouziZhang2011}
M.~Soner, N.~Touzi, and J.~Zhang.
\newblock {Quasi-sure Stochastic Analysis through Aggregation}.
\newblock {\em Electron. J. Probab.}, 16:1844 -- 1879, 2011.

\bibitem{tevzadze2013robust}
R.~Tevzadze, T.~Toronjadze, and T.~Uzunashvili.
\newblock Robust utility maximization for a diffusion market model with
  misspecified coefficients.
\newblock {\em Finance Stoch.}, 17:535--563, 2013.

\bibitem{vogel1988stability}
S.~Vogel.
\newblock Stability results for stochastic programming problems.
\newblock {\em Optimization}, 19(2):269--288, 1988.

\bibitem{yang2019constrained}
Z.~Yang, G.~Liang, and C.~Zhou.
\newblock Constrained portfolio-consumption strategies with uncertain
  parameters and borrowing costs.
\newblock {\em Math. Financ. Econom.}, 13:393--427, 2019.

\bibitem{ZhouLi2000}
X.~Y. Zhou and D.~Li.
\newblock {Continuous-time mean-variance portfolio selection: A stochastic LQ
  framework}.
\newblock {\em Appl. Math. Optim.}, 42:19--33, 2000.

\end{thebibliography}
     
\end{document}